\documentclass[11pt,reqno,english]{amsart}
\usepackage{color}
\usepackage{amsmath,amssymb,amsfonts,amsthm,bm}
\usepackage{mathrsfs}
\usepackage{graphicx}
\usepackage{enumerate}
\usepackage[numbers, sort&compress]{natbib}
\usepackage[shortlabels]{enumitem}
\usepackage{a4wide}
\usepackage[colorlinks=true]{hyperref}
\usepackage{amssymb}
 
\usepackage{amsfonts,amsthm,amsmath,latexsym,bm,graphics}

\def\theequation{\thesection.\@arabic\c@equation}
\makeatother

\renewcommand{\theequation}{\thesection.\arabic{equation}}
\newtheorem{lemma}{Lemma}[section]

\newtheorem{definition}{Definition}

\newtheorem{proposition}{Proposition}[section]
\newtheorem{corollary}{Corollary}[section]
\newtheorem{remark}{Remark}[section]

\newtheorem{theorem}{Theorem}[section]

\numberwithin{equation}{section}

\makeatletter

\@namedef{subjclassname@2020}{%
	\textup{2020} Mathematics Subject Classification}
 
\begin{document}

\title[Incompressible Oldroyd-B Model]{Sharp Decay Characterization for the Incompressible Oldroyd-B Model in Critical $L^p$ Spaces}
%
%
\author{Zhi Chen}
\address{School of  Mathematics and Statistics, Anhui Normal University, Wuhu 241002, China}
\email{zhichenmath@ahnu.edu.cn}
\author{Mingwen Fei}
\address{School of  Mathematics and Statistics, Anhui Normal University, Wuhu 241002, China}
\email{mwfei@ahnu.edu.cn}
\author{Lvqiao Liu}
\address{School of  Mathematics and Statistics, Anhui Normal University, Wuhu 241002, China}
\email{lvqiaoliu@ahnu.edu.cn}
\author{Jiahong Wu}
\address{Department of Mathematics, University of Notre Dame, Notre Dame, IN 46556, USA}
\email{jwu29@nd.edu}




\begin{abstract}
   This paper establishes a sharp characterization of temporal decay rates for the incompressible Oldroyd-B model in a critical $L^p$ framework, covering the physically relevant and mathematically delicate case where both the fluid viscosity and the stress tensor damping are absent. We prove that an $L^2$-type condition on the low-frequencies part of the initial data $(u_0, \tau_0)$ is almost both necessary and sufficient for obtaining optimal upper and lower bounds on the temporal decay of solutions in critical Besov spaces. A key contribution is a new decomposition of the stress tensor into its incompressible and compressible parts, combined with the introduction of an effective tensor to handle the loss of regularity in the high-frequencies velocity field. This is the first result to reveal such precise two-sided asymptotics for the incompressible Oldroyd-B model without viscosity or damping. 
\end{abstract}


\maketitle

\tableofcontents


\section{Introduction}

The Oldroyd-B(OB) model, first introduced by Oldroyd \cite{MR94085} in 
1958, is one of the most classical and widely used constitutive models 
for describing the motion of viscoelastic fluids. It captures the 
essential competition between viscous dissipation and elastic stress 
relaxation that characterizes non-Newtonian fluid behavior. More detailed physical background and derivations can be found in \cite{Constantin-Kliegl,MR2843021,Brereton_1978}. 

\medskip 
Mathematically, 
the Oldroyd-B model presents significant 
challenges compared to the classical Navier--Stokes equations. The 
stress tensor $\tau$ satisfies a transport-diffusion equation that is 
nonlinearly coupled to the velocity field through the term 
$Q(\tau,\nabla u)$, which is bilinear in $\tau$ and $\nabla u$ and 
does not enjoy the same cancellation properties as the nonlinearity 
in the Navier--Stokes equations. This coupling, together with the 
possible absence of viscosity in the velocity equation ($\nu=0$ in \eqref{GOB}) and 
the absence of damping in the stress equation ($a=0$ in \eqref{GOB}), makes the 
global well-posedness and large-time behavior 
substantially more delicate than for purely viscous flows. 

\medskip 
Understanding the large-time behavior of such flows is central to the mathematical 
theory of non-Newtonian fluid mechanics. This paper is devoted to a sharp 
characterization of the decay rates. 
We establish conditions on the initial data that are almost both necessary 
and sufficient for obtaining optimal upper and lower bounds on the 
temporal decay of solutions in a critical $L^p$ framework.

\subsection{Oldroyd-B model}
In this paper, we consider the following Oldroyd-B system
\begin{equation}\label{OB-1}
		\begin{cases}
			\partial_tu+(u\cdot\nabla)u+\nabla P=  \nu_1\operatorname{div}\tau,\\
			\partial_t\tau+(u\cdot\nabla)\tau-\eta\Delta\tau+\mathrm{Q}(\tau,\nabla u)=\nu_2D(u),\\
			\operatorname{div} u=0,\\
			u(x,0)=u_0(x),~~~\tau(x,0)=\tau_0(x),
		\end{cases}
	\end{equation}
where $u(t,x)=(u^1(t,x),u^2(t,x),\cdots,u^{d}(t,x))$ denotes the velocity, $\tau(t,x)=\left(\tau^{i,j}(t,x)\right)_{d\times d}$ is the non-Newtonian part of the stress tensor ($\tau$ is a $d\times d$ symmetric matrix here and $[\operatorname{div} \tau]^i= \partial_j\tau^{i,j}$\footnote{$i,j=1,2,\cdots,d$. Hereafter, we shall use the Einstein summation convention over repeated indices.})
	and $P$ is a scalar pressure of fluid.
	$D(u)$ is the symmetric part of the velocity gradient,
i.e.,
		$D(u)=\frac{1}{2}(\nabla u+(\nabla u)^\top)$
	with $(\nabla u)_{ij}=\partial_ju^i$ and $(\nabla u)^\top_{ij}=\partial_iu^j$.
	The term  $\mathrm{Q}(\tau,\nabla u)$ is a given bilinear form
\begin{equation*}
		\mathrm{Q}(\tau,\nabla u)=\tau\Omega(u)-\Omega(u)\tau+b\big(D(u)\tau+\tau D(u)\big),
	\end{equation*}
	where $b$ is a parameter in $[-1,1],$ $\Omega(u)$ is the skew-symmetric part of $\nabla u,$ i.e.,
$\Omega(u)=\frac{1}{2}(\nabla u-(\nabla u)^\top)$.   The constants $\nu_1,\nu_2>0$ 
are coupling parameters and $\eta>0$ is the stress diffusion coefficient. 
\par

We note that system \eqref{OB-1} corresponds to the case $\nu=a=0$ 
in the general Oldroyd-B model
	\begin{equation}\label{GOB}
	\begin{cases}
		\partial_tu+(u\cdot\nabla)u-\nu\Delta u+\nabla P= \nu_1 \operatorname{div}\tau,\\
		\partial_t\tau+(u\cdot\nabla)\tau-\eta\Delta\tau+ a\tau+\mathrm{Q}(\tau,\nabla u)=\nu_2D(u),\\
		\operatorname{div} u=0,\\
		u(x,0)=u_0(x),~~~\tau(x,0)=\tau_0(x).
	\end{cases}
\end{equation}
in which $\nu\geq 0$ is the fluid viscosity and $a\geq 0$ is a damping 
coefficient. The absence of viscosity in the velocity equation ($\nu=0$) 
and the absence of damping in the stress equation ($a=0$) make 
system~\eqref{OB-1} considerably more delicate than the general 
model~\eqref{GOB}, and understanding its large-time behavior 
is the central goal of this paper. 
\par
 
Let ($u_L, \tau_L$) be the solution to the corresponding linear system of \eqref{OB-1}
\begin{equation}\label{OB}
		\begin{cases}
			\partial_tu+\nabla P=  \nu_1\operatorname{div}\tau,\\
			\partial_t\tau-\eta\Delta\tau=\nu_2D(u),\\
			\operatorname{div} u=0,\\
			u(x,0)=u_0,~~~\tau(x,0)=\tau_0.
		\end{cases}
	\end{equation}
Moreover, let $(u, \tau)$ be the solution to  \eqref{OB-1} and   $(\tilde{u}, \tilde{\tau})=(u-u_L, \tau-\tau_L)$, then $(\tilde{u}, \tilde{\tau})$ satisfies the error system   
\begin{equation}\label{OBC-4}
		\begin{cases}
			\partial_t\tilde{u}+\nabla \tilde{P}=\nu_1\operatorname{div}\tilde{\tau}-(u\cdot\nabla)u,\\
			\partial_t\tilde{\tau}-\eta\Delta\tilde{\tau}=\nu_2 D(\tilde{u})-(u\cdot\nabla)\tau-\mathrm{Q}(\tau,\nabla u),\\
			\operatorname{div} \tilde{u}=0,\\
			\tilde{u}(x,0)=0,~~~\tilde{\tau}(x,0)=0.
		\end{cases}
	\end{equation}
	
	\subsection{Function spaces and main results}
    	\begin{definition}\label{def1}
		For $s\in\mathbb{R}$ and $1\leq p,r\leq \infty$, the homogeneous Besov space $\dot{B}^s_{p,r}=\dot{B}^s_{p,r}(\mathbb{R}^d)$ is the set of tempered distributions $f\in\mathcal{S}_h'(\mathbb{R}^d)$ with the norm
		\begin{align*}
			\|f\|_{\dot{B}^{s}_{p,r}}=\left\|\Big(2^{sk}
			\|\dot{\Delta}_k f\|_{L^p}\Big)\right\|_{\ell^r}.
		\end{align*}
		In the case where $f$ depends also on the time variable, we shall  consider the Chemin-Lerner space $\tilde{L}^\rho_T(\dot{B}^s_{p,r})=\tilde{L}^\rho(0,T;\dot{B}^s_{p,r})$  $ (1 \leq \rho \leq  \infty)$with the norm
		\begin{align*}
			\|f\|_{\tilde{L}^\rho_T(\dot{B}^s_{p,r})}=
			\left\|\Big(2^{sk}
			\|\dot{\Delta}_k f\|_{L^\rho_T(L^p)}\Big)\right\|_{\ell^r} .
		\end{align*}
	\end{definition}
By Minkowski's inequality, it is easy to see that
\begin{align*}
	\|f\|_{\tilde{L}_T^\rho(\dot{B}^{s}_{p,r})}\lesssim 	\|f\|_{L_T^\rho(\dot{B}^{s}_{p,r})}\ \text{ if } \ \rho\leq r, \qquad
		\|f\|_{\tilde{L}_T^\rho(\dot{B}^{s}_{p,r})}\gtrsim	\|f\|_{L_T^\rho(\dot{B}^{s}_{p,r})}\ \text{ if } \ \rho\geq r.
\end{align*}
Restricting the norms in Definition \ref{def1} to low or high frequencies parts of distributions will be crucial in current work.  More precisely, for some fixed integer $k_0$,   set
\begin{align*}
	f^\ell    =\dot{S}_{k_0} f \quad \text{ and }\  f^h=f-\dot S_{k_0} f
\end{align*}
and\footnote{Here,  we used a small overlap between low and high frequencies norms  due to our  technical reasons. For $1\leq p\leq \infty$, $s,s'\in\mathbb{R}$ with  $s<s'$,  and  $f\in \dot{B}^{s}_{p,1}\cap\dot{B}^{s'}_{p,1}$, it is easy to verity that
    \begin{align*}
        \|f^\ell   \|_{\dot{B}^{s'}_{p,1}}\lesssim \|f^\ell   \|_{\dot{B}^{s}_{p,1}} \quad \text{ and }\quad 	\|f^h\|_{\dot{B}^{s}_{p,1}}\lesssim \|f^h\|_{\dot{B}^{s'}_{p,1}}.
\end{align*}
}  
\begin{align*}
	\|f\|^\ell_{\dot{B}^s_{p,1}}=\sum_{k\leq k_0}2^{ks}\|\dot{\Delta}_kf\|_{L^p},~~~&\qquad \|f\|^h_{\dot{B}^s_{p,1}}=\sum_{k\geq k_0-1}2^{ks}\|\dot{\Delta}_kf\|_{L^p},
\\
	\|f\|^\ell_{\tilde{L}^\infty_T(\dot{B}^s_{p,1})}=\sum_{k\leq k_0}2^{ks}\|\dot{\Delta}_kf\|_{L^\infty_T(L^p)},~~&\qquad 
	\|f\|^h_{\tilde{L}^\infty_T(\dot{B}^s_{p,1})}=\sum_{k\geq k_0-1}2^{ks}\|\dot{\Delta}_kf\|_{L^\infty_T(L^p)},
\end{align*}
where $k_0$ is called threshold between low frequencies and high frequencies which will been chosen in later. More details about Besov spaces can be found in \cite{BCD-2011}.

	
\par
We now introduce  the following  energy functional
\begin{equation}\label{EE}
    E(t)=E^\ell(t)+E^h(t),
\end{equation}
		where $E^\ell(t)$ and $E^h(t)$ are the low and high frequencies pieces, 
        more precisely,
		\begin{align*}
			E^\ell(t)=&\|(u,\tau)\|^\ell_{\tilde{L}^\infty_t({\dot{B}^{\frac{d}{2}-1}_{2,1}})}+\int_{0}^{t}\!\! \|(u,\tau)(t')\|^\ell_{\dot{B}^{\frac{d}{2}+1}_{2,1}} \mathrm{d} t' ,
		\end{align*}
and
		\begin{align*}
			E^h(t)= &\|u\|^h_{\tilde{L}_t^\infty(\dot{B}^{\frac{d}{p}+1}_{p,1})}\!+\!
            \|\tau\|^h_{\tilde{L}_t^\infty(\dot{B}^{\frac{d}{p}}_{p,1})}
			+\!\!\int_0^t \!\! \big(\|u(t')\|^h_{\!\dot{B}^{\frac{d}{p}+1}_{p,1}}\!+
            \!\|\tau(t')\|^h_{\!\dot{B}^{\frac{d}{p}+2}_{p,1}}\big)dt' .
		\end{align*}
 
The necessity of splitting into low and high frequencies stems from a 
fundamental asymmetry in the dissipative structure of 
system~\eqref{OB-1}. Since the velocity equation contains no 
viscosity ($\nu=0$), the dissipation available to $u$ comes entirely 
from the coupling with $\tau$ through the term $\nu_1\,\mathrm{div}\,\tau$. 
To understand the frequencies-dependent behavior of this coupling, 
we consider the linear system~\eqref{OB}. Taking the Fourier 
transform in $x$, one finds that the pair $(\hat{u},\hat{\tau})$ 
satisfies a system whose characteristic roots are
\[
\lambda_\pm = \frac{-\eta|\xi|^2 \pm \sqrt{\eta^2|\xi|^4 
		- 2\nu_1\nu_2|\xi|^2}}{2}.
\]
At low frequencies $|\xi|\ll 1$, both roots satisfy 
$\mathrm{Re}(\lambda_\pm) = -\frac{\eta}{2}|\xi|^2$, so the solution 
decays like the heat kernel at rate $e^{-c|\xi|^2 t}$. In this regime 
the coupling between $u$ and $\tau$ provides effective dissipation of 
order $|\xi|^2$ for both unknowns, and $L^2$-based Besov spaces 
$\dot{B}^s_{2,r}$ are natural and sufficient to close the estimates. At high frequencies $|\xi|\gg 1$, the two roots satisfy $\lambda_+\approx 0$ 
and $\lambda_-\approx -\eta|\xi|^2$ as $|\xi|\to\infty$. The root 
$\lambda_+$ approaching zero means that the velocity $u$ at high 
frequencies receives essentially no dissipation from the 
coupling. It behaves more like a solution of a transport equation than 
a parabolic one. This causes a genuine loss of one derivative for $u$ 
relative to $\tau$ at high frequencies. Therefore, to handle the high frequencies part, we need to impose higher regularity assumptions on $u_0$.

\medskip
Throughout this paper, we focus on the case of $d\geq2$ and $p$ satisfies
        \begin{align}\label{ppp}
			2\leq p\leq\min\Big\{4,\frac{2d}{d-2}\Big\}\mbox{ and, additionally, }p\neq4\mbox{ if }d=2.
		\end{align}
The upper bound $p\leq\min\{4,\frac{2d}{d-2}\}$ ensures that the 
critical Besov product estimates are 
applicable. The exclusion $p\neq 4$ when $d=2$ is a technical 
restriction arising from a borderline embedding in those same estimates, 
where the product law degenerates at this endpoint.

\medskip
This paper aims to establish  an almost necessary and sufficient condition for the sharp time-decay rates of solutions in the critical $L^p$ space, which have been widely used in the literatures (see, e.g., \cite{MR2675485, Xu-2019, Xin-Xu-2021, CD2010}).  Our main results are stated as follows.
\begin{theorem}\label{Thm}
 Assume that initial data $(u_0,\tau_0)$ satisfies
\begin{equation}\label{E-0}	 	
			E_0= \|(u_0,\tau_0)\|^\ell_{\dot{B}^{\frac{d}{2}-1}_{2,1}}+\|u_0\|^h_{{\dot{B}^{\frac{d}{p}+1}_{p,1}}}+\|\tau_0\|^h_{{\dot{B}^{\frac{d}{p}}_{p,1}}}\leq c_0,
		\end{equation}
    for some sufficiently small $c_0>0$ and $p$ satisfies \eqref{ppp}. 
 Let 
 $\sigma_0=\frac{d}{2}-\frac{2d}{p}$. 
 \begin{itemize}
     \item  Then
  for any given time $t_0> 0$, the solution $(u,\tau)$ to the Cauchy problem \eqref{OB-1} satisfy  the following estimates
\begin{align}\label{Theo1}
\|(u,\tau)^\ell(t)\|_{\dot{B}^{\sigma_1}_{2,\infty}}\leq C,~~t>0,\   C>0,
\end{align}
\begin{align}\label{Theo2}
\|(u,\tau)(t)\|_{\dot{\mathbb{B}}^\sigma_{2,p}}\lesssim \langle t \rangle^{-\frac{1}{2}(\sigma-\sigma_1)},~~t>t_0, \ \sigma_1<\sigma\leq \frac{d}{2},
\end{align}
if and only if
\begin{equation}\label{ini}
    (u_0,\tau_0)^\ell\in \dot{B}^{\sigma_1}_{2,\infty}, \ \sigma_0\leq \sigma_1<\frac{d}{2}-1.
\end{equation}
\item Moreover, there exists a time $t_0>0$, if the solutions  $(u,\tau)$ satisfy the following decay estimates 
$$
\langle t \rangle^{-\frac{1}{2}(\sigma-\sigma_1)} \lesssim\|(u,\tau)(t)\|_{\dot{\mathbb{B}}^\sigma_{2,p}}\lesssim \langle t \rangle^{-\frac{1}{2}(\sigma-\sigma_1)},~~t>t_0,
$$
then, we have \begin{equation}\label{ini-2}
    (u_0,\tau_0)^\ell\in \dot{\bf{B}}^{\sigma_1}_{2,\infty},\ \sigma_0<\sigma_1<\frac{d}{2}-1. 
\end{equation}
On the other hand, if $(u_0,{\mathbb{P}\Lambda^{-1}\mathrm{div}\tau}_0)^\ell\in \dot{\bf{B}}^{\sigma_1}_{2,\infty},~~\tau_0^\ell\in \dot{{B}}^{\sigma_1}_{2,\infty}$, we have
\begin{equation}\label{6127}
    \langle t \rangle^{-\frac{1}{2}(\sigma-\sigma_1)} \lesssim\|(u,\tau)(t)\|_{\dot{\mathbb{B}}^\sigma_{2,p}}\lesssim \langle t \rangle^{-\frac{1}{2}(\sigma-\sigma_1)},~~t>t_0,
\end{equation}

 \end{itemize}
where  $\mathbb{P}=I+\nabla (-\Delta)^{-1}\operatorname{div}$ is the Leray projection operator and 
$\dot{\textbf{B}}^{\sigma_1}_{2,\infty}=\{ f\in\dot{B}^{\sigma_1}_{2,\infty}:    \exists    M_0, M>0,       \exists ( j_k)_{k \in \mathbb{N}} \subset \mathbb{Z}$  such that   
     $ j_k\rightarrow-\infty ~~\text{as}~~  k \rightarrow+\infty,  ~|j_k-j_{k+1}|\leq M$~ {and}~$2^{\sigma_1j_k}\|\dot{\Delta}_{j_k}f\|_{L^2}\geq M_0  ~\text{for~ any}~~      k \in \mathbb{N}
 \}$, 
and
$$
\|(u,\tau)\|_{\dot{\mathbb{B}}^\sigma_{2,p}}=\|(u,\tau)^\ell\|_{\dot{B}^{\sigma}_{2,1}}+\|(u,\tau)\|^h_{\dot{B}^{\frac{d}{p}+1}_{p,1}}.
$$
\end{theorem}
In particular, in the case of $p=2$, we have
\begin{corollary}
There exists a time $t_1>0$ such that for $\sigma_1<\sigma \leq \frac{d}{2}$,  $\sigma_0< \sigma_1<\frac{d}{2}-1$, if the solution
$(u,\tau)$ to the Cauchy problem \eqref{OB-1} satisfies 
$$
\langle t \rangle^{-\frac{1}{2}(\sigma-\sigma_1)} \lesssim\|\Lambda^{\sigma}(u,\tau)(t)\|_{\dot{B}^0_{2,1}}\lesssim \langle t \rangle^{-\frac{1}{2}(\sigma-\sigma_1)},~~t>t_1,
$$
then we have $(u_0,\tau_0)^\ell\in \dot{\bf{B}}^{\sigma_1}_{2,\infty}$. 
\par
On the other hand, if $(u_0,{\mathbb{P}\Lambda^{-1}\mathrm{div}\tau}_0)^\ell\in \dot{\bf{B}}^{\sigma_1}_{2,\infty},~~\tau_0^\ell\in \dot{{B}}^{\sigma_1}_{2,\infty}$, we have
$$
\langle t \rangle^{-\frac{1}{2}(\sigma-\sigma_1)} \lesssim\|\Lambda^{\sigma}(u,\tau)(t)\|_{\dot{B}^0_{2,1}}\lesssim \langle t \rangle^{-\frac{1}{2}(\sigma-\sigma_1)},~~t>t_1.
$$
\end{corollary}

\begin{remark}
    From previous works \cite{Chen}, the condition \eqref{E-0} guarantees the global existence and uniqueness of solution to the Cauchy problem \eqref{OB-1} and the solution $(u, \tau)  $ satisfies
    \begin{equation}\label{E-t}
        E(t) \lesssim E_0.
    \end{equation}
\end{remark}

 \begin{remark}
 The space $\dot{\textbf{B}}^{\sigma_1}_{2,\infty} $ has been introduced in \cite{MR3493117} to deal with the decay  characterization to the Navier-Stokes equations.  
 \end{remark}
\subsection{Related literatures}
 The well-posedness of system \eqref{GOB} has been  widely studied in the past decades. 
    For the parameters $\nu,\eta>0$ and $a>0$, Constantin and Kliegl \cite{Constantin-Kliegl} established the global well-posedness of strong solutions in two spatial dimensions.
     In the corotational case  $b=0$,  the existence of global weak solutions was demonstrated by Lions and Masmoudi \cite{MR1763488} and the case $b\neq 0$ is still open up to now.
For the case $\nu> 0$, $\eta=0$ and $a>0$, provided the coupling parameter and initial data are sufficiently small,
    Guillop\'{e} and Saut \cite{Guillope-1,Guillope-2} proved the existence of  global solution, 
     Chemin and Masmoudi \cite{Chemin-Masmoudi} then obtained  global solutions within the framework of critical Besov spaces. 
  Building upon techniques developed for the compressible Navier–Stokes equations  \cite{CD2010,MR2675485}, Zi, Fang and Zhang \cite{Fang-Zhang-Zi} extended these results to the critical $L^p$ framework and removed the smallness assumption for the coupling
parameter. Further developments in this direction can be found in \cite{MR4920797,MR2388753,Chen,Fang-Hieber-Zi,MR4152236,MR3473592,MR3023375,Zi}.

For the case \(\nu = 0\) and \(\eta > 0\), system \eqref{GOB} reduces to a coupling between the forced Euler equations for the velocity field and a kinetic description of the particle dynamics.  
In the presence of a damping mechanism in the equation for \(\tau\) (i.e., \(a > 0\)), Elgindi and Rousset \cite{Elgindi-Rousset} established a   global well-posedness result in the Sobolev space for the two-dimensional Oldroyd-B model with small initial data. The three-dimensional case was subsequently addressed by Elgindi and Liu   \cite{MR3349425} in the Sobolev setting, and later extended by Huang, Liu, and Zi \cite{MR4693221} within the framework of Besov spaces. It is worth emphasizing that the damping term in the \(\tau\)-equation plays an essential role in the analyses presented in \cite{MR3349425,MR4693221,Elgindi-Rousset}. In the absence of damping (i.e., \(a = 0\)),    the linearization of model \eqref{GOB} exhibits a damped wave equation structure. Exploiting this property, Zhu \cite{Zhu} proved the global existence of  strong solution   with $\eta=a=0$.  More recently, Constantin, Wu, Zhao and Zhu \cite{Constantin-3} proved that there exists a unique global-in-time solution in Sobolev space with the fractional Laplacian. 
 A key insight of \cite{Constantin-3} is that the non-Newtonian stress tensor can exert a regularizing effect on the dynamics of viscoelastic fluids, even in the absence of explicit damping.  

The large-time behavior of solutions plays a crucial role in understanding many physical phenomena, and numerous research has been devoted to this topic. The Fourier-splitting approach has been  proven particularly effective for studying the large-time asymptotics of fully dissipative systems (see, e.g., \cite{MR837929, MR1396285, MR1103459}). However, this method is not directly applicable to certain partially dissipative systems when one seeks to establish optimal decay rates. In this context, we mention the work \cite{MR4792291}, which investigates decay characterizations for the compressible Navier–Stokes equations. Moreover, in the presence of a damping mechanism (i.e., $a>0$), Huang, Wang, Wen and Zi \cite{MR4335131} studied the large-time behavior of strong solutions to model \eqref{GOB} with vanishing viscosity in Sobolev spaces. In particular, under the additional assumption $(u_0, \tau_0) \in L^1$, they established the time decay rate of solutions. Subsequently, Huang, Liu and Zi \cite{MR4693221} obtained faster decay rates without imposing further smallness conditions. For related results, we also refer the reader to \cite{MR3909068, MR4374896, MR3927165} and the references therein.

\par
 In the case of \eqref{OB-1},  i.e., $\nu=a=0$ and $\eta>0$ , Wang, Wu, Xu and Zhong \cite{MR4348795} made an extra assumption on the initial data $(u_0,\tau_0)
  $, i.e., $(u_0,\tau_0)\in L^1(\mathbb{R}^d) $, and  with a sufficiently small parameter  $\epsilon>0$  such that
  \begin{align*} 
  \|(u_0,\tau_0)\|_{L^1\cap H^r} \leq \epsilon, \qquad r>1+\frac{d}{2},
 \end{align*}
then the optimal decay rates for the $L^2$-norm of
$u$ and $\nabla u$, the $L^\infty$-norm of $u$ and $\nabla u$ as well as the $L^2$-norm of $\mathbb{P}\nabla \cdot\tau$ are obtained. 
Later, the first author and collaborators removed the smallness of $L^1$-norm in \cite{MR5009527}, and obtained a similar optimal time-decay rates by adding some additional $L^2$ type conditions on the low frequencies of the initial data. {\bf A natural question is whether the low-frequencies assumption \eqref{ini} is sharp or not for the large-time behavior of strong solutions to \eqref{OB-1} in critical $L^p$ space?}

\subsection{Strategy of the proof}
 
In this part, we outline the main steps and innovative techniques for establishing the  sharp large-time decay estimates for solutions to the incompressible Oldroyd-B system \eqref{OB-1} in critical spaces. The central challenge lies in overcoming the loss of regularity in the high frequencies of the velocity field $u$ and handling the non-standard structure of the stress tensor $\tau$.  

 \begin{itemize}
     \item[(1)] { \bf The role of $L^2$-type initial data and sharp decay framework }
     \newline
     Under the $L^2$-type assumption on the initial data specified in \eqref{ini}, we derive sharp upper and lower bounds for the time-decay rates of solutions in critical $L^p$ spaces. Moreover, this condition is shown to be not only sufficient but also necessary for attaining the optimal upper decay estimates. 
     We also point out that, since
     \[
     L^1(\mathbb{R}^d)\hookrightarrow \dot{B}^{0}_{1,\infty}\hookrightarrow \dot{B}^{-\frac{d}{2}}_{2,\infty},
     \]
     the class of initial data considered here differs from the $L^1$ initial data employed in \cite{MR4348795}.
     \vspace{5pt}
     \item[(2)] { \bf Decomposition of the stress tensor and linear estimates }
\newline
     Similar to the compressible Navier-Stokes equations, the high-frequencies part of the velocity $u$ suffers from a loss of regularity.  To circumvent this, we adopt the effective tensor formalism introduced by \cite{MR5009527}.     
     For the low-frequencies analysis of the tensor $\tau $, we employ a  decomposition into its “incompressible part" and “compressible part" to obtain precise estimates for the linear OB system \eqref{OB}. 
     In order to exhibit bounds of the high frequencies part of the solution, we introduce the following {\it effective tensor}  in the OB system
\begin{align*} 
	W=\nabla(-\Delta)^{-1}(\frac{1}{2}u+\operatorname{div}\tau),
\end{align*}
which was motivated by the work \cite{Haspot} to deal with the compressible Navier-stokes equations. 
  \vspace{5pt}
     \item[(3)] { \bf Pointwise lower bound estimates to linear OB system \eqref{OB}}
\newline
     In this work, we establish not only upper bounds but also lower bounds, the derivation of which is technically nontrivial. To address this challenge, we work with the pair $(u, \Lambda^{-1}\operatorname{div}  \tau)$ instead of $(u, \tau)$, where the second component represents a potential function associated with the incompressible part of $\tau$. This formulation reveals a crucial cancellation property (see Lemma \ref{ls10}), which plays a pivotal role in establishing the desired lower bounds and confirming the sharpness of the decay rates.
       \vspace{5pt}
     \item[(4)] { \bf Handling the nonlinear term $Q( \tau, \nabla u)$ }
\newline
To establish the enhanced decay rate for the solution $(\tilde{u},\tilde{\tau})$ of the error system \eqref{OBC-4}, it is necessary to estimate the nonlinear term $Q( \tau, \nabla u)$ in the low-frequencies regime. This task is particularly delicate due to the absence of an incompressibility condition on $\tau$. A natural estimate, such as
 \begin{equation*}
       \|{\tau}^\ell \cdot \nabla\tilde{u}^h\|^\ell_{\dot{B}^{\sigma_1-\sigma'}_{2,\infty}}\lesssim \|\nabla\tilde{ u}^h\|_{\dot{B}^{\frac{d}{p}-1}_{p,1}}\|\tau^\ell\|_{\dot{B}^{\frac{d}{2}-1}}
    \end{equation*}
is not enough to get our desired estimate  because $\|\nabla\tilde{u}^h\|_{\dot{B}^{\frac{d}{p}-1}_{p,1}}$ lacks sufficient temporal decay.
To overcome this difficulty, we exploit the following refined decomposition
\begin{equation*}
       \tau^\ell\cdot \nabla \tilde{u}^h=\tau_{ik}^\ell\partial_k\tilde{u}_j^h=\partial_k(\tau_{ik}^\ell\tilde{u}_j^h )-\partial_k\tau_{ik}^\ell\tilde{u}_j^h.
     \end{equation*}
The two resulting terms are then estimated separately by choosing carefully  regularity indices $\sigma'$. For the first term, we set $\sigma'=1$, which allows us to gain one derivative. For the second term, we exploit the fact that $\|\nabla\tilde{\tau}^\ell\|_{\dot{B}^{\frac{d}{2}-1}_{2,1}}$ exhibits faster temporal decay when $\sigma'$ is taken sufficiently small. 
 \end{itemize}
 
\subsection{Arrangement and notation of the paper}  {

The rest of this paper is organized as follows. In Section~\ref{sec:linear}, 
we study the linear Oldroyd-B system~\eqref{OB} and establish 
pointwise upper and lower bounds for the solution $(u_L,\tau_L)$, 
leading to the decay characterization of Proposition~\ref{prop1}. 
These linear estimates, particularly the lower bounds of 
Lemma~\ref{ls10}, are the foundation for the necessity argument 
in Section~\ref{sec:necessary}. Section~\ref{sec:sufficient} is the 
technical core of the paper. Section~\ref{subsec:sufficient}  is devoted to proving the sufficient 
condition in Theorem~\ref{Thm}(i) by establishing the 
time-weighted functional inequality $\tilde{\mathcal{D}}_p(t)\lesssim 1$ 
of Proposition~\ref{cob2}, which requires careful low and 
high frequencies analysis of the nonlinear terms including the treatment 
of $Q(\tau,\nabla u)$ described in item~(4) above. 
Section~\ref{sec:necessary} establishes the necessary condition in 
Theorem~\ref{Thm}(ii)--(iii), completing the proof of 
Theorem~\ref{Thm}. An appendix collects the Besov space product 
estimates used throughout the paper.

We denote by $C$ the generic constant which may change line by line. $A\lesssim B$ means that there is a generic constant C such that $A \leq  C B$. The notation $A \approx B $ implies that $A\lesssim B$ and $B\lesssim A$.}

\vskip .2in 
\section{Decay characterization for the linear Oldroyd-B system}
\label{sec:linear}

This section is devoted to the decay analysis of the linear Oldroyd-B system 
\eqref{OB}, which forms the analytical foundation for the nonlinear 
results established in Sections \ref{sec:sufficient}. 
We proceed in four steps. 

  \vspace{5pt}
First, in Subsection \ref{sec:pointwise}, we establish pointwise upper and lower 
bounds for the solution $(u_L, \tau_L)$ to the linear system. Specifically, 
Lemma \ref{ls10} provides the lower bound which shows that the solution cannot decay 
faster than $e^{-\eta|\xi|^2 t}$ at low frequencies, while Lemma \ref{ls1} 
establishes the corresponding upper bounds at both low and high frequencies. 
These two lemmas together yield the decay characterization of Proposition 
\ref{wp1}, which identifies the low-frequencies regularity of the initial data as 
the precise condition governing the decay rate.

  \vspace{5pt}
Second, in Subsection \ref{sec:wiegner}, we apply a Wiegner-type argument 
adapted to the incompressible Oldroyd-B structure to derive the global energy 
bounds recorded in Proposition \ref{prop1}. 

  \vspace{5pt}
Third, in Subsection \ref{sec:lowfreq}, we carry out the low-frequencies analysis 
of the linear system. A key difficulty here is the absence of dissipation in 
the velocity equation, which prevents a direct energy estimate. To overcome this, 
we decompose the stress tensor $\tau$ into its incompressible part 
$v = \Lambda^{-1}\operatorname{div}\tau$ and its compressible part 
$w = \Lambda^{-1}\operatorname{curl}\tau$ as in \eqref{tau-v-w}, and work 
instead with the triplet $(u, v, w)$. This decomposition reveals a hidden 
dissipative structure, captured in Lemma \ref{lem3.1}, which yields the key 
differential inequality \eqref{L-k-estimate} controlling the low-frequencies 
evolution. 

  \vspace{5pt}
Finally, in Subsection \ref{sec:highfreq}, we handle the high-frequencies regime. 
Here the main challenge is a genuine loss of one derivative for the velocity 
field $u$ relative to the stress tensor $\tau$, which manifests through the 
characteristic root $\lambda_+ \approx 0$ as $|\xi| \to \infty$. To address 
this, we introduce the effective tensor 
\begin{equation*}
	W = \nabla(-\Delta)^{-1}\!\left(\tfrac{1}{2}u + 
	\operatorname{div}\tau\right),
\end{equation*}
as in \eqref{effective-tesor} and establish the high-frequencies 
bound of Lemma \ref{lem3.2}. Combining all these estimates, we conclude the 
section with the proof of Proposition \ref{prop1}, which gives a complete decay 
characterization in terms of the low-frequencies regularity of the initial data.

\vskip .1in 
\par
\subsection{Frequencies pointwise estimates} 
\label{sec:pointwise}

In the following lemma we give the lower bounds of the low frequencies.
\begin{lemma}\label{ls10}
Let $(u,\tau)$ be the solution to the Cauchy problem  \eqref{OB}, then there exists positive constants 
$C $ and $c$ such that the following estimate holds
\begin{align*}
|\hat{u}(t,\xi)|+|\widehat{\tau}(t,\xi)|\geq Ce^{-\eta|\xi|^2t}(|\hat{u_0}|+|\widehat{\mathbb{P}\Lambda^{-1} div\tau_0}|),~~~|\xi|\leqslant c.
\end{align*}

\end{lemma}
\begin{proof} 
We rewrite the linear system \eqref{OB} as
\begin{equation}\label{OB-222}
	\begin{cases}
		\partial_tu-\nu_1\Lambda\mathbb{P} v=0,\\
		\partial_t\mathbb{P}v-\eta\Delta \mathbb{P}v+\frac{\nu_2}{2}\Lambda u=0,\\
        u(0,x)=u_0(x),~~v(0,x)=\Lambda^{-1}\operatorname{div}\tau_0(x),
	\end{cases}
\end{equation}
where $v=\Lambda^{-1}\operatorname{div}\tau$.
\par

By differentiating in $t$ to the first equation of \eqref{OB-222} and substituting the resultant equation to the second equation, we find that $u$ satisfies exactly the damped wave equation,
\begin{equation*}
  \partial_{tt}u+\eta\Lambda^2\partial_tu-\frac{\nu_1\nu_2}{2}\Delta u=0.
\end{equation*}
Taking the Fourier transform to above equation with respect to the space variable yields
\begin{equation}\label{FF}
  \partial_{tt}\hat u+\eta|\xi|^2\partial_t \hat u+\frac{\nu_1\nu_2}{2}|\xi|^2\hat u=0.
\end{equation}
The values $\lambda_{\pm}$ are the two roots of the corresponding characteristic equation,
\begin{equation*}
  \lambda^2+\eta|\xi|^2\lambda+\frac{\nu_1\nu_2}{2}|\xi|^2=0,
\end{equation*}
with
\begin{equation}\label{lambda}
     \lambda_{\pm}=\frac{-\eta|\xi|^2\pm\sqrt{\eta^2|\xi|^4-2\nu_1\nu_2|\xi|^2}}{2},
\end{equation}
and
\begin{equation*}
  \lambda_++\lambda_-=-\eta|\xi|^2,\qquad \lambda_+\lambda_-=\frac{\nu_1\nu_2}{2}|\xi|^2.
\end{equation*}
The solution of \eqref{FF} take the form
\begin{equation*}
  \hat{u}(t,\xi) =\frac{\nu_1\widehat{\mathbb{P}\mathrm{div}\tau}_0-\lambda_-\hat{u}_0}{\lambda_+-\lambda_-}e^{\lambda_+t}+\frac{\lambda_+\hat{u}_0-\nu_1\widehat{\mathbb{P}\mathrm{div}\tau}_0}{\lambda_+-\lambda_-}e^{\lambda_-t}.
\end{equation*}
Further, the solution of \eqref{FF} is given by
\begin{equation}\label{FU}
  \hat{u}(t,\xi)=\mathcal{A}(t,\xi)\hat{u}_0+\nu_1\mathcal{B}(t,\xi)\widehat{\mathbb{P}\Lambda^{-1}\mathrm{div}\tau}_0,
\end{equation}
with
\begin{equation*}
  \mathcal{A}(t,\xi)=\frac{\lambda_+e^{\lambda_-t}-\lambda_-e^{\lambda_+t}}{\lambda_+-\lambda_-},~~~
  \mathcal{B}(t,\xi)=\frac{e^{\lambda_+t}-e^{\lambda_-t}}{\lambda_+-\lambda_-}|\xi|.
\end{equation*}

Next, taking the Fourier transform to the space variable of the second equation in system \eqref{OB-222}, we obtain
\begin{equation*}
  \hat{Z}_t+\eta|\xi|^2\hat{Z}=-\frac{\nu_2}{2}|\xi|\hat{u},
\end{equation*}
where $Z=\mathbb{P}\Lambda^{-1}\operatorname{div}\tau$.

By the Duhamel's principle, one has
\begin{equation}\label{ZZ}
  \hat{Z}(t,\xi)=e^{-\eta|\xi|^2t}\hat{Z}_0-\frac{\nu_2}{2}e^{-\eta|\xi|^2t}\int^t_0e^{\eta|\xi|^2t'}|\xi|\hat{u}dt'.
\end{equation}
Using the fact in \eqref{FU} and direct calculation yields
\begin{equation}\label{UH}
\aligned
  &\int^t_0 e^{\eta|\xi|^2t'}(|\xi|\hat{u})dt' \\
 &=\frac{|\xi|\hat{u}_0}{\lambda_+-\lambda_-} \int^t_0\lambda_+e^{(\lambda_-+\eta|\xi|^2)t'}-\lambda_-e^{(\lambda_++\eta|\xi|^2)t'} dt'\\
 & \quad+\frac{\nu_1|\xi|\widehat{\mathbb{P}\mathrm{div}\tau}_0}{\lambda_+-\lambda_-}\int^t_0
 e^{(\lambda_++\eta|\xi|^2)t'}-e^{(\lambda_-+\eta|\xi|^2)t'} dt'\\
 &=\frac{|\xi|\hat{u}_0}{\lambda_+-\lambda_-}\bigg(e^{-\lambda_-t}-e^{-\lambda_+t}\bigg)+\frac{\nu_1|\xi|\widehat{\mathbb{P}\mathrm{div}\tau}_0}{\lambda_+-\lambda_-}
 \bigg(\frac{1}{\lambda_+}e^{-\lambda_+t} -\frac{1}{\lambda_-}e^{-\lambda_-t} +\frac{1}{\lambda_-}-\frac{1}{\lambda_+}\bigg).
\endaligned
\end{equation}
Recall the expression of $\lambda_{\pm}$ in \eqref{lambda}, we can calculate
\begin{equation}\label{u0}
\aligned
e^{-\eta|\xi|^2t} \frac{|\xi|\hat{u}_0}{\lambda_+-\lambda_-}\bigg(e^{-\lambda_-t}-e^{-\lambda_+t}\bigg) 
=&\frac{e^{\lambda_+t}-e^{\lambda_-t}}{\lambda_+-\lambda_-}(|\xi|\hat{u}_0),
\endaligned
\end{equation}
and
\begin{equation}\label{t0}
\aligned
&e^{-\eta|\xi|^2t}\frac{\nu_1|\xi|\widehat{\mathbb{P}\mathrm{div}\tau}_0}{\lambda_+-\lambda_-}
 \bigg(\frac{1}{\lambda_+}e^{-\lambda_+t}-\frac{1}{\lambda_-}e^{-\lambda_-t}+\frac{1}{\lambda_-}-\frac{1}{\lambda_+}\bigg)\\
&=(\frac{\lambda_-e^{\lambda_-t}-\lambda_+e^{\lambda_+t}}{\lambda_+-\lambda_-}+e^{-\eta|\xi|^2t})\frac{(\nu_1|\xi|\widehat{\mathbb{P}\mathrm{div}\tau}_0)}{\frac{\nu_1\nu_2}{2}|\xi|^2}.
\endaligned
\end{equation}
Note the fact
\begin{equation}\label{ZZ0}
\nu_1\frac{(|\xi|\widehat{\mathbb{P}\mathrm{div}\tau}_0)}{\frac{\nu_1\nu_2}{2}|\xi|^2}=\nu_1\frac{(|\xi||\xi|\widehat{Z}_0)}{|\xi|^2}\frac{2}{\nu_1\nu_2}=\widehat{Z}_0\frac{2}{\nu_2}.
\end{equation}
Combining \eqref{UH}-\eqref{ZZ0} with \eqref{ZZ} to conclude
\begin{equation*}
\aligned
  \hat{Z}=-\frac{\nu_2}{2}\mathcal{B}(t,\xi)\widehat{u_0}
  -\mathcal{C}(t,\xi)\hat{Z}_0,
\endaligned
\end{equation*}
with
\begin{equation*}
  \mathcal{C}(t,\xi)=\frac{\lambda_-e^{\lambda_-t}-\lambda_+e^{\lambda_+t}}{\lambda_+-\lambda_-}.
\end{equation*}

We now derive estimates for the low-frequencies component.  Let $b=\sqrt{2\nu_1\nu_2|\xi|^2-\eta^2|\xi|^4}$ and
 direct calculation gives
\begin{align*}
\begin{cases}
\qquad \frac{e^{\lambda_+t}-e^{\lambda_-t}}{\lambda_+-\lambda_-}=e^{-\frac{\eta}{2}|\xi|^2t}\frac{\sin(bt)}{b},\\
\frac{\lambda_+e^{\lambda_-t}-\lambda_-e^{\lambda_+t}}{\lambda_+-\lambda_-}=e^{-\frac{\eta}{2}|\xi|^2t}
\Bigg(\cos(bt)+\frac{\eta}{2}\frac{\sin(bt)}{b}|\xi|^2\Bigg),\\
\frac{\lambda_+e^{\lambda_+t}-\lambda_-e^{\lambda_-t}}{\lambda_+-\lambda_-}=e^{-\frac{\eta}{2}|\xi|^2t}
\Bigg(\cos(bt)-\frac{\eta}{2}\frac{\sin(bt)}{b}|\xi|^2\Bigg).
\end{cases}
\end{align*}

Therefore, we can rewrite the expression of $\hat{u}$ and $\hat{Z}$ by
\begin{align*}
\hat{u}(t,\xi)=e^{-\frac{\eta}{2}|\xi|^2t}\hat{u}^*(t,\xi),
\end{align*}
and
\begin{align*}
\hat{Z}(t,\xi)=e^{-\frac{\eta}{2}|\xi|^2t}\hat{Z}^*(t,\xi),
\end{align*}
with
\begin{align}
\hat{u}^*(t,\xi)=\left(\cos(bt)+\frac{\eta}{2}\frac{\sin(bt)}{b}|\xi|^2\right)\hat{u}_0+\nu_1\frac{\sin(bt)}{b}|\xi|\widehat{Z}_0\label{uz system-1}
\end{align}
and
\begin{align}
\hat{Z}^*(t,\xi)&=-\frac{\nu_2}{2}\frac{\sin(bt)}{b}|\xi|
\hat{u}_0+\Bigg(\cos(bt)-\frac{\eta}{2}\frac{\sin(bt)}{b}|\xi|^2\Bigg)\widehat{Z}_0.\label{uz system-2}
\end{align}

Thus, direct calculations yield
\begin{align*}
|\hat{u}^*(t,\xi)|^2=&|\cos(bt)|^2|\hat{u}_0|^2+\nu^2_1|\xi|^2\frac{|{\sin(bt)}|^2}{b^2}|\widehat{Z}_0|^2+\nu_1\cos(bt)\frac{\sin(bt)}{b}|\xi|\Bigg(\hat{u}_0\overline{\widehat{Z}_0}+\overline{\hat{u}_0}\widehat{Z}_0\Bigg)\\
&+\left( \frac{\eta^2}{4}|\xi|^4\frac{|{\sin(bt)}|^2}{b^2} +   {\eta} \frac{\sin(bt)\cos(bt)}{b}|\xi|^2\right)|\hat{u}_0|^2\\
&+ \nu_1 \frac{\eta}{2}\frac{|\sin(bt)|^2}{b^2}|\xi|^3  \Bigg(\hat{u}_0\overline{\widehat{Z}_0}+\overline{\hat{u}_0}\widehat{Z}_0\Bigg),
\end{align*}
and
\begin{align*}
|\hat{Z}^*(t,\xi)|^2=&\frac{\nu^2_2}{4}\frac{|{\sin(bt)}|^2}{b^2}|\xi|^2
|\widehat{u_0}(0)|^2+(\cos(bt))^2|\widehat{Z}_0(0)|^2\\
&- (\cos(bt)) \frac{\nu_2}{2}\frac{\sin(bt)}{b}|\xi|\Bigg(\hat{u}_0(0)\overline{\widehat{Z}_0(0)}+\overline{\hat{u}_0(0)}\widehat{Z}_0(0) \Bigg)\\
&+\left( \frac{\eta^2}{4}\frac{|{\sin(bt)}|^2}{b^2}|\xi|^4 - {\eta} \frac{\sin(bt)\cos(bt)}{b}|\xi|^2  \right)
|\widehat{Z}_0(0)|^2\\
&+\nu_2 \frac{\eta}{4}\frac{|\sin(bt)|^2}{b^2}|\xi|^3  \Bigg(\hat{u}_0\overline{\widehat{Z}_0}+\overline{\hat{u}_0}\widehat{Z}_0\Bigg).
\end{align*}

Noticing that $b\approx \sqrt{2\nu_1\nu_2}|\xi|$ when $|\xi|<< 1$, a direct calculation implies
\begin{align*}
|\hat{u}^*(t,\xi)|^2+\frac{2\nu_1}{\nu_2}|\hat{Z}^*(t,\xi)|^2 &=(|\cos(bt)|^2+\frac{1}{4}|{\sin(bt)}|^2)|\widehat{u_0}|^2\\
&\quad+(\frac{\nu_1}{2\nu_2}\sin^2(bt)+\frac{2\nu_1}{\nu_2}\cos^2(bt))|\widehat{Z}_0|^2+  o(|\xi|),
\end{align*}
then we get
\begin{equation}\label{ineq}
    |\hat{u}^*(t,\xi)|^2+\frac{2\nu_1}{\nu_2}|\hat{Z}^*(t,\xi)|^2|\geq c(|\widehat{Z}_0  |^2+|\widehat{u_0} |^2), ~~~\text{if}~|\xi|<< 1,
\end{equation}
for some positive constant $c\leq \frac12 \min\{\frac{1}{4} , \frac{\nu_1}{2\nu_2}, \frac{2\nu_1}{\nu_2}\} $.

Recalling that $Z=\mathbb{P}\Lambda^{-1}\operatorname{div}\tau$ and using \eqref{ineq}, we have
 $$
 |\hat{u}(t,\xi)|+|\widehat{\tau}(t,\xi)|\geq Ce^{-\eta|\xi|^2t}(|\hat{u_0}|+|\widehat{\mathbb{P}\Lambda^{-1} div\tau_0}|)
 $$
  with $|\xi|<< 1.$ This completes the proof of Lemma \ref{ls10}.
\end{proof}

Now we show the upper bounds of the low and high frequencies.
  
 \begin{lemma}\label{ls1}
Let $(u,\tau)$  the solution to the Cauchy problem  \eqref{OB}, for any fixed $c>0$, there exists  positive constant $C$ such that  
\begin{align}
|\hat{u}(t,\xi)|+|\hat{\tau}(t,\xi)|\lesssim e^{-Ct}(|\hat{u_0}(t)|+|\hat{\tau_0}(t)|),~~~|\xi|\geqslant c~~\label{1}\\
|\hat{u}(t,\xi)|+|\hat{\tau}(t,\xi)|\lesssim e^{-C|\xi|^2t}(|\hat{u_0}(t)|+|\hat{\tau_0}(t)|),~~~|\xi|\leqslant c~~\label{2}.
\end{align}
\end{lemma}
\begin{proof}
Without loss of generality,  letting $\eta=1, \nu_1=\nu_2=1$ and recalling \eqref{FU}, we have
 \begin{equation*}
  \hat{u}(t,\xi)=\mathcal{A}(t,\xi)\hat{u}_0+\mathcal{B}(t,\xi)\widehat{\mathbb{P}\Lambda^{-1}\mathrm{div}\tau}_0.
\end{equation*}

In the following part, we will give the precise expression of $\hat{\tau}$. By taking the Fourier transform of the second equation in system \eqref{OB}, we obtain
\begin{equation*}
  \hat{\tau}^{j,k}_t+|\xi|^2\hat{\tau}^{j,k}=\frac{1}{2}i(\xi_j\hat{u}^k+\xi_k\hat{u}^j)
\end{equation*}
which together with the Duhamel's principle leads to 
\begin{equation}\label{tau11}
  \hat{\tau}^{j,k}=e^{-|\xi|^2t}\hat{\tau}_0^{j,k}+\frac{1}{2}ie^{-|\xi|^2t}\int^t_0e^{|\xi|^2t'}(\xi_j\hat{u}^k+\xi_k\hat{u}^j)dt'.
\end{equation}
Noting that 
\begin{equation}\label{tau23}
\begin{aligned}
 & \int^t_0  e^{|\xi|^2t'}(\xi_j\hat{u}^k+\xi_k\hat{u}^j) dt'\\
 &=\frac{\xi_j\hat{u}^k_0+\xi_k\hat{u}^j_0}{\lambda_+-\lambda_-} \int^t_0\lambda_+e^{(\lambda_-+|\xi|^2)t'}-\lambda_-e^{(\lambda_++|\xi|^2)t'} dt'\\
&\quad+\frac{\xi_j\widehat{\mathbb{P}\mathrm{div}\tau}^k_0+\xi_k\widehat{\mathbb{P}\mathrm{div}\tau}^j_0}{\lambda_+-\lambda_-}\int^t_0
 e^{(\lambda_++|\xi|^2)t'}-e^{(\lambda_-+|\xi|^2)t'} dt'\\
 &=\frac{\xi_j\hat{u}^k_0+\xi_k\hat{u}^j_0}{\lambda_+-\lambda_-}\bigg(e^{-\lambda_-t}-e^{-\lambda_+t}\bigg)\\&\quad+\frac{\xi_j\widehat{\mathbb{P}\mathrm{div}\tau}^k_0+\xi_k\widehat{\mathbb{P}\mathrm{div}\tau}^j_0}{\lambda_+-\lambda_-}
 \bigg(\frac{1}{\lambda_+}e^{-\lambda_+t}-\frac{1}{\lambda_-}e^{-\lambda_-t}+\frac{1}{\lambda_-}-\frac{1}{\lambda_+}\bigg).
 \end{aligned}
\end{equation}
Recalling the definition \eqref{lambda} of $\lambda_{\pm}$, we obtain
\begin{equation}\label{tau24}
\begin{aligned}
ie^{-|\xi|^2t}\frac{\xi_j\hat{u}^k_0+\xi_k\hat{u}^j_0}{\lambda_+-\lambda_-}\bigg(e^{-\lambda_-t}-e^{-\lambda_+t}\bigg)=i\frac{e^{\lambda_+t}-e^{\lambda_-t}}{\lambda_+-\lambda_-}(\xi_j\hat{u}^k_0+\xi_k\hat{u}^j_0),
\end{aligned}
\end{equation}
and
\begin{equation}\label{tau24-1}
\begin{aligned}
&ie^{- |\xi|^2t}\frac{\xi_j\widehat{\mathbb{P}\mathrm{div}\tau}^k_0+\xi_k\widehat{\mathbb{P}\mathrm{div}\tau}^j_0}{\lambda_+-\lambda_-}
 \bigg(\frac{1}{\lambda_+}e^{-\lambda_+t}-\frac{1}{\lambda_-}e^{-\lambda_-t}+\frac{1}{\lambda_-}-\frac{1}{\lambda_+}\bigg)\\
&= i\frac{\lambda_-e^{-\lambda_+t}-\lambda_+e^{-\lambda_-t}-\lambda_-+\lambda_+}{\lambda_+\lambda_-}(\xi_j\widehat{\mathbb{P}\mathrm{div}\tau}^k_0
+\xi_k\widehat{\mathbb{P}\mathrm{div}\tau}^j_0)\frac{1}{\lambda_+-\lambda_-}e^{(\lambda_++\lambda_-)t}\\
&=2(\frac{\lambda_-e^{\lambda_-t}-\lambda_+e^{\lambda_+t}}{\lambda_+-\lambda_-}+e^{-|\xi|^2t})\frac{i(\xi_j\widehat{\mathbb{P}\mathrm{div}\tau}^k_0
+\xi_k\widehat{\mathbb{P}\mathrm{div}\tau}^j_0)}{|\xi|^2}.
\end{aligned}
\end{equation}

Substituting \eqref{tau23}, \eqref{tau24} and \eqref{tau24-1} into \eqref{tau11}, we  conclude   
\begin{equation*}
\begin{aligned}
  \hat{\tau}^{j,k}&=e^{-|\xi|^2t}\hat{\tau}_0^{j,k}+\frac{1}{2}\mathcal{B}(t,\xi)i(\xi_j\widehat{\Lambda^{-1}u_0}^k+\xi_k\widehat{\Lambda^{-1}u_0}^j)
  \\&\quad+\left(\mathcal{C}(t,\xi)
  +e^{-|\xi|^2t}\right)\frac{i\bigg(\xi_j\widehat{\mathbb{P}\mathrm{div}\tau}^k_0
+\xi_k\widehat{\mathbb{P}\mathrm{div}\tau}^j_0\bigg)}{|\xi|^2}.
\end{aligned}
\end{equation*}

   Recall the definition \eqref{lambda} of $\lambda_{\pm}$ again, one has
    \begin{equation*}
     \lambda_{\pm}=\frac{-|\xi|^2\pm\sqrt{|\xi|^4-2|\xi|^2}}{2}.
\end{equation*}
 Note the Green matrix of the system \eqref{uz system-1}-\eqref{uz system-2}  have the following equality 
\begin{equation*}
\hat{G}=
  \left(
  \begin{array}{cc}
    \frac{\lambda_+e^{\lambda_-t}-\lambda_-e^{\lambda_+t}}{\lambda_+-\lambda_-} & \frac{e^{\lambda_+t}-e^{\lambda_-t}}{\lambda_+-\lambda_-}|\xi| \\
    -\frac{1}{2}\frac{e^{\lambda_+t}-e^{\lambda_-t}}{\lambda_+-\lambda_-}|\xi| & -\frac{\lambda_-e^{\lambda_-t}-\lambda_+e^{\lambda_+t}}{\lambda_+-\lambda_-} \\
  \end{array}
  \right)
  =
  \left(
    \begin{array}{cc}
      \mathcal{A} & \mathcal{ B} \\
      -\frac{1}{2}\mathcal{B} & -\mathcal{C }\\
    \end{array}
\right).
\end{equation*}
Thanks to Lemma 4.1 and Lemma 4.2 in \cite{MR2675485}, we  obtain the estimate \eqref{1}.
The readers may also refer to \cite{Chen}.
\par
 If \(|\xi|^4-2|\xi|^2 \geq 0\), i.e., \(|\xi| \geq \sqrt{2}\), we obtain for $\sqrt{2}\leq|\xi| \leq c$
 \begin{equation}\label{lamb1}
     \lambda_{-} \leq \lambda_{+} = \frac{-|\xi|^2 + \sqrt{|\xi|^4 - 2|\xi|^2}}{2} = \frac{- |\xi|^2}{ |\xi|^2 + \sqrt{|\xi|^4 - 2|\xi|^2}} \leq -c_1|\xi|^2,  
 \end{equation}
 where
    \[
    c_1 = \frac{1}{ c^2 + \max\{0, \sqrt{c^4 - 2c^2}\}}.
    \]
    If \(|\xi|^4-2|\xi|^2 <0\), i.e., \(0 \leq |\xi| < \sqrt{2}\), one has
\begin{equation}\label{lamb2}
    \operatorname{Re}(\lambda_{\pm}) = -\frac{|\xi|^2}{2}.
\end{equation}
Therefore, combining \eqref{lamb1}  with \eqref{lamb2}, we can conclude that for \(c_2 = \min\{c_1, \frac{1}{2}\}\),
    \[
    \operatorname{Re}(\lambda_{\pm}) \leq -c_2|\xi|^2 \ \text{and} \quad |e^{\lambda_{\pm}t}| \leq e^{-c_2|\xi|^2t}.
    \]
    Then for \( 1\leq|\xi| \leq c\) and let \(\mathcal{E}(t, \xi) \triangleq \frac{e^{\lambda_{+}t} - e^{\lambda_{-}t}}{\lambda_{+} - \lambda_{-}}\), one can deduce that
    \[
    \begin{aligned}
    |\mathcal{E}(t, \xi)|  \leq Ce^{-\frac{c_2}{2}|\xi|^2t}.
    \end{aligned}
    \]
\par
For $0<|\xi|<1$, note that $|\lambda_{+}-\lambda_-|\geq |\xi|$, we have
$$
|\mathcal{E}(t, \xi)| \leq \frac{e^{-c|\xi|^2t}}{|\xi|}.
$$
 Finally,  by utilizing  the following equalities
    \[
    \mathcal{A}(t, \xi) = -\lambda_{+}\mathcal{E}(t, \xi) + e^{\lambda_{+}t}, \quad \mathcal{B}(t, \xi) = |\xi|\mathcal{E}(t, \xi), \quad \mathcal{C}(t, \xi) = -\lambda_{+}\mathcal{E}(t, \xi) - e^{\lambda_{-}t},\]
we then finish the proof of the inequality \eqref{2} in Lemma \ref{ls1}.
\end{proof}

{  By applying the similar argument to prove Proposition 3.1 in \cite{MR4792291}, we  can establish the following proposition.}
\begin{proposition}\label{wp1}
Let $\sigma, \sigma_1\in \mathbb{R}$ with $\sigma>\sigma_1$ and $(u_L, \tau_L)$ is a solution to system \eqref{OB} with initial data $(u_0,\tau_0) \in \dot{B}^\sigma_{2,1}$. Then, for any given time $t_L \geq 0$, 
$$
\|(u_L, \tau_L)(t)\|_{\dot{\mathbb{B}}^\sigma_{2,p}}\lesssim \langle t \rangle^{-\frac{1}{2}(\sigma-\sigma_1)},~~t>t_L,
$$
if and only if $(u_0,\tau_0)^\ell\in \dot{B}^{\sigma_1}_{2,\infty}$.

Moreover, if 
$$
\langle t \rangle^{-\frac{1}{2}(\sigma-\sigma_1)} \lesssim\|(u_L, \tau_L)(t)\|_{\dot{\mathbb{B}}^\sigma_{2,p}}\lesssim \langle t \rangle^{-\frac{1}{2}(\sigma-\sigma_1)},~~t>t_L,
$$
we have $(u_0,\tau_0)^\ell\in \dot{\textbf{B}}^{\sigma_1}_{2,\infty}$. 
On the other hand, if $(u_0,{\mathbb{P}\Lambda^{-1}\mathrm{div}\tau}_0)^\ell\in \dot{\textbf{B}}^{\sigma_1}_{2,\infty} $ and $\tau_0^\ell\in \dot{{B}}^{\sigma_1}_{2,\infty}$, we have
$$
\langle t \rangle^{-\frac{1}{2}(\sigma-\sigma_1)} \lesssim\|(u_L, \tau_L)(t)\|_{\dot{\mathbb{B}}^\sigma_{2,p}}\lesssim \langle t \rangle^{-\frac{1}{2}(\sigma-\sigma_1)},~~t>t_L.
$$
\end{proposition}

We now apply Wiener's argument to the linear incompressible Oldroyd-B system \eqref{OB}.
\par
\subsection{Wiegner's argument for the linear system}
\label{sec:wiegner}
\par

\begin{proposition}\label{prop1}
Let $p$ satisfy \eqref{ppp}, there exists a positive constant $C$ such that  
\begin{align}\label{wi1}
&\|(u_L,\tau_L)\|^\ell_{\tilde{L}^\infty_t(\dot{B}^{\frac{d}{2}-1}_{2,1})}+\|(u_L,\tau_L)\|^\ell_{\tilde{L}^1_t(\dot{B}^{\frac{d}{2}+1}_{2,1})}+\|e^{Ct}(\nabla u_L,\tau_L)\|^h_{\tilde{L}^\infty_t(\dot{B}^{\frac{d}{p}}_{p,1})}\nonumber\\
&\ \ \ +\|(u_L,\nabla \tau_L)\|^h_{\tilde{L}^1_t(\dot{B}^{\frac{d}{p}+1}_{p,1})}+\|e^{Ct}\tau_L\|^h_{\tilde{L}^\infty_t(1,t;\dot{B}^{\frac{d}{p}+2}_{p,1})}\lesssim E_{0}.
\end{align}
 Moreover, assume that $\sigma_0\leq \sigma_1<\frac{d}{2}-1$ and $\sigma>\sigma_1$, then 
\begin{align}\label{wi2}
\|(u_L,\tau_L)^\ell(t)\|_{\dot{B}^{\sigma}_{2,1}}\lesssim \langle t \rangle^{-\frac{1}{2}(\sigma-\sigma_1)},~~t>t_L,
\end{align}
if and only if $(u_0,\tau_0)^\ell\in \dot{B}^{\sigma_1}_{2,\infty}$.

Moreover, if 
\begin{align}\label{wi3}
\langle t \rangle^{-\frac{1}{2}(\sigma-\sigma_1)} \lesssim\|(u_L,\tau_L)^\ell(t)\|_{\dot{B}^{\sigma}_{2,1}}\lesssim \langle t \rangle^{-\frac{1}{2}(\sigma-\sigma_1)},~~t>t_L,
\end{align}
one has  $(u_0,\tau_0)^\ell\in \dot{\textbf{B}}^{\sigma_1}_{2,\infty}$. On the other hand, if
 $(u_0,{\mathbb{P}\Lambda^{-1}\mathrm{div}\tau}_0)^\ell\in \dot{\textbf{B}}^{\sigma_1}_{2,\infty}$ and $\tau_0^\ell\in \dot{{B}}^{\sigma_1}_{2,\infty}$, we have
 \begin{align}\label{wi4}
\langle t \rangle^{-\frac{1}{2}(\sigma-\sigma_1)} \lesssim\|(u_L,\tau_L)^\ell(t)\|_{\dot{B}^{\sigma}_{2,1}}\lesssim \langle t \rangle^{-\frac{1}{2}(\sigma-\sigma_1)},~~t>t_L.
\end{align}
\end{proposition}

The inequalities \eqref{wi2}, \eqref{wi3} and \eqref{wi4} follow directly from Proposition \ref{wp1}, while the inequality \eqref{wi1} will be established in the next two subsections.
 
\subsection{Low frequencies estimates of the solution to the linear system}
\label{sec:lowfreq}
\qquad \par \vspace{3pt}

In the low-frequencies estimates, as it has been pointed out by  Wu and Zhao \cite{Wu-Zhao},  we can not adopt the energy estimate directly due to the lack of dissipation in the velocity equation.  In \cite{Wu-Zhao} the authors introduced a new term $\Lambda^{-1}\mathbb{P}\operatorname{div} \tau$ and presented the  estimate of the low-frequencies of  $(u,\Lambda^{-1}\mathbb{P}\operatorname{div} \tau)$.  
  
Different from \cite{Wu-Zhao}, we decompose the stress tensor $\tau$ into two parts, namely,
\begin{align}\label{tau-v-w}
\tau=-\Lambda^{-1}\nabla v-\Lambda^{-1}\operatorname{div} w,
\end{align}
where
\begin{align*}
w=\Lambda^{-1}\operatorname{curl}\tau,\qquad
v=\Lambda^{-1}\operatorname{div}\tau,
\end{align*}
with the components
$
[\operatorname{curl}\tau]^{i,j,m}=\partial_m\tau^{i,j}-\partial_j\tau^{i,m}
$
and
$
[\operatorname{div}\tau]^{i}=\partial_j\tau^{i,j}
$
\footnote{Here and throughout, the indices $i,j,m$ run from $1$ to $d$.}.
This decomposition follows from the identities
\begin{align*}
[\nabla\operatorname{div}\tau]^{i,j}&=\partial_j\partial_m\tau^{i,m},\ \ \ 
[\operatorname{div}(\operatorname{curl}\tau)]^{i,j}=\partial_m\partial_m\tau^{i,j}-\partial_m\partial_j\tau^{i,m},
\end{align*}
which together imply
\begin{align*}
\Delta\tau
=\nabla\operatorname{div}\tau+\operatorname{div}(\operatorname{curl}\tau)
=\Lambda\nabla v+\Lambda\operatorname{div} w,
\end{align*}
and consequently \eqref{tau-v-w} holds.
In fact, we establish a new key identity
that allows us to replace the pair $(u,\Lambda^{-1}\mathbb{P}\operatorname{div}\tau)$ in the low‑frequencies estimates by the triplet $(u,v,w)$. More precisely, we will prove the following lemma.

\begin{lemma}\label{lem3.1}
	Let $(u,\tau)$ be the solution of system \eqref{OB}.  
     Then for any $k_0\in\mathbb{Z}$, 
     there exists a positive constant $C$ such that
    \begin{equation} \label{L-k-estimate}
        \frac{1}{2}\frac{d}{dt}\mathcal{L}_k(t)+C2^{2k}\|(u_k,\tau_k)\|_{L^2} \leq 0,
    \end{equation}
 for all $k\leq k_0$ and  $t\geq 0$, with
\begin{align*}
	\mathcal{L}_k(t)&\triangleq \sqrt{ 2\|u_k\|^2_{L^2}+2\|v_k\|^2_{L^2}+\|w_k\|^2_{L^2}+\|\Lambda u_k\|^2_{L^2}+2\langle \Lambda u_k,v_k\rangle }, \quad k\in\mathbb{Z}
\end{align*}
and $(u_k,v_k,w_k)= (\dot{\Delta}_k u,\dot{\Delta}_k v,\dot{\Delta}_k w)$.
\end{lemma}

\begin{proof}
Without loss of generality, we let $\eta=1, \nu_1=\nu_2=1$. Applying the  operators $\Lambda^{-1}\operatorname{div}$ and  $\Lambda^{-1}\operatorname{curl}$ to the second equation of \eqref{OB},  and noticing that $\operatorname{div} \tau=\Lambda v$,
one gets 
\begin{equation}\label{OB-2}
	\begin{cases}
		\partial_tu-\mathbb{P}\Lambda v=0,\\
		\partial_tv-\Delta v+\frac{1}{2}\Lambda u=0,\\
		\partial_tw-\Delta w-\frac{1}{2}\Lambda^{-1}\operatorname{curl}(\nabla u)^\top=0,
	\end{cases}
\end{equation}
where we have used the facts
\begin{align*}
	\Big[\Lambda^{-1}\operatorname{div}\big((\nabla u)^\top\big)\Big]^i=\Lambda^{-1}\partial_j\partial_iu^j =0
\end{align*}
 and
 \begin{align*}
\Big[\Lambda^{-1}\mathrm{curl}\big(\nabla u\big)\Big]^{i,j,m}=\Lambda^{-1}(\partial_m\partial_ju^i-\partial_j\partial_mu^i)=0.
\end{align*}

Then system \eqref{OB-2} can be localized as
\begin{equation}\label{OB-3}
	\begin{cases}
		\partial_t u_k-\mathbb{P}\Lambda v_k=0,\\
		\partial_t v_k-\Delta v_k+\frac{1}{2}\Lambda u_k=0,\\
		\partial_t w_k-\Delta w_k-\frac{1}{2}\Lambda^{-1}\operatorname{curl}(\nabla u)_k^\top=0,
	\end{cases}
\end{equation}

Taking the $L^2$-inner product of the first equation of \eqref{OB-3} with $u_k$ and of the third  with   $w_k$,
it holds that
\begin{equation}\label{u-k-1}
	\frac{1}{2}\frac{d}{dt}\|u_k\|^2_{L^2}-\langle \mathbb{P}\operatorname{div}\tau_k, u_k \rangle=0
\end{equation}
and
\begin{equation}\label{w-k-1}
	\frac{1}{2}\frac{d}{dt}\|w_k\|^2_{L^2}+\|\Lambda w_k\|^2_{L^2}+\langle\operatorname{div}\tau_k, u_k \rangle =0,
\end{equation}
where we have employed 
\begin{align}\label{key-identity}
	\langle -\frac{1}{2}\Lambda^{-1}\operatorname{curl}(\nabla u)_k^\top,w_k\rangle &=\langle -\frac{1}{2}\Lambda^{-1}\operatorname{curl}(\nabla u)_k^\top,\Lambda^{-1}\operatorname{curl} \tau_k\rangle
	\nonumber \\
	&=-\frac{1}{2}\int_{\mathbb{R}^3} (\Lambda^{-2}\partial_m\partial_iu_k^j-\Lambda^{-2}\partial_j\partial_iu_k^m)(\partial_m\tau_k^{i,j}-\partial_j\tau_k^{i,m})\mathrm{d}x
	\nonumber\\
	&=\frac{1}{2}\int_{\mathbb{R}^3}\big( \Lambda^{-2} \partial_m\partial_m\partial_iu_k^j \tau_k^{i,j}+\Lambda^{-2} \partial_j\partial_j\partial_iu_k^m \tau_k^{i,m} \big) \mathrm{d}x
	\nonumber \\
	&=-\int_{\mathbb{R}^3}\tau_k:\nabla u_k\mathrm{d}x\nonumber\\&=\langle \operatorname{div}\tau_k, u_k\rangle,\nonumber
\end{align}
which follows from  integration by parts, the divergence‑free condition $\eqref{OB}_3$ and the symmetry of $\tau$.

Furthermore, applying the definition of the operator $\mathbb{P}$ together with Plancherel’s theorem leads to
\begin{equation}
	\langle \mathbb{P}\operatorname{div} \tau_k, u_k \rangle=\langle \operatorname{div}\tau_k,u_k\rangle.\label{u-k-11}
\end{equation}
It follows from \eqref{u-k-1}, \eqref{w-k-1} and \eqref{u-k-11} that
\begin{equation}\label{uw-k-1}
	\frac{1}{2}\frac{d}{dt}(\|u_k\|^2_{L^2}+\|w_k\|^2_{L^2})+\|\Lambda w_k\|^2_{L^2}=0.
\end{equation}

By taking the $L^2$-inner product of  $\eqref{OB-3}_2$ with $v_k$,  one obtains
\begin{equation}\label{v-k-1}
	\frac{1}{2}\frac{d}{dt}\|v_k\|^2_{L^2}+ \|\Lambda v_k\|^2_{L^2}+\langle \frac{1}{2}\Lambda u_k,v_k\rangle =0,
\end{equation}
which together with \eqref{u-k-1} yields that
\begin{equation}\label{uv-k-1}
	\frac{1}{2}\frac{d}{dt}(2\|v_k\|^2_{L^2}+\|u_k\|^2_{L^2})+2 \|\Lambda v_k\|^2_{L^2}
	=
	0.
\end{equation}

By taking the $L^2$-inner product of  $\eqref{OB-3}_1$ with $\Lambda v_k$,  it holds that
\begin{equation}\label{uv-k-21}
	\langle \frac{d}{dt}u_k,\Lambda v_k\rangle +\|\operatorname{div} v_k\|^2_{L^2}=\|\Lambda v_k\|^2_{L^2},
\end{equation}
where  we have  used the following identities
\begin{equation*}
	\langle \mathbb{P}\Lambda v_k,\Lambda v_k\rangle =\|\Lambda v_k\|^2_{L^2}+\langle \nabla \operatorname{div}v_k,v_k\rangle =\|\Lambda v_k\|^2_{L^2}-\|\operatorname{div} v_k\|^2_{L^2}.
\end{equation*}
Simililarly, the $L^2$-inner product of $\eqref{OB-3}_2$ with $\Lambda u_k$ yields that
\begin{equation}\label{uv-k-22}
	\langle\frac{d}{dt}\Lambda v_k,u_k\rangle-\langle \Delta v_k,\Lambda u_k\rangle +\frac{1}{2}\|\Lambda u_k\|^2_{L^2}=0.
\end{equation}
Combining \eqref{uv-k-21} and \eqref{uv-k-22} together, there holds 
\begin{align}\label{uv-k-23}
	\frac{d}{dt}\langle u_k,\Lambda v_k\rangle +\frac{1}{2}\|\Lambda u_k\|^2_{L^2} +\|\operatorname{div} v_k\|_{L^2}^2
	=&
	\|\Lambda v_k\|^2_{L^2}
+\langle  \Delta v_k,\Lambda u_k\rangle.
\end{align}

Finally, applying the operator  $\Lambda $ on $\eqref{OB-3}_1$ with $\Lambda u_k$,  we obtain
\begin{equation}\label{u-k-2}
	\frac{1}{2}\frac{d}{dt}\|\Lambda u_k\|^2_{L^2}=-\langle \Delta v_k,\Lambda u_k\rangle.
\end{equation}
By combining \eqref{uw-k-1}, \eqref{uv-k-1}, \eqref{uv-k-23} and \eqref{u-k-2} together, one gets that
\begin{align*}
	&\frac{1}{2}\frac{d}{dt} \mathcal{L}_k^2(t)
	+\frac{1}{2}\|\Lambda u_k\|^2_{L^2}+
	\|\Lambda v_k\|^2_{L^2}+ \|\Lambda w_k\|^2_{L^2}
	+ \|\operatorname{div} v_k\|_{L^2}^2
	=0,
\end{align*}
which yields that for all $k\leq k_0$,
\begin{equation}\label{L-k-1}
	\frac{1}{2}\frac{d}{dt}\mathcal{L}^2_k(t)+C2^{2k}\|(u_k,\tau_k)\|^2_{L^2} \leq 0,
\end{equation}
where we have used 
\begin{align*}
	\frac{1}{2}\|\Lambda u_k\|^2_{L^2}+
	\|\Lambda v_k\|^2_{L^2}+\|\Lambda w_k\|^2_{L^2}
	+\|\operatorname{div} v_k\|_{L^2}^2\geq C 2^{2k}\|(u_k,\tau_k)\|^2_{L^2},
\end{align*}
 for some constant $C>0$.
 
Thanks to  the following inequalities
\begin{equation*}
	|2(\Lambda u_k,v_k)|\leq 2\|\Lambda u_k\|_{L^2}\|v_k\|_{L^2}\leq \frac{2}{3}\|\Lambda u_k\|^2_{L^2}+\frac{3}{2}\|v_k\|^2_{L^2},
\end{equation*}
we have
\begin{equation*}\label{identity of L-k}
	\mathcal{L}_k(t)\approx\|(u_k,v_k,w_k,\Lambda u_k)\|_{L^2}\approx\|(u_k,v_k,w_k)\|_{L^2}\approx\|(u_k,\tau_k)\|_{L^2},   \quad k\leq k_0.
\end{equation*}

Combining the above estimate with \eqref{L-k-1}, one gets that \eqref{L-k-estimate}  holds true for all  $k\leq k_0$ and $t\geq0$.
Thus, the proof of  Lemma \ref{lem3.1} is completed.
\end{proof}
\par

\subsection{High frequencies estimates of the solution to the linear system }\label{sec:highfreq}
\qquad \par \vspace{3pt}
 Applying the operator $\mathbb{P}$ on the first equation of  \eqref{OB} and taking  $\eta=1, \nu_1=1, \nu_2=1$ without losing generality,  one obtains
\begin{align}\label{OB-1111}
	\begin{cases}
		\partial_tu-\mathbb{P}\mathrm{div}\ \tau=0,\\
		\partial_t\tau-\Delta\tau-D(u)=0.
	\end{cases}
\end{align}
Invoked by Haspot's method  \cite{Haspot} we introduce  the following {\it effective tensor}
\begin{align}\label{effective-tesor}
	W=\nabla(-\Delta)^{-1}(\frac{1}{2}u+\operatorname{div}\tau).
\end{align}

It follows directly that 
\begin{equation}\label{L1}
	\nabla\operatorname{div}W=\Delta W=-\nabla \big(\frac{1}{2} u+\operatorname{div} \tau\big)~~\text{ and }~~\tau=\mathbb{P}\tau-W+\frac{1}{2}\nabla(-\Delta)^{-1}u.
\end{equation}
 From  \eqref{OB-1}  we obtain that $(u,W, \mathbb{P}\tau)$ satisfies\footnote{One can easy to get that  $\operatorname{div} D(u)=\frac{1}{2}\Delta u $ and $\mathbb{P}D(u)=\frac{1}{2}(\nabla u)^\top$.}
\begin{equation}\label{u-tau-omega}
	\begin{cases}
		\partial_tu+\frac{1}{2}u=-\mathbb{P}\mathrm{div}W,\\
		\partial_t\mathbb{P}\tau-\Delta\mathbb{P}\tau=\frac{1}{2}(\nabla u)^\top,\\
		\partial_tW-\Delta W=-\frac{1}{2}\nabla(-\Delta)^{-1}\mathbb{P}\mathrm{div}W-\frac{1}{4}\nabla(-\Delta)^{-1}u.
	\end{cases}
\end{equation}

Then we can obtain the following lemma. 
\begin{lemma}\label{lem3.2}
	Let $(u,\tau)$ be the solution of system \eqref{OB}, 
     then for large $k_0\in\mathbb{Z}$, there exists two  positive constants $\gamma$ and  $C$ such that
	\begin{equation}\label{eq78'}
	\frac{d}{dt}\mathcal{J}_k(t)+C\|(2^k u_k,2^{2k}\tau_k)\|_{L^p}\leq \mathcal{J}_k(0)
\end{equation}
	for all $t\geq 0$ and $k\geq k_0-1$, where 
	\begin{align*}
		\mathcal{J}_k(t)&\triangleq 2\gamma\|e^{Ct}\nabla u_k\|_{L^p}+\gamma\|e^{Ct}\mathbb{P}\tau_k\|_{L^p}+\|e^{Ct}W_k\|_{L^p}.
	\end{align*}
	
\end{lemma}

\begin{proof}
	Applying $\dot{\Delta}_k$ to $\eqref{u-tau-omega}_2$, multiplying the resultant equality by  $e^{Cpt}|(\mathbb{P}\tau_k)^{i,j}|^{p-2}(\mathbb{P}\tau_k)^{i,j}$ and integrating over $\mathbb{R}^d$, we arrive at 
	\begin{align}\label{eq-2.10}
		&\frac{1}{p}\frac{d}{dt}\|e^{Ct}(\mathbb{P}\tau_k)^{i,j}\|^p_{L^p}+e^{Cpt}c_p2^{2k}\|\mathbb{P}\tau_k\|_{L^p}^p
        \nonumber\\
		&\leq(\frac{1}{2}\|e^{Ct}(\nabla u_k)^{i,j}\|_{L^p}\|e^{Ct}(\mathbb{P}\tau_k)^{i,j}\|^{p-1}_{L^p}+\frac{1}{p}{Cp}e^{Cpt}\|(\mathbb{P}\tau_k)^{i,j}\|^p_{L^p},
	\end{align}
where we have used 
  \begin{align*}
  -\int_{\mathbb{R}^d}e^{Cpt}\Delta(\mathbb{P}\tau_k)^{i,j}|(\mathbb{P}\tau_k)^{i,j}|^{p-2}(\mathbb{P}\tau_k)^{i,j}\mathrm{d}x
  \geq e^{Cpt}c_p2^{2k}\|\mathbb{P}\tau_k\|_{L^p}^p
  \end{align*}
   for some positive constant $c_p$ dependent of  $p$.

By choosing small $C$ in \eqref{eq-2.10} we immediately deduce that
	\begin{equation}\label{tau-k-h}
		\frac{d}{dt}\|e^{Ct}\mathbb{P}\tau_k\|_{L^p}+c_p2^{2k}\|e^{Ct}\mathbb{P}\tau_k\|_{L^p}\leq\frac{1}{2}\|e^{Ct}\nabla u_k\|_{L^p}.
	\end{equation}
	Similarly, we can estimate  $W$ as follows
	\begin{equation}\label{omega-k-h}
		\frac{d}{dt}\|e^{Ct}W_k\|_{L^p}+c_p2^{2k}e^{Ct}\|W_k\|_{L^p}\leq\|e^{Ct}W_k\|_{L^p}+\|e^{Ct}\nabla(-\Delta)^{-1}u_k\|_{L^p}.
	\end{equation}
    
	On the other hand, applying the operator $\partial_i\dot{\Delta}_k$ ($i=1,2,\cdots,d$) to  $\eqref{u-tau-omega}_1$ yields
	\begin{equation}\label{nabla-u-k}
		\partial_t\partial_iu_k+\frac{1}{2}\partial_iu_k=-\partial_i\mathbb{P}\mathrm{div}W_k.
	\end{equation}
Multiplying \eqref{nabla-u-k} by $e^{Cpt}|\partial_iu_k|^{p-2}\partial_i u_k$ and  integrating over $\mathbb{R}^d$, one has

	\begin{align}\label{nabla-u-k-h}
		&\frac{1}{p}\frac{d}{dt}\|e^{Ct}\partial_iu_k\|^p_{L^p}+\frac{1}{2}e^{Cpt}\|\partial_i u_k\|^p_{L^p}\nonumber\\
      &=\int_{\mathbb{R}^d}e^{Ctp}(-\partial_i\mathbb{P}\mathrm{div}W_k)|\partial_iu_k|^{p-2}\partial_iu_k\mathrm{d}x+Ce^{Cpt}\|\partial_iu_k\|^p_{L^p}
		 \\
&\leq\|e^{Ct}\partial_i\mathbb{P}\mathrm{div}W_k\|_{L^p}\|e^{Ct}\partial_iu_k\|^{p-1}_{L^p}+Ce^{Cpt}\|\partial_iu_k\|^p_{L^p}\nonumber
	\end{align}
	which then implies
	\begin{equation}\label{nabla-u-k-h1}
		\frac{d}{dt}\|e^{Ct}\nabla u_k\|_{L^p}+\frac{1}{4}\|e^{Ct}\nabla u_k\|_{L^p}\leq 2^{2k}\|e^{Ct}W_k\|_{L^p}.
	\end{equation}

	By adding  \eqref{tau-k-h} to $\eqref{nabla-u-k-h1}\times4$, it follows that
	\begin{align}\label{u-tau-k-h}
		&\frac{d}{dt}(4\|e^{Ct}\nabla u_k\|_{L^p}+\|e^{Ct}\mathbb{P}\tau_k\|_{L^p})+\frac{1}{2}\|e^{Ct}\nabla u_k\|_{L^p}+2c_p2^{2k}\|e^{Ct}\mathbb{P}\tau_k\|_{L^p}
		\nonumber\\
		& \leq2^{2k+2}\|e^{Ct}W_k\|_{L^p}.
	\end{align}

	Moreover, multiplying \eqref{u-tau-k-h} by $\gamma>0$ which will be determined  later and then  adding the resulting inequality to  \eqref{omega-k-h}, one gets 
	\begin{align*} 
		&\frac{d}{dt}(4\gamma\|e^{Ct}\nabla u_k\|_{L^p}+\gamma\|e^{Ct}\mathbb{P}\tau_k\|_{L^p}+\|e^{Ct}W_k\|_{L^p})+\frac{1}{2}\gamma\|e^{Ct}\nabla u_k\|_{L^p}\nonumber\\
		&\qquad+2c_p\gamma2^{2k}\|e^{Ct}\mathbb{P}\tau_k\|_{L^p}+c_p2^{2k}\|e^{Ct}W_k\|_{L^p}
		\nonumber\\
		&\leq \gamma2^{2k+2}\|e^{Ct}W_k\|_{L^p}+\|e^{Ct}W_k\|_{L^p}+\|e^{Ct}\nabla(-\Delta)^{-1}u_k\|_{L^p}.
	\end{align*}
    
	Notice that it holds
	\begin{align*}
		\|\nabla(-\Delta)^{-1}u_k\|_{L^p}\leq2^{-2k}\|\nabla u_k\|_{L^p}\leq2^{2-2k_0}\|\nabla u_k\|_{L^p},~~~~\mbox{for all }k\geq k_0-1.
	\end{align*}
	Choosing $\gamma$ sufficiently small and $k_0$ suitable large such that   
	\begin{align}\label{eq56}
		\frac{d}{dt}\mathcal{J}_k(t)&+C(\|e^{Ct}\nabla u_k\|_{L^p}+2^{2k}\|e^{Ct}\mathbb{P}\tau_k\|_{L^p}+2^{2k}\|e^{Ct}W_k\|_{L^p})
		 \leq 0
	\end{align}
     for $k\geq k_0-1.$

	Due to the second equality in \eqref{L1}, we have
	\begin{equation}\label{zz4}
		\|(\nabla u_k,2^{2k}W_k,2^{2k}\mathbb{P}\tau_k)\|_{L^p}\approx	\|(\nabla u_k,2^{2k}\tau_k)\|_{L^p},
	\end{equation}
	for all $k\geq k_0-1$, which together with \eqref{eq56} yields  
\eqref{eq78'}, and  then  the proof of Lemma \ref{lem3.2} is completed.
\end{proof}

\textit{Completeness of the proof of  Proposition \ref{prop1}.}
Now we will finish the proof of  Proposition \ref{prop1}. In view of \eqref{E-0} and Lemma \ref{lem3.2} we have 
$$
\|e^{Ct}(\nabla u_L,\tau_L)\|^h_{\tilde{L}^\infty_t(\dot{B}^{\frac{d}{p}}_{p,1})}\lesssim E_0<<1.
$$
In order to gain regularity and decay altogether for the high frequencies of tensor $\tau_L$, we reformulate the second equation as follows
\begin{align*}
\partial_t(\chi(t)\tau_L)-\Delta(\chi(t)\tau_L)=\chi'(t)\tau_L+\chi(t)D(u_L),
\end{align*}
where $\chi(t)\in C^1(R_+)$ satisfies $\chi(t)=t$ for $0\leq t\leq \frac{1}{2}$ and $\chi(t)=e^{Ct}$ for $t>1$. Then it follows the maximal regularity estimate  {for the parabolic equation (see (3.39) in \cite{BCD-2011})} that
$$
\|\chi(t)\tau_L\|^h_{\tilde{L}^\infty_t(\dot{B}^{\frac{d}{p}+2}_{p,1})}\lesssim \|e^{Ct}(\tau_L,\nabla u_L)\|^h_{\tilde{L}^\infty_t(\dot{B}^{\frac{d}{p}}_{p,1})}\lesssim E_0<<1,
$$
where we used the fact $\chi(t)\tau_L|_{t=0}=0$.  Thus we complete the  proof to Proposition \ref{prop1}.

\vskip .2in
\section{Proof of Theorem 
\ref{Thm}}
\label{sec:sufficient}
\subsection{Sufficient condition}\label{subsec:sufficient}
{ 
This section is devoted to establishing the sufficient condition in Theorem 
\ref{Thm}, namely that the low-frequencies assumption $(u_0,\tau_0)^\ell\in 
\dot{B}^{\sigma_1}_{2,\infty}$ with $\sigma_0<\sigma_1<\frac{d}{2}-1$ implies 
the optimal upper decay bound for the solution $(u,\tau)$ to the Cauchy problem 
\eqref{OB}. The central result of this section is Proposition \ref{cob2}, 
which establishes the time-weighted functional inequality $\widetilde{D}_p(t) 
\lesssim 1$, and from which the upper bound in \eqref{wi2} follows directly. 
\begin{proposition}\label{cob2}
	Assume that  $(u_0,\tau_0)^\ell\in \dot{B}^{\sigma_1}_{2,\infty}$ with $\sigma_0< \sigma_1<\frac{d}{2}-1$, then
	$$
	\tilde{D}_p(t)\lesssim 1,\ t>0,
	$$
where the difference functional $\tilde{D}_p(t)$ is defined by
	\begin{align*}
		\tilde{D}_p(t) =\sup_{\sigma_1<\sigma<\frac{d}{2}}\|\langle s \rangle^{\frac{1}{2}(\sigma-\sigma_1+\sigma_2)}(\tilde{u},\tilde{\tau})\|^\ell_{\tilde{L}^\infty_t(\dot{B}^{\sigma}_{2,1})}+
		\|\langle s \rangle^{\alpha_*}(\nabla \tilde{u}, \tilde{\tau})\|^h_{\tilde{L}^\infty_t(\dot{B}^{\frac{d}{p}}_{p,1})}+\|s^{\alpha_*}\tilde{\tau}\|^h_{\tilde{L}^\infty(1,t;\dot{B}^{\frac{d}{p}+1}_{p,1})}
	\end{align*}
	with $\alpha_*=\frac{1}{2}(\frac{d}{2}-\sigma_1+\sigma_2)-$ and the number $\sigma_2=\min\{\sigma’_2,\sigma''\}\in(0,1]$ is given by
	\begin{align*}
		\sigma'_2=
		\begin{cases}
			1,\qquad\qquad\qquad\qquad\qquad\qquad~~~if~\sigma_1<\sigma\leq \frac{d}{2}-1, \sigma_1<\frac{d}{2}-2,\\
			1-,\qquad\qquad\qquad\qquad\ \quad\qquad~~if~\sigma_1<\sigma\leq \frac{d}{2}-1, \sigma_1=\frac{d}{2}-2,\\
			\frac{d}{2}-1-\sigma_1,\qquad\qquad\qquad\quad~~if~\sigma_1<\sigma\leq \frac{d}{2}-1, \frac{d}{2}-2<\sigma_1<\frac{d}{2}-1,\\
			\min\{1/2,(\frac{d}{2}-1-\sigma_1)-\},~~if~ \frac{d}{2}-1<\sigma<\frac{d}{2},
		\end{cases}
	\end{align*}
	and $\sigma''$ is a  small enough positive constant.
\end{proposition}

The proof is organized as follows. We begin in Proposition \ref{bdd-negative-norm} by showing that the low-frequencies norm 
$\|(u,\tau)(t)\|^\ell_{\dot{B}^{\sigma_1}_{2,\infty}}$ remains uniformly bounded 
for all time, which serves as a crucial ingredient throughout the section. We 
then establish the low-frequencies decay estimates for the error term 
$(\tilde{u},\tilde{\tau})=(u-u_L,\tau-\tau_L)$ in Lemmas \ref{cl1}, 
\ref{cl3} and \ref{cl5}, covering the short-time regime $0<t\leq 2$, 
the intermediate regime $t>2$ with $\sigma_1<\sigma\leq\frac{d}{2}-1$, and 
the remaining range $\frac{d}{2}-1<\sigma\leq\frac{d}{2}$ respectively. A 
key technical difficulty in these estimates is the treatment of the nonlinear 
term $Q(\tau,\nabla u)$ in the low-frequencies regime, where the absence of an 
incompressibility condition on $\tau$ prevents a direct estimate. This is 
overcome by the refined decomposition
\begin{equation*}
	\tau^\ell_{ik}\partial_k\tilde{u}^h_j 
	= \partial_k(\tau^\ell_{ik}\tilde{u}^h_j) - \partial_k\tau^\ell_{ik}\tilde{u}^h_j,
\end{equation*}
which allows the two resulting terms to be estimated separately with carefully 
chosen regularity indices. The high-frequencies estimates for $(\tilde{u},\tilde{\tau})$ 
are then carried out in Lemma \ref{cl10}, where the effective tensor formalism 
introduced in Section \ref{sec:highfreq} is again employed to handle the loss of 
one derivative in the velocity field at high frequencies. The required bound on 
the intermediate quantity $\|s^{\frac{1}{2}(\frac{d}{2}-\sigma_1)-}u\|_{\tilde{L}^\infty(1,t;\dot{B}^{\frac{d}{p}+1}_{p,1})}$ 
is obtained in Lemma \ref{cl11}. Collecting all these estimates, the proof of 
Proposition \ref{cob2} is completed in Subsection \ref{ddd} via 
a standard bootstrap argument, exploiting the smallness provided by the factors 
$\langle t_L\rangle^{-\sigma''/4}$ and $t_L^{-\sigma_2/4}$ for sufficiently large $t_L$.
}
\subsubsection{Bounds for the low frequencies}

As shown by Proposition \ref{prop1}, the assumption $(u_0,\tau_0)^\ell\in \dot{B}^{\sigma_1}_{2,\infty}$ is equivalent to the upper bound decay of solutions, i.e. 
$$
\|(u_L,\tau_L)^\ell(t)\|_{\dot{B}^{\sigma}_{2,1}}\lesssim \langle t \rangle^{-\frac{1}{2}(\sigma-\sigma_1)},~~\sigma>\sigma_1, t>0.
$$
According to  Duhamel's principle and Lemma \ref{ls1}, we have 
$$
\|\dot{\Delta}_k(\tilde{u},\tilde{\tau})\|_{L^2}\leq C\int^t_0e^{-c(t-t')2^{2k}}(\|\dot{\Delta}_kF\|_{L^2}+\|\dot{\Delta}_kH\|_{L^2})dt',\ \ k\leq k_0
$$
with
\begin{align}
F =-\mathbb{P}(u\cdot \nabla u),\label{FFF}
\end{align}
and
\begin{align}
H =-u\cdot \nabla \tau-Q(\tau,\nabla u).\label{HHH}
\end{align}
 
On the other hand, for $\sigma>\sigma_1$  there holds
\begin{align*}
&t^{\frac{\sigma-\sigma_1+\sigma'}{2}}\sum_{k\leq k_0}2^{k(\sigma-\sigma_1+\sigma')}2^{k(\sigma_1-\sigma')}e^{-c2^{2k}t}\|\dot{\Delta}_k(F,H)\|_{L^2}\\
&\lesssim \|F,H\|^\ell_{\dot{B}^{\sigma_1-\sigma'}_{2,\infty}}\sum_{k\in \mathbb{Z}}(\sqrt{t}2^k)^{\sigma-\sigma_1+\sigma'}e^{-c2^{2k}t},
\end{align*}
  where  we used the fact
\begin{equation*}
   \sup_{t\geq 0}\sum_{k\in \mathbb{Z}}t^{\frac{s}{2}}2^{ks}e^{-c2^{2k}t}\leq C_s, \qquad s>0.
\end{equation*}
Thus we get
\begin{align}\label{key ine1}
\|(\tilde{u},\tilde{\tau})(t)\|^\ell_{\dot{B}^{\sigma}_{2,1}}\lesssim \int^t_0\langle t-t'\rangle^{-\frac{1}{2}(\sigma-\sigma_1+\sigma')}\|F,H\|^\ell_{\dot{B}^{\sigma_1-\sigma'}_{2,\infty}}dt',\quad \sigma>\sigma_1, \forall \sigma' \in (0,1].
\end{align}

\begin{proposition}\label{bdd-negative-norm}
	Let $p$ satisfy \eqref{ppp} and $(u,\tau)$ be the solution to system \eqref{OB-1}. Assume further that   $(u_0,\tau_0)^\ell\in \dot{B}^{ \sigma_1}_{2,\infty} $ with $ \sigma_0 \leq \sigma_1< \frac{d}{2}-1$ and $\sigma_0=\frac{d}{2}-\frac{2d}{p}$.   Then there exists a positive constant $ {C}$, depending on  $(u_0,\tau_0)$, such that
	\begin{align}\label{negative-bosov-of-utau}
		\|(u,\tau)(t)\|_{\dot{B}^{\sigma_1}_{2,\infty}}^\ell\leq  {C},\quad \text{ for all } t\geq 0.
	\end{align}
	
\end{proposition}
\begin{proof} 
	Noticing that the system \eqref{OB-1} can be localized as follows
	\begin{align}\label{OB-11}
		\begin{cases}
			\partial_t u_k-\mathbb{P}\operatorname{div} \tau_k =F_k,\\
			\partial_t \tau_k -\Delta\tau_k -D(u_k)=H_k,
		\end{cases}
	\end{align}
{ where $F_k$ and $H_k$ are defined by
\begin{align}
		F_k=- \mathbb{P}\dot{\Delta }_k(u\cdot\nabla u),\quad H_k=- \dot{\Delta }_k(u\cdot\nabla \tau)
	-\dot{\Delta}_k
	Q(\tau, \nabla u).\label{fg}
\end{align}
}

Taking $L^2$-inner product of $\eqref{OB-11}_1$ with $u_k$ and of $\eqref{OB-11}_2$ with $\tau_k$, one gets
	\begin{align}\label{uk-11}
			\frac{1}{2}\frac{d}{dt}\|u_k\|^2_{L^2}-\langle \mathbb{P}\operatorname{div}\tau_k,u_k\rangle=\langle F_k,u_k\rangle
	\end{align}
and
\begin{align}\label{tauk-11}
			\frac{1}{2}\frac{d}{dt}\|\tau_k\|^2_{L^2}+\|\Lambda\tau_k\|^2_{L^2}-\langle D(u_k),\tau_k\rangle=\langle H_k,\tau_k\rangle.
	\end{align}
    
	Thanks to the divergence free condition $\operatorname{div}u=0$ and the symmetrical characteristic of $\tau$, it holds that
	\begin{align*}
		\langle \mathbb{P}\operatorname{div}\tau_k,u_k\rangle+\langle D(u_k),\tau_k\rangle =0,
	\end{align*}
	which combining with \eqref{uk-11} and \eqref{tauk-11} yields that
	\begin{align*}
	\frac{1}{2}\frac{d}{dt}\|(u_k,\tau_k)\|^2_{L^2}+\|\Lambda\tau_k\|^2_{L^2}\leq\|(F_k,H_k)\|_{L^2}\|(u_k,\tau_k)\|_{L^2}.
   \end{align*}
   
	It follows from  a routine procedure that
	\begin{equation}\label{utauk-11}
		\|(u,\tau)(t)\|^\ell_{\dot{B}^{\sigma_1}_{2,\infty}}
		\lesssim
		\|(u_0,\tau_0)\|^\ell_{\dot{B}^{\sigma_1}_{2,\infty}}
		+
		\int_0^t\|(F,H)(s)\|^\ell_{\dot{B}^{\sigma_1}_{2,\infty}}
		\mathrm{d}s,\ t\geq 0.
	\end{equation}
for all $t\geq 0$.

In order to  estimate $\|(F,H)\|^\ell_{L^1_t(\dot{B}^{\sigma_1}_{2,\infty})}$, it suffices to consider the three terms $\mathbb{P}(u\cdot\nabla u)$, $u\cdot\nabla\tau$ and $\mathrm{Q}(\tau,\nabla u)$.  Furthermore, we write
	\begin{align*}
		\mathbb{P}(u\cdot\nabla u)&=\mathbb{P}(u\cdot\nabla u^\ell)+\mathbb{P}(u\cdot\nabla u^h),\\ u\cdot\nabla\tau&=u\cdot\nabla\tau^\ell+u\cdot\nabla\tau^h,\\
	\mathrm{Q}(\tau,\nabla u)&=\mathrm{Q}(\tau,\nabla u^\ell)+\mathrm{Q}(\tau,\nabla u^h).
\end{align*}
In what follows we will estimate the above terms respectively.

	According to the product law on Besov spaces (see, e.g., Theorems 2.47 and 2.52 of \cite{BCD-2011}) and Bernstein’s lemma, we can  estimate $\mathbb{P}(u\cdot\nabla u^\ell)$ as
	\begin{align*}
		\|\mathbb{P}(u\cdot\nabla u^\ell)\|^\ell_{\dot{B}^{\sigma_1}_{2,\infty}}
		\lesssim&
		\|u^\ell\cdot\nabla u^\ell\|^\ell_{\dot{B}^{\sigma_1}_{2,\infty}}
		+\|u^h\cdot\nabla u^\ell\|^\ell_{\dot{B}^{\sigma_1}_{2,\infty}}
		\nonumber\\
		\lesssim&\|\nabla u^\ell\|_{\dot{B}^{\frac{d}{p}}_{p,1}}\|u^\ell\|_{\dot{B}^{\sigma_1}_{2,\infty}}+\|u^h\|_{\dot{B}^{\frac{d}{p}}_{p,1}}\|\nabla u^\ell\|_{\dot{B}^{\sigma_1}_{2,\infty}}
		\nonumber\\
		\lesssim&\big(\|u\|^\ell_{\dot{B}^{\frac{d}{2}+1}_{2,1}}+\|u\|^h_{\dot{B}^{\frac{d}{p}+1}_{p,1}}\big)\|u\|^\ell_{\dot{B}^{\sigma_1}_{2,\infty}}.
	\end{align*}
Similarly, we have
\begin{align*}
\|u\cdot\nabla\tau^\ell\|^\ell_{\dot{B}^{\sigma_1}_{2,\infty}}
	\lesssim& \|u^\ell\cdot\nabla\tau^\ell\|^\ell_{\dot{B}^{\sigma_1}_{2,\infty}}
	+\|u^h\cdot\nabla\tau^\ell\|^\ell_{\dot{B}^{\sigma_1}_{2,\infty}}
	\nonumber\\
	\lesssim&
	\|\nabla\tau^\ell\|_{\dot{B}^{\frac{d}{p}}_{p,1}}\|u^\ell\|_{\dot{B}^{\sigma_1}_{2,\infty}}+\|u^h\|_{\dot{B}^{\frac{d}{p}}_{p,1}}\|\nabla\tau^\ell\|_{\dot{B}^{\sigma_1}_{2,\infty}}
	\nonumber\\
	\lesssim&\|\tau\|^\ell_{\dot{B}^{\frac{d}{2}+1}_{2,1}}\|u\|^\ell_{\dot{B}^{\sigma_1}_{2,\infty}}+\|u\|^h_{\dot{B}^{\frac{d}{p}+1}_{p,1}}\|\tau\|^\ell_{\dot{B}^{\sigma_1}_{2,\infty}},
\end{align*}
and
\begin{align*}
	\|\mathrm{Q}(\tau,\nabla u^\ell)\|^\ell_{\dot{B}^{\sigma_1}_{2,\infty}}
	\lesssim&
	\|\mathrm{Q}(\tau^\ell,\nabla u^\ell)\|^\ell_{\dot{B}^{\sigma_1}_{2,\infty}}
	+
	\|\mathrm{Q}(\tau^h,\nabla u^\ell)\|^\ell_{\dot{B}^{\sigma_1}_{2,\infty}}
	\nonumber\\
	\lesssim&
	\|\nabla u^\ell\|_{\dot{B}^{\frac{d}{p}}_{p,1}}\|\tau^\ell\|_{\dot{B}^{\sigma_1}_{2,\infty}}
	+
	\|\tau^h\|_{\dot{B}^{\frac{d}{p}}_{p,1}}\|\nabla u^\ell\|_{\dot{B}^{\sigma_1}_{2,\infty}}
	\nonumber\\
	\lesssim&\|u\|^\ell_{\dot{B}^{\frac{d}{2}+1}_{2,1}}\|\tau\|^\ell_{\dot{B}^{\sigma_1}_{2,\infty}}
	+
	\|\tau\|^h_{\dot{B}^{\frac{d}{p}+2}_{p,1}}\|u\|^\ell_{\dot{B}^{\sigma_1}_{2,\infty}}.
\end{align*}

We now turn to  bound the remaining terms $\mathbb{P}(u\cdot\nabla u^h), u\cdot\nabla \tau^h$ and $\mathrm{Q}(\tau,\nabla u^h)$.
By employing  \eqref{equ34} in Lemma \ref{lemma12}, one obtains that for  $2\leq p\leq d$
	\begin{align*}
			\|\mathbb{P}(u\cdot\nabla u^h)\|^\ell_{\dot{B}^{\sigma_1}_{2,\infty}}&\lesssim
				\|u\cdot\nabla u^h\|^\ell_{\dot{B}^{\sigma_1}_{2,\infty}}
				\lesssim(\|u^\ell\|_{\dot{B}^{\frac{d}{2}-1}_{2,1}}+\|u^h\|_{\dot{B}^{\frac{d}{p}-1}_{p,1}})\|\nabla u^h\|_{\dot{B}^{\frac{d}{p}-1}_{p,1}}\\
			&\lesssim(\|u\|^\ell_{\dot{B}^{\frac{d}{2}-1}_{2,1}}+\|u\|^h_{\dot{B}^{\frac{d}{p}+1}_{p,1}})\|u\|^h_{\dot{B}^{\frac{d}{p}+1}_{p,1}},
	\\
			\|u\cdot\nabla \tau^h\|^\ell_{\dot{B}^{\sigma_1}_{2,\infty}}&\lesssim(\|u^\ell\|_{\dot{B}^{\frac{d}{2}-1}_{2,1}}+\|u^h\|_{\dot{B}^{\frac{d}{p}-1}_{p,1}})\|\nabla \tau^h\|_{\dot{B}^{\frac{d}{p}-1}_{p,1}}\\
			&\lesssim(\|u\|^\ell_{\dot{B}^{\frac{d}{2}-1}_{2,1}}+\|u\|^h_{\dot{B}^{\frac{d}{p}+1}_{p,1}})\|\tau\|^h_{\dot{B}^{\frac{d}{p}+2}_{p,1}},
\\
			\|\mathrm{Q}(\tau,\nabla u^h)\|^\ell_{\dot{B}^{\sigma_1}_{2,\infty}}&\lesssim(\|\tau^\ell\|_{\dot{B}^{\frac{d}{2}-1}_{2,1}}+\|\tau^h\|_{\dot{B}^{\frac{d}{p}-1}_{p,1}})\|\nabla u^h\|_{\dot{B}^{\frac{d}{p}-1}_{p,1}}\\
			&\lesssim(\|\tau\|^\ell_{\dot{B}^{\frac{d}{2}-1}_{2,1}}+\|\tau\|^h_{\dot{B}^{\frac{d}{p}}_{p,1}})\|u\|^h_{\dot{B}^{\frac{d}{p}+1}_{p,1}}.
		\end{align*}
	While by using \eqref{equ56} in Lemma \ref{lemma12},  one gets that for $p\geq d$
	\begin{align*}
			\|\mathbb{P}(u\cdot\nabla u^h)\|^\ell_{\dot{B}^{\sigma_1}_{2,\infty}}&  \lesssim(\|u^\ell\|_{\dot{B}^{\frac{d}{2}-1}_{2,1}}+\|u^h\|_{\dot{B}^{\frac{d}{p}}_{p,1}})\|\nabla u^h\|_{\dot{B}^{\frac{d}{p}-1}_{p,1}}\\
			&\lesssim(\|u\|^\ell_{\dot{B}^{\frac{d}{2}-1}_{2,1}}+\|u\|^h_{\dot{B}^{\frac{d}{p}+1}_{p,1}})\|u\|^h_{\dot{B}^{\frac{d}{p}+1}_{p,1}},\\
			\|u\cdot\nabla \tau^h\|^\ell_{\dot{B}^{\sigma_1}_{2,\infty}}&\lesssim(\|u^\ell\|_{\dot{B}^{\frac{d}{2}-1}_{2,1}}+\|u^h\|_{\dot{B}^{\frac{d}{p}}_{p,1}})\|\nabla \tau^h\|_{\dot{B}^{\frac{d}{p}-1}_{p,1}}\\
			&\lesssim(\|u\|^\ell_{\dot{B}^{\frac{d}{2}-1}_{2,1}}+\|u\|^h_{\dot{B}^{\frac{d}{p}+1}_{p,1}})\|\tau\|^h_{\dot{B}^{\frac{d}{p}+2}_{p,1}},
	\\
			\|\mathrm{Q}(\tau,\nabla u^h)\|^\ell_{\dot{B}^{\sigma_1}_{2,\infty}}&\lesssim(\|\tau^\ell\|_{\dot{B}^{\frac{d}{2}-1}_{2,1}}+\|\tau^h\|_{\dot{B}^{\frac{d}{p}}_{p,1}})\|\nabla u^h\|_{\dot{B}^{\frac{d}{p}-1}_{p,1}}\\
			&\lesssim(\|\tau\|^\ell_{\dot{B}^{\frac{d}{2}-1}_{2,1}}+\|\tau\|^h_{\dot{B}^{\frac{d}{p}}_{p,1}})\|u\|^h_{\dot{B}^{\frac{d}{p}+1}_{p,1}}.
		\end{align*}
Combining the bounds above, it follows that
\begin{align}\label{bound-of-fg}
		\|(F,H)(t)\|^\ell_{\dot{B}^{\sigma_1}_{2,\infty}}&
		\lesssim D_{p,1}(t) \|(u,\tau)(t)\|^\ell_{\dot{B}^{\sigma_1}_{2,\infty}}+
		D_{p,2}(t),
\end{align}
with
\begin{align*}
	D_{p,1}(t)&=\|(u,\tau)\|^\ell_{\dot{B}^{\frac{d}{2}+1}_{2,1}}+\|u\|^h_{\dot{B}^{\frac{d}{p}+1}_{p,1}}+\|\tau\|^h_{\dot{B}^{\frac{d}{p}+2}_{p,1}},
\\
D_{p,2}(t)&=(\|(u,\tau)\|^\ell_{\dot{B}^{\frac{d}{2}-1}_{2,1}}+\|u\|^h_{\dot{B}^{\frac{d}{p}+1}_{p,1}}+\|\tau\|^h_{\dot{B}^{\frac{d}{p}}_{p,1}})(\|u\|^h_{\dot{B}^{\frac{d}{p}+1}_{p,1}}+\|\tau\|^h_{\dot{B}^{\frac{d}{p}+2}_{p,1}}).
\end{align*}

Inserting \eqref{bound-of-fg} into \eqref{utauk-11}, we arrive at 
\begin{align}\label{negative-bosov-of-utau-1}
	\|(u,\tau)(t)\|^\ell_{\dot{B}^{\sigma_1}_{2,\infty}}\lesssim&\|(u_0,\tau_0)\|^\ell_{\dot{B}^{\sigma_1}_{2,\infty}}+\int_0^tD_{p,1}(s)\|(u,\tau)(s)\|^\ell_{\dot{B}^{\sigma_1}_{2,\infty}}\mathrm{d} s
+\int_{0}^{t}
	D_{p,2}(s)\mathrm{d}s.
\end{align}
Therefore, with the help of \eqref{E-t}, one deduces 
\begin{equation*}
	\int^t_0(D_{p,1}(s)+D_{p,2}(s))\mathrm{d}s\leq E(t)+E^2(t)\leq CE_0.
\end{equation*}

Thus we complete the proof of Proposition \ref{bdd-negative-norm}.
\end{proof}

\begin{lemma}\label{cl1}
Let $p$ satisfy \eqref{ppp} and $\sigma_0\leq \sigma_1<\frac{d}{2}-1.$ It holds that
\begin{align*}
 \|(\tilde{u},\tilde{\tau})(t)\|^\ell_{\dot{B}^{\sigma}_{2,1}} 
&\lesssim\int^t_0\langle t-t'\rangle^{-\frac{1}{2}(\sigma-\sigma_1+1)}\|F,u\cdot \nabla \tau\|^\ell_{\dot{B}^{\sigma_1-1}_{2,\infty}}dt'\\
&\ \ \ \ +\int^t_0\langle t-t'\rangle^{-\frac{1}{2}(\sigma-\sigma_1+\sigma')}\|Q(\tau\cdot \nabla u)\|^\ell_{\dot{B}^{\sigma_1-\sigma'}_{2,\infty}}dt'\\
&\lesssim \langle t \rangle^{-\frac{1}{2}(\sigma-\sigma_1+\sigma_2)}E_0,
\end{align*}
for $0< t\leq 2$ and  $\sigma_2\in (0,1]$. 
\end{lemma}
\begin{proof}

Thanks to   \eqref{key ine1} we have
\begin{align*}
\|(\tilde{u},\tilde{\tau})(t)\|^\ell_{\dot{B}^{\sigma}_{2,1}}\lesssim \int^t_0\langle t-t'\rangle^{-\frac{1}{2}(\sigma-\sigma_1+\sigma')}\|F,H\|^\ell_{\dot{B}^{\sigma_1-\sigma'}_{2,\infty}}dt',~~\sigma>\sigma_1, \sigma'\in(0,1].
\end{align*}

Next we divide into two steps to proceed.

\underline{\bf Estimate on the term $F$.}
In this part
 we will show that
\begin{align}\label{6452}
 \int^t_0\langle t-t'\rangle^{-\frac{1}{2}(\sigma-\sigma_1+\sigma')}\|F\|^\ell_{\dot{B}^{\sigma_1-\sigma'}_{2,\infty}}dt' \lesssim \langle t \rangle^{-\frac{1}{2}(\sigma-\sigma_1+\sigma_2)}E_0.
\end{align}
Now, choosing $\sigma'=1$ and noticing $\frac{d}{2}-1\leq \frac{d}{p}$, we have
\begin{align*}
\|\mathbb{P}(u\cdot \nabla u)\|^\ell_{\dot{B}^{\sigma_1-1}_{2,\infty}}&\leq  \|u\otimes u^\ell\|^\ell_{\dot{B}^{\sigma_1}_{2,\infty}}
+\|u\otimes u^h\|^\ell_{\dot{B}^{\sigma_1}_{2,\infty}}.
\end{align*}
Using the inequality \eqref{A6} of Corollary \ref{app5} , we get
\begin{align*}
\|u\otimes u^\ell\|^\ell_{\dot{B}^{\sigma_1}_{2,\infty}}&\leq \|u\otimes u^\ell\|^\ell_{\dot{B}^{\sigma_1+\frac{d}{p}-\frac{d}{2}}_{2,\infty}}\\
&\leq \|u\|_{\dot{B}^{\frac{d}{p}-1}_{p,1}}\|u^\ell\|_{\dot{B}^{\sigma_1+\frac{d}{p}-\frac{d}{2}+1}_{2,\infty}}\\
&\leq \|u\|_{\dot{B}^{\frac{d}{p}-1}_{p,1}}\|u^\ell\|_{\dot{B}^{\sigma_1}_{2,\infty}}\\
&\lesssim E_0\|u^\ell\|_{L^\infty_t(\dot{B}^{\sigma_1}_{2,\infty})}.
\end{align*}
 
If $2\leq p<d$, it follows from Lemma \ref{lemma12} with $s=\frac{d}{p}-1$ that
\begin{align*}
\|u\otimes u^h\|^\ell_{\dot{B}^{\sigma_1}_{2,\infty}}\leq (\|u\|_{\dot{B}^{\frac{d}{p}-1}_{p,1}}+\|u^\ell\|_{L^{p*}})\|u^h\|_{\dot{B}^{1-\frac{d}{p}}_{p,1}} \lesssim E_0^2.
\end{align*}
If $p=d$, one gets 
\begin{align*}
\|u\otimes u^h\|^\ell_{\dot{B}^{\sigma_1}_{2,\infty}}\leq \|u\otimes u^h\|^\ell_{\dot{B}^{\sigma_0}_{2,\infty}}\lesssim  \|u\otimes u^h\|^\ell_{L^{\frac{d}{2}}}\lesssim \|u\|_{\dot{B}^{0}_{d,1}}\|u^h\|_{\dot{B}^{0}_{d,1}},
\end{align*}
which together with $\dot{B}^{\frac{d}{p}}_{2,1}\hookrightarrow L^{p*}$($\frac{1}{p}+\frac{1}{p_*}=\frac12$) and $\dot{B}^{\frac{d}{p}-1}_{p,1}\hookrightarrow \dot{B}^{0}_{d,1}$ gives
$$
\|u\otimes u^h\|^\ell_{\dot{B}^{\sigma_1}_{2,\infty}}\lesssim (\|u^\ell\|_{\dot{B}^{\frac{d}{2}-1}_{2,1}}+\|u^h\|_{\dot{B}^{\frac{d}{p}-1}_{p,1}})\|u^h\|_{\dot{B}^{\frac{d}{p}-1}_{p,1}}\lesssim E_0^2.
$$
If $p>d$, applying the Lemma \ref{lemma12} with $s=1-\frac{d}{p}$ again implies that
$$
\|u\otimes u^h\|^\ell_{\dot{B}^{\sigma_1}_{2,\infty}}\leq (\|u\|_{\dot{B}^{1-\frac{d}{p}}_{p,1}}+\|u^\ell\|_{L^{p*}})\|u^h\|_{\dot{B}^{\frac{d}{p}-1}_{p,1}}.
$$
By using the embedding $\dot{B}^{1+\sigma_0}_{2,1}\hookrightarrow \dot{B}^{1-\frac{d}{p}}_{p,1}$ with  $\frac{d}{2}-1<1+\sigma_0$ and  $p>d$, we obtain
$$
\|u\otimes u^h\|^\ell_{\dot{B}^{\sigma_1}_{2,\infty}}\lesssim (\|u^\ell\|_{\dot{B}^{\frac{d}{2}-1}_{2,1}}+\|u^h\|_{\dot{B}^{\frac{d}{p}}_{p,1}})\|u^h\|_{\dot{B}^{\frac{d}{p}-1}_{p,1}}\lesssim E_0^2.
$$
Noting that $0<t\leq 2$, we then get the desired estimate \eqref{6452}  by using \eqref{negative-bosov-of-utau}.
\newline
\underline{\bf Estimate on the term $H$.} 
In this part, we will show that
\begin{align}\label{6453}
 \int^t_0\langle t-t'\rangle^{-\frac{1}{2}(\sigma-\sigma_1+\sigma')}\|H\|^\ell_{\dot{B}^{\sigma_1-\sigma'}_{2,\infty}}dt' \lesssim \langle t \rangle^{-\frac{1}{2}(\sigma-\sigma_1+\sigma_2)}E_0.
\end{align}
Recalling the definition of $H$ in \eqref{HHH},
and noting that $-u\cdot \nabla \tau=-\partial_j(u_j\tau_{k,l}),$ 
and using a similar argument as the estimate of $F$, we have
\begin{align}
 \|(u\cdot \nabla \tau^\ell)\|^\ell_{\dot{B}^{\sigma_1-1}_{2,\infty}} 
 \leq \|u\|_{\dot{B}^{\frac{d}{p}-1}_{p,1}}\|\tau^\ell\|_{\dot{B}^{\sigma_1}_{2,\infty}}
\lesssim E_0\|\tau^\ell\|_{L^\infty_t(\dot{B}^{\sigma_1}_{2,\infty})},\label{utau-1}
\end{align}
and
\begin{equation}\label{utau-2}
\|u \cdot \tau^h\|^\ell_{\dot{B}^{\sigma_1}_{2,\infty}}\lesssim (\|u^\ell\|_{\dot{B}^{\frac{d}{2}-1}_{2,1}}+\|u^h\|_{\dot{B}^{\frac{d}{p}}_{p,1}})\|\tau^h\|_{\dot{B}^{\frac{d}{p}-1}_{p,1}}\lesssim E_0^2.
\end{equation}
Together \eqref{utau-1} with \eqref{utau-2} gives
\begin{align}
 \|(u\cdot \nabla \tau )\|^\ell_{\dot{B}^{\sigma_1-1}_{2,\infty}} &
\lesssim E_0^2+ E_0\|\tau^\ell\|_{L^\infty_t(\dot{B}^{\sigma_1}_{2,\infty})}.\label{utau-3}
\end{align}
Moreover, we tackle the term $Q(\tau,\nabla u)$, note
\begin{align*}
&\|Q(\tau, \nabla u)\|^\ell_{\dot{B}^{\sigma_1-\sigma'}_{2,\infty}}\\&\leq \|\tau \nabla u\|^\ell_{\dot{B}^{\sigma_1-\sigma'}_{2,\infty}}\\
&\leq \|\tau^\ell  \nabla u^\ell\|^\ell_{\dot{B}^{\sigma_1-\sigma'}_{2,\infty}}+\|\tau^\ell  \nabla u^h\|^\ell_{\dot{B}^{\sigma_1-\sigma'}_{2,\infty}}
+\|(\tau^h    \nabla u )\|^\ell_{\dot{B}^{\sigma_1-\sigma'}_{2,\infty}} .
\end{align*}
{
In the following part, we will estimate the above three terms respectively. By using the estimate \eqref{ap3} of Lemma \ref{app1}, one gets
\begin{align}
\|\tau^\ell\nabla u^\ell\|^\ell_{\dot{B}^{\sigma_1-\sigma'}_{2,\infty}}&\lesssim\|\tau^\ell\|_{\dot{B}^{\frac{d}{2}-1}_{2,1}}\|\nabla u^\ell\|_{\dot{B}^{\sigma_1-\sigma'+1}_{2,\infty}}\nonumber\\
&\lesssim E^2_0,\label{tauu-1}
\end{align}
where  $\sigma_1-\sigma'+1\leq \frac{d}{2}$, $\frac{d}{2}+\sigma_1-\sigma'>0$.

Next we treat the term $\|\tau^\ell\nabla u^h\|^\ell_{\dot{B}^{\sigma_1-\sigma'}_{2,\infty}}$.
If $p>2$,
using the \eqref{A.3} in Corollary \ref{app3} implies
\begin{align}\label{clfw1}
\|\tau^\ell\nabla u^h\|^\ell_{\dot{B}^{\sigma_1-\sigma'}_{2,\infty}}&\lesssim \|\tau^\ell\nabla u^h\|_{\dot{B}^{\sigma_1+\frac{d}{p}-\frac{d}{2}}_{2,\infty}}\nonumber\\
&\lesssim \|\tau^\ell\|_{\dot{B}^{\frac{d}{p}}_{2,1}}\|\nabla u^h\|_{\dot{B}^{\sigma_1+\frac{d}{p}-\frac{d}{2}}_{p,1}}\nonumber\\
&\lesssim \|\tau^\ell\|_{\dot{B}^{\frac{d}{2}-1}_{2,1}}\|\nabla u^h\|_{\dot{B}^{\frac{d}{p}-1}_{p,1}}\nonumber\\
&\lesssim E^2_0,
\end{align}
where $\frac{d}{2}-1\leq \frac{d}{p}$, $0<\sigma'\leq  \frac{d}{2}-\frac{d}{p}$ and  $\sigma'$ is small enough.
  \quad \newline
If $p=2$, 
note that $\sigma_1>\sigma_0$ and $\sigma'$ is small enough, using \eqref{A.4} of Corollary \ref{app3} to give
\begin{align}\label{clfw2}
&\|\tau^\ell\nabla u^h\|^\ell_{\dot{B}^{\sigma_1-\sigma'}_{2,\infty}}\lesssim \|\tau^\ell\nabla u^h\|^\ell_{\dot{B}^{\sigma_0}_{2,\infty}}\nonumber\\
&\lesssim \|\tau^\ell\|_{\dot{B}^{\frac{d}{2}-1}_{2,1}}\|\nabla u^h\|_{\dot{B}^{1-\frac{d}{2}}_{2,1}}\lesssim \|\tau^\ell\|_{\dot{B}^{\frac{d}{2}-1}_{2,1}}\|\nabla u^h\|_{\dot{B}^{\frac{d}{p}}_{p,1}}\lesssim E_0^2.
\end{align}
Combining \eqref{clfw2} with \eqref{clfw1} yields 
\begin{equation}\label{4257}
    \|\tau^\ell\nabla u^h\|^\ell_{\dot{B}^{\sigma_1-\sigma'}_{2,\infty}} \lesssim E_0^2.
\end{equation}

Finally, we turn to deal with the term $\|\tau^h\nabla u \|^\ell_{\dot{B}^{\sigma_1-\sigma'}_{2,\infty}}$. 
If $2\leq p<d$, using Lemma \ref{lemma12} with $s=\frac{d}{p}-1$ yields
\begin{align}\label{4.28}
\|(\tau^h  \nabla u)\|^\ell_{\dot{B}^{\sigma_1-\sigma'}_{2,\infty}}\leq\|(\tau^h  \nabla u)\|^\ell_{\dot{B}^{\sigma_0}_{2,\infty}}\leq (\|\nabla u\|_{\dot{B}^{\frac{d}{p}-1}_{p,1}}+\|(\nabla u)^\ell\|_{L^{p*}})\|\tau^h\|_{\dot{B}^{1-\frac{d}{p}}_{p,1}} \lesssim E_0^2,
\end{align}
where $\sigma_1>\sigma_0$.
\quad \newline
If $p=d$, one gets
\begin{align}\label{4.29}
\|\tau^h \nabla u\|^\ell_{\dot{B}^{\sigma_1-\sigma'}_{2,\infty}}
\leq \|\tau^h \nabla u\|^\ell_{\dot{B}^{\sigma_0}_{2,\infty}}\lesssim  \|\tau^h \nabla u\|^\ell_{L^{\frac{d}{2}}}\lesssim \|\nabla u\|_{\dot{B}^{0}_{d,1}}\|\tau^h\|_{\dot{B}^{0}_{d,1}},
\end{align}
which together with $\dot{B}^{\frac{d}{p}}_{2,1}\hookrightarrow L^{p*}$ and $\dot{B}^{\frac{d}{p}-1}_{p,1}\hookrightarrow \dot{B}^{0}_{d,1}$ yields
\begin{align}\label{4.30}
\|\tau^h \nabla u\|^\ell_{\dot{B}^{\sigma_1-\sigma'}_{2,\infty}}\lesssim (\|u^\ell\|_{\dot{B}^{\frac{d}{2}}_{2,1}}+\|u^h\|_{\dot{B}^{\frac{d}{p}}_{p,1}})\|\tau^h\|_{\dot{B}^{\frac{d}{p}-1}_{p,1}}
\lesssim E^2_0.
\end{align}
If $p>d$, applying Lemma \ref{app2} with $s=1-\frac{d}{p}$ again implies that
\begin{align}\label{4.31}
\|\tau^h \nabla u\|^\ell_{\dot{B}^{\sigma_0}_{2,\infty}}\leq (\|\nabla u\|_{\dot{B}^{1-\frac{d}{p}}_{p,1}}+\|(\nabla u)^\ell\|_{L^{p*}})\|\tau^h\|_{\dot{B}^{\frac{d}{p}-1}_{p,1}}.
\end{align}
By using the embedding $\dot{B}^{1+\sigma_0}_{2,1}\hookrightarrow \dot{B}^{1-\frac{d}{p}}_{p,1}$ with $\frac{d}{2}-1<1+\sigma_0$  and noting that  $p>d$, we obtain
\begin{align}\label{4.32}
\|\tau^h  \nabla u\|^\ell_{\dot{B}^{\sigma_1-\sigma'}_{2,\infty}}&\lesssim (\|u^\ell\|_{\dot{B}^{\frac{d}{2}}_{2,1}}+\|u^h\|_{\dot{B}^{\frac{d}{p}+1}_{p,1}})\|\tau^h\|_{\dot{B}^{\frac{d}{p}-1}_{p,1}}
\lesssim E^2_0.
\end{align}
Combining \eqref{4.28}-\eqref{4.32} implies
\begin{equation}\label{49677}
    \|\tau^h\nabla u \|^\ell_{\dot{B}^{\sigma_1-\sigma'}_{2,\infty}}\lesssim E_0^2.
\end{equation}
Then,  collecting \eqref{utau-3}, \eqref{tauu-1},\eqref{4257}, \eqref{49677} with  \eqref{negative-bosov-of-utau} gives \eqref{6453}.

Finally Lemma \ref{cl1} follows  immediately from  \eqref{6452} and \eqref{6453}} with the help of $\langle t\rangle \approx 1$ and $\langle t-t'\rangle\approx 1$ for $0< t\leq 2$.
\end{proof}

\begin{lemma}\label{cl3}
Let $p$ satisfy \eqref{ppp} and $\sigma_0\leq \sigma_1<\frac{d}{2}-1.$ It holds that
\begin{align*}
\|(\tilde{u},\tilde{\tau})(t)\|^\ell_{\dot{B}^{\sigma}_{2,1}}\leq&\int^t_0\langle t-t'\rangle^{-\frac{1}{2}(\sigma-\sigma_1+\sigma')}\|F,H\|^\ell_{\dot{B}^{\sigma_1-\sigma'}_{2,\infty}}dt'\\
\leq& \langle t \rangle^{-\frac{1}{2}(\sigma-\sigma_1+\min\{\sigma'_2,\sigma''\})}(E_0 +1)\\
&+\langle t \rangle^{-\frac{1}{2}(\sigma-\sigma_1+\min\{\sigma'_2,\sigma''\})}(E_0+\|(u^\ell_L,\tau^\ell_L)\|_{\tilde L^\infty_t(\dot{B}^{\sigma_1+\sigma_2}_{2,\infty})})\tilde{D}_p(t)  \\&+\langle t \rangle^{-\frac{1}{2}(\sigma-\sigma_1+\sigma'')}\langle t\rangle^{-\frac{\min\{\sigma'_2,\sigma''\}}{2}}(\tilde{D}_p(t)+\tilde{D}^2_p(t)  ),
\end{align*}
for $t>2$, $\sigma_1<\sigma\leq \frac{d}{2}-1$ and
\begin{equation*}
\sigma'_2=
		\begin{cases}
			1,\qquad\qquad\qquad \qquad  ~~~\text{if}~\ ~~\sigma_1<\frac{d}{2}-2,\\
			1-,\ \qquad\quad\quad\quad   \quad~\text{if}~\ ~~~~\sigma_1=\frac{d}{2}-2,\\
			\frac{d}{2}-1-\sigma_1,\qquad\qquad~~\text{if}~~\ ~\frac{d}{2}-2<\sigma_1<\frac{d}{2}-1,
		\end{cases}
	\end{equation*}
    for some sufficient small $\sigma''>0$.
\end{lemma}
\begin{proof}
For $t>2$ and $\sigma'\in(0,1]$, we write
\begin{align*}
 &\int^t_0\langle t-t'\rangle^{-\frac{1}{2}(\sigma-\sigma_1+\sigma')}\|F,H\|^\ell_{\dot{B}^{\sigma_1-\sigma'}_{2,\infty}}dt'\\
&=\int^1_0\langle t-t'\rangle^{-\frac{1}{2}(\sigma-\sigma_1+\sigma')}\|F,H\|^\ell_{\dot{B}^{\sigma_1-\sigma'}_{2,\infty}}dt'+\int^t_1\langle t-t'\rangle^{-\frac{1}{2}(\sigma-\sigma_1+\sigma')}\|F,H\|^\ell_{\dot{B}^{\sigma_1-\sigma'}_{2,\infty}}dt'.
\end{align*}
We here only handle the integral in $[1,t]$ and the case of integral in  $[0,1]$ follows from the estimates in Lemma \ref{cl1}. 

Noticing that $F+H=F_L+\tilde{F}_L+\tilde{F}$ 
with
\begin{align}
F_L=&-\mathbb{P}(u_L\cdot \nabla u_L)-u_L\cdot \nabla \tau_L-Q(\tau_L,\nabla u_L),
\\
    \tilde{F}_L=&-\mathbb{P}(\tilde{u}\cdot \nabla u_L)-\tilde{u}\cdot \nabla \tau_L-Q(\tilde{\tau},\nabla u_L),\label{FF-1}\\
    \tilde{F}=&-\mathbb{P}(u\cdot \nabla\tilde{ u})-u\cdot \nabla \tilde{\tau}-Q(\tau,\nabla \tilde{u}).\label{FF-2}
\end{align}
Then  we will bound the following  three terms respectively,
\begin{equation}
    \begin{aligned}
        \mathcal{I}_1=&\int^t_1\langle t-t'\rangle^{-\frac{1}{2}(\sigma-\sigma_1+\sigma')}\|F_L\|^\ell_{\dot{B}^{\sigma_1-\sigma'}_{2,\infty}}dt',\\
        \mathcal{I}_2=&\int^t_1\langle t-t'\rangle^{-\frac{1}{2}(\sigma-\sigma_1+\sigma')}\|\tilde F_L\|^\ell_{\dot{B}^{\sigma_1-\sigma'}_{2,\infty}}dt',\\
        \mathcal{I}_3=&\int^t_1\langle t-t'\rangle^{-\frac{1}{2}(\sigma-\sigma_1+\sigma')}\|\tilde F \|^\ell_{\dot{B}^{\sigma_1-\sigma'}_{2,\infty}}dt'.
    \end{aligned}
\end{equation}
\underline{\bf Estimate of $\mathcal{I}_1$}.
Let  $\sigma'=1$, 
following a similar estimate as \eqref{4.28}-\eqref{4.32}, we have
\begin{align*}
\|\mathbb{P}(u_L\otimes u_L)\|^\ell_{\dot{B}^{\sigma_1}_{2,\infty}}&\leq \|\mathbb{P}(u^\ell_L\otimes  u^\ell_L)\|^\ell_{\dot{B}^{\sigma_1}_{2,\infty}}
+\|\mathbb{P}(u^h_L\otimes  u^\ell_L)\|^\ell_{\dot{B}^{\sigma_1}_{2,\infty}}+\|\mathbb{P}(u_L\otimes  u^h_L)\|^\ell_{\dot{B}^{\sigma_1}_{2,\infty}}\\
&\leq \|\mathbb{P}(u^\ell_L\otimes u^\ell_L)\|^\ell_{\dot{B}^{\sigma_1}_{2,\infty}}
+\|\mathbb{P}(u^h_L\otimes u^\ell_L)\|^\ell_{\dot{B}^{\frac{d}{p}-\frac{d}{2}+\sigma_1}_{2,\infty}}+\|\mathbb{P}(u_L\otimes  u^h_L)\|^\ell_{\dot{B}^{\sigma_1}_{2,\infty}}\\
&\leq \|u^\ell_L\|_{\dot{B}^{\frac{d}{2}-1}_{2,1}}\|u^\ell_L\|_{\dot{B}^{\sigma_1+1}_{2,1}}+\|u^h_L\|_{\dot{B}^{\frac{d}{p}-1}_{p,1}}\|u^\ell_L\|_{\dot{B}^{\frac{d}{2}-1}_{2,1}}\\
&\ \ \ +
(\|u^\ell_L\|_{\dot{B}^{\frac{d}{2}-1}_{2,1}}+\|u^h_L\|_{\dot{B}^{\frac{d}{p}}_{p,1}})\|u^h_L\|_{\dot{B}^{\frac{d}{p}-1}_{p,1}}\\
&\lesssim (1+E_0)\langle t \rangle^{-\frac{1}{2}(\frac{d}{2}-\sigma_1)},~~~t>1,
\end{align*}
where we have used the fact in \eqref{ap3}, \eqref{A.3}, Proposition \ref{prop1} and the exponential decay of high frequencies of $(u^h_L,\tau^h_L)$. 
\par
Similarly, we have
\begin{align*}
\|u_L\cdot\nabla\tau_L\|^\ell_{\dot{B}^{\sigma_1-1}_{2,\infty}}&=\|u_L\cdot  \tau_L\|^\ell_{\dot{B}^{\sigma_1}_{2,\infty}}\\&\lesssim \|u^\ell_L\|_{\dot{B}^{\frac{d}{2}-1}_{2,1}}\|\tau^\ell_L\|_{\dot{B}^{\sigma_1+1}_{2,1}}
+\|u^h_L\|_{\dot{B}^{\frac{d}{p}-1}_{p,1}}\|\tau^\ell_L\|_{\dot{B}^{\frac{d}{2}-1}_{2,1}}\\
&\ \ \ \ \ +
(\|u^\ell_L\|_{\dot{B}^{\frac{d}{2}-1}_{2,1}}+\|u^h_L\|_{\dot{B}^{\frac{d}{p}}_{p,1}})\|\tau^h_L\|_{\dot{B}^{\frac{d}{p}-1}_{p,1}}\\
&\lesssim (1+E_0)\langle t \rangle^{-\frac{1}{2}(\frac{d}{2}-\sigma_1)},~~~t>1.
\end{align*}
{
Next, we tackle the term  $Q(\tau_L,\nabla u_L)$. 
Using \eqref{ap3} of Lemma \ref{app1} yields
\begin{align*} 
        \|\tau^\ell_L\nabla u^\ell_L\|^\ell_{\dot{B}^{\sigma_1-\sigma'}_{2,\infty}}\lesssim\|\tau^\ell_L\|_{\dot{B}^{\frac{d}{2}-\frac{\sigma'}{2}}_{2,1}}\|\nabla u^\ell_L\|_{\dot{B}^{\sigma_1-\frac{\sigma'}{2}}_{2,\infty}}
\lesssim \langle t \rangle^{-\frac{1}{2}(\frac{d}{2}+1-\sigma_1-\sigma')},
\end{align*}
where $\sigma_1-\frac{\sigma'}{2}\leq \frac{d}{2}$
and $\frac{d}{2}+\sigma_1-\sigma'\geq 0$.
\par
We then handle the term
$\|\tau^\ell_L\nabla u^h_L\|^\ell_{\dot{B}^{\sigma_1-\sigma'}_{2,\infty}}$.
If $p>2$, using the \eqref{A.3} of Corollary \ref{app3}, we get
\begin{align*}
&\|\tau^\ell_L\nabla u^h_L\|^\ell_{\dot{B}^{\sigma_1-\sigma'}_{2,\infty}}\lesssim \|\tau^\ell_L\nabla u^h_L\|_{\dot{B}^{\sigma_1+\frac{d}{p}-\frac{d}{2}}_{2,\infty}} \lesssim \|\tau^\ell_L\|_{\dot{B}^{\frac{d}{2}-1}_{2,1}}\|\nabla u^h_L\|_{\dot{B}^{\frac{d}{p}-1}_{p,1}}\lesssim (1+E_0)\langle t \rangle^{-\frac{1}{2}(\frac{d}{2}-\sigma_1)},
\end{align*}
where    $0<\sigma'\leq  \frac{d}{2}-\frac{d}{p}$ and  $ \sigma'$ is sufficiently small. \quad \newline 
If  $p=2$, noting that $\sigma_1>\sigma_0$,  one has
\begin{align*}
\|\tau^\ell_L\nabla u^h_L\|^\ell_{\dot{B}^{\sigma_1-\sigma'}_{2,\infty}}\lesssim \|\tau^\ell_L\nabla u^h_L\|_{\dot{B}^{\sigma_0}_{2,\infty}}
\lesssim \|\tau^\ell_L\|_{\dot{B}^{\frac{d}{2}-1}_{2,1}}\|\nabla u^h_L\|_{\dot{B}^{1-\frac{d}{2}}_{2,1}}\lesssim (1+E_0)\langle t \rangle^{-\frac{1}{2}(\frac{d}{2}-\sigma_1)},
\end{align*}
where we have used  \eqref{A.4} and the  smallness of $ \sigma'$.
\par
We now turn to estimate
$
\|\tau^h_L\nabla u_L\|^\ell_{\dot{B}^{\sigma_1-\sigma'}_{2,\infty}}.
$
If $2\leq p<d$ and let $ \sigma'=2\sigma''$ is small enough,  using Lemma \ref{app2} with $s=\frac{d}{p}-1$ yields
\begin{align*}
\|(\tau^h_L  \nabla u_L)\|^\ell_{\dot{B}^{\sigma_1-\sigma'}_{2,\infty}}&\leq\|(\tau^h_L  \nabla u_L)\|^\ell_{\dot{B}^{\sigma_0}_{2,\infty}}\\&\leq (\|\nabla u_L\|_{\dot{B}^{\frac{d}{p}-1}_{p,1}}+\|(\nabla u_L)^\ell\|_{L^{p*}})\|\tau^h_L\|_{\dot{B}^{1-\frac{d}{p}}_{p,1}},\\&\lesssim
 (1+E_0)\langle t \rangle^{-\frac{1}{2}(\frac{d}{2}-\sigma_1)}
\end{align*}
where $\sigma_1>\sigma_0$ and $1-\frac{d}{p}\leq \frac{d}{p}+2$.
\newline 
If $p=d$, one get 
\begin{align*}
&\|(\tau^h_L  \nabla u_L)\|^\ell_{\dot{B}^{\sigma_1-\sigma'}_{2,\infty}} 
 \leq \|(\tau^h_L  \nabla u_L)\|^\ell_{\dot{B}^{\sigma_0}_{2,\infty}}\lesssim  \|(\tau^h_L  \nabla u_L)\|^\ell_{L^{\frac{d}{2}}}\lesssim \|\nabla u_L\|_{\dot{B}^{0}_{d,1}}\|\tau^h_L\|_{\dot{B}^{0}_{d,1}}.
\end{align*}

Employing the embedding $\dot{B}^{\frac{d}{p}}_{2,1}\hookrightarrow L^{p*}$  gives
$$
\|(\tau^h_L  \nabla u_L)\|^\ell_{\dot{B}^{\sigma_1}_{2,\infty}}\lesssim (\|u^\ell_L\|_{\dot{B}^{\frac{d}{2}}_{2,1}}+\|u^h_L\|_{\dot{B}^{\frac{d}{p}}_{p,1}})\|\tau^h_L\|_{\dot{B}^{\frac{d}{p}-1}_{p,1}}.
$$
If $p>d$, applying Lemma \ref{app2} with $s=1-\frac{d}{p}$ again implies that
$$
\|(\tau^h_L  \nabla u_L)\|^\ell_{\dot{B}^{\sigma_0}_{2,\infty}}\leq (\|\nabla u_L\|_{\dot{B}^{1-\frac{d}{p}}_{p,1}}+\|(\nabla u_L)^\ell\|_{L^{p*}})\|\tau^h_L\|_{\dot{B}^{\frac{d}{p}-1}_{p,1}},
$$
by using the embedding $\dot{B}^{1+\sigma_0}_{2,1}\hookrightarrow \dot{B}^{1-\frac{d}{p}}_{p,1}$ with $\frac{d}{2}-1<1+\sigma_0$ and noting that $p>d$, we obtain
$$
\|(\tau^h_L  \nabla u_L)\|^\ell_{\dot{B}^{\sigma_1}_{2,\infty}}\lesssim (\|u^\ell_L\|_{\dot{B}^{\frac{d}{2}}_{2,1}}+\|u^h_L\|_{\dot{B}^{\frac{d}{p}+1}_{p,1}})\|\tau^h_L\|_{\dot{B}^{\frac{d}{p}-1}_{p,1}} \lesssim  (1+E_0)\langle t \rangle^{-\frac{1}{2}(\frac{d}{2}-\sigma_1)}.
$$


It is noted that 
\begin{align*}
&\int^t_0\langle t-t'\rangle^{-\frac{1}{2}(\sigma-\sigma_1+2\sigma'')}\langle t' \rangle^{-\frac{1}{2}(\frac{d}{2}+1-\sigma_1-2\sigma'')}dt'\\
&\leq \langle t\rangle^{-\frac{1}{2}(\sigma-\sigma_1+2\sigma'')},  \ \text{if}\  \frac{d}{2}+1-\sigma_1-2\sigma''> 2,
\end{align*}
 and
\begin{align}\label{AA10}
&\int^t_0 \langle t-t'\rangle^{-\frac{1}{2}(\sigma-\sigma_1+1)}\langle t \rangle^{-\frac{1}{2}(\frac{d}{2}-\sigma_1)}dt'\\
&
\lesssim 
\begin{cases}
\langle t \rangle^{-\frac{1}{2}(\sigma-\sigma_1+1)},~~\ \ \qquad \text{if}~\ \frac{1}{2}(\frac{d}{2}-\sigma_1)>1,\nonumber\\
\langle t \rangle^{-\frac{1}{2}(\sigma-\sigma_1+1)-},~~\qquad \text{if}~\ \frac{1}{2}(\frac{d}{2}-\sigma_1)=1,\nonumber\\
\langle t \rangle^{-\frac{1}{2}(\sigma-\sigma_1+\frac{d}{2}-1-\sigma_1)}, ~~\text{if}~\ \frac{1}{2}(\frac{d}{2}-\sigma_1)<1,
\end{cases}
\end{align}
for $\sigma_1< \sigma\leq\frac{d}{2}-1$.
As a result, we obtain the desired upper bound of $\mathcal{I}_1$ by
\begin{equation}\label{upper-1}
    \langle t \rangle^{-\frac{1}{2}(\sigma-\sigma_1+\min\{\sigma'_2,2\sigma''\})}(1+E_0 ).
\end{equation}} 
\underline{\bf Estimate of $\mathcal{I}_2$}.
Fristly, direct calculation shows
\begin{align*}
\| \mathbb{P}(\tilde{u}\cdot  \nabla u_L)\|^\ell_{\dot{B}^{\sigma_1-1}_{2,\infty}}&\leq \|\mathbb{P}(\tilde{u}^\ell\otimes  u^\ell_L)\|^\ell_{\dot{B}^{\sigma_1}_{2,\infty}}
+\|\mathbb{P}(\tilde{u}^\ell\otimes   u^h_L)\|^\ell_{\dot{B}^{\sigma_1}_{2,\infty}}+\|\mathbb{P}(\tilde{u}^h\otimes  u_L)\|^\ell_{\dot{B}^{\sigma_1}_{2,\infty}}.
\end{align*}
Using the Lemma \ref{app1}, we have
\begin{align*}
\|\mathbb{P}(\tilde{u}^\ell\otimes  u^\ell_L)\|^\ell_{\dot{B}^{\sigma_1}_{2,\infty}}\lesssim \|u^\ell_L\|_{\dot{B}^{\sigma_1+\sigma_2}_{2,\infty}}
\|\tilde{u}^\ell\|_{\dot{B}^{\frac{d}{2}-\sigma_2}_{2,1}}\lesssim \langle t \rangle^{-\frac{1}{2}(\frac{d}{2}-\sigma_1)}\|u^\ell_L\|_{\dot{B}^{\sigma_1+\sigma_2}_{2,\infty}}\tilde{D}_p(t).
\end{align*}
Arguing as inequality \eqref{clfw1}  yields
\begin{align*}
\|\mathbb{P}(\tilde{u}^\ell\otimes   u^h_L)\|^\ell_{\dot{B}^{\sigma_1}_{2,\infty}}&\lesssim \| u^h_L\|_{\dot{B}^{\frac{d}{p}-1}_{p,1}}\|\tilde{u}^\ell\|_{\dot{B}^{\frac{d}{2}-1}_{2,1}}\lesssim  \langle t \rangle^{-\frac{1}{2}(\frac{d}{2}-\sigma_1)}E_0\tilde{D}_p(t),
\end{align*}
where we have used $\langle t \rangle^{-\frac{1}{2}(\frac{d}{2}-1-\sigma_1+\sigma_2)}e^{-Ct}\leq \langle t\rangle^{-\frac{1}{2}(\frac{d}{2}-\sigma_1)} $.
\newline
Similar to  the process of  \eqref{4.28}-\eqref{4.32}, one has
\begin{align*}
\|\mathbb{P}(\tilde{u}^h\otimes u_L)\|^\ell_{\dot{B}^{\sigma_1}_{2,\infty}}&\lesssim (\| u^\ell_L\|_{\dot{B}^{\frac{d}{2}-1}_{2,1}}+\| u^h_L\|_{\dot{B}^{\frac{d}{p}}_{p,1}})\|\tilde{u}^h\|_{\dot{B}^{\frac{d}{p}-1}_{p,1}}\\
&\lesssim   \langle t \rangle^{-\frac{1}{2}(\frac{d}{2}-\sigma_1)}E_0\tilde{D}_p(t),
\end{align*}
for $t>1$. We thus obtain
$$
\|\mathbb{P}(\tilde{u}\otimes   u_L)\|^\ell_{\dot{B}^{\sigma_1}_{2,\infty}}\lesssim \langle t \rangle^{-\frac{1}{2}(\frac{d}{2}-\sigma_1)}(E_0+\|u^\ell_L\|_{\dot{B}^{\sigma_1+\sigma_2}_{2,\infty}})\tilde{D}_p(t).
$$

Secondly, we have
\begin{align*}
\|(\tilde{u}\cdot  \tau_L)\|^\ell_{\dot{B}^{\sigma_1}_{2,\infty}}&\leq \|(\tilde{u}^\ell\cdot  \tau^\ell_L)\|^\ell_{\dot{B}^{\sigma_1}_{2,\infty}}
+\|(\tilde{u}^\ell\cdot \tau^h_L)\|^\ell_{\dot{B}^{\sigma_1}_{2,\infty}}+\|(\tilde{u}^h\cdot  \tau_L)\|^\ell_{\dot{B}^{\sigma_1}_{2,\infty}}\\
&\lesssim \langle t \rangle^{-\frac{1}{2}(\frac{d}{2}-\sigma_1)}(E_0+\|\tau^\ell_L\|_{\dot{B}^{\sigma_1+\sigma_2}_{2,\infty}})\tilde{D}_p(t).
\end{align*}
Noticing the fact  in \eqref{AA10}
for $\sigma_1< \sigma\leq\frac{d}{2}-1$. Thus, we  obtain 
\begin{equation}\label{pu87}
    \begin{aligned}
      &\int^t_1\langle t-t'\rangle^{-\frac{1}{2}(\sigma-\sigma_1+\sigma')}\|(\mathbb{P}(\tilde{u}\cdot \nabla u_L),\tilde{u}\cdot \nabla \tau_L)\|^\ell_{\dot{B}^{\sigma_1-\sigma'}_{2,\infty}}dt'\\
     & \lesssim \langle t \rangle^{-\frac{1}{2}(\sigma-\sigma_1+\sigma_2')}(E_0+\|(u^\ell_L,\tau^\ell_L)\|_{\dot{B}^{\sigma_1+\sigma_2}_{2,\infty}})\tilde{D}_p(t) .
    \end{aligned}
\end{equation}

Finally, we treat the term
\begin{align*}
\|Q(\tilde{\tau},\nabla u_L)\|^\ell_{\dot{B}^{\sigma_1-\sigma'}_{2,\infty}}.
\end{align*}
We write the typical term of $Q(\tilde{\tau},\nabla u_L)$ by
\begin{equation*}
    \tilde{\tau}  \nabla u_L= \tilde{\tau}^\ell \nabla u^\ell_L+\tilde{\tau}^\ell \nabla u^h_L+\tilde{\tau}^h\nabla u_L
\end{equation*}
Direct calculation shows
\begin{align*}
\|\tilde{\tau}^\ell\nabla u^\ell_L\|^\ell_{\dot{B}^{\sigma_1-\sigma'}_{2,\infty}}\lesssim\|\tilde{\tau}^\ell\|_{\dot{B}^{\frac{d}{2}-1}_{2,1}}\|\nabla u^\ell_L\|_{\dot{B}^{\sigma_1-\sigma'+1}_{2,\infty}}\lesssim \langle t \rangle^{-\frac{1}{2}(\frac{d}{2}+1-\sigma_1-\sigma')}\tilde{D}_p(t),
\end{align*}
which and  the similar process in \eqref{clfw1} and \eqref{clfw2} implies 
\begin{align*}
\|\tilde{\tau}^\ell\nabla u^h_L\|^\ell_{\dot{B}^{\sigma_1-\sigma'}_{2,\infty}}&\leq \|\tilde{\tau}^\ell\|_{\dot{B}^{\frac{d}{2}-1}_{2,1}}\|\nabla u^h_L\|_{\dot{B}^{\frac{d}{p}-1}_{p,1}}+\|\tilde{\tau}^\ell\|_{\dot{B}^{\frac{d}{2}-1}_{2,1}}\|\nabla u^h_L\|_{\dot{B}^{1-\frac{d}{2}}_{2,1}}\\
&\lesssim \|\tilde{\tau}^\ell\|_{\dot{B}^{\frac{d}{2}-1}_{2,1}}\|\nabla u^h_L\|_{\dot{B}^{\frac{d}{p}-1}_{p,1}},
\end{align*}
and
\begin{align*}
\| \tilde{\tau}^h \nabla u_L \|^\ell_{\dot{B}^{\sigma_1-\sigma'}_{2,\infty}} &\lesssim (\| u^\ell_L\|_{\dot{B}^{\frac{d}{2}}_{2,1}}+\| u^h_L\|_{\dot{B}^{\frac{d}{p}+1}_{p,1}})\|\tilde{\tau}^h\|_{\dot{B}^{\frac{d}{p}-1}_{p,1}}\\
&\lesssim \langle t \rangle^{-\frac{1}{2}(d-2\sigma_1+\sigma_2)}\tilde{D}_p(t)+ \langle t \rangle^{-\frac{1}{2}(d-2\sigma_1+\sigma_2)}E_0.
\end{align*}
As a result, we obtain 
\begin{equation*}
    \|Q(\tilde{\tau},\nabla u_L)\|^\ell_{\dot{B}^{\sigma_1-\sigma'}_{2,\infty}} \lesssim  \langle t \rangle^{-\frac{1}{2}(\frac{d}{2}+1-\sigma_1-\sigma')}\tilde{D}_p(t)+\langle t \rangle^{-\frac{1}{2}(d-2\sigma_1+\sigma_2)}\tilde{D}_p(t)+ \langle t \rangle^{-\frac{1}{2}(d-2\sigma_1+\sigma_2)}E_0.
\end{equation*}
Moreover,  one has   
\begin{align*}
&\int^t_0\langle t-t'\rangle^{-\frac{1}{2}(\sigma-\sigma_1+2\sigma'')}\langle t' \rangle^{-\frac{1}{2}(d-2\sigma_1+\sigma_2)}E_0dt'\\
&\leq \langle t\rangle^{-\frac{1}{2}(\sigma-\sigma_1+2\sigma'')}E_0
\end{align*}
and
\begin{align*}
&\int^t_0\langle t-t'\rangle^{-\frac{1}{2}(\sigma-\sigma_1+2\sigma'')}\langle t' \rangle^{-\frac{1}{2}(\frac{d}{2}+1-\sigma_1-\sigma')}\tilde{D}_p(t)dt'\\
&\leq \langle t\rangle^{-\frac{1}{2}(\sigma-\sigma_1+2\sigma'')}\tilde{D}_p(t).
\end{align*}
Together with \eqref{pu87}, we hence obtain the upper bound of $\mathcal{I}_2$ by
\begin{equation}\label{upper-2}
    \langle t \rangle^{-\frac{1}{2}(\sigma-\sigma_1+\min\{\sigma'_2,\sigma''\})}(E_0 \tilde{D}_p+\langle t\rangle^{-\frac{\sigma''}{2}}\tilde{D}_p+\|(u^\ell_L,\tau^\ell_L)\|_{L^\infty_t(\dot{B}^{\sigma_1+\sigma_2}_{2,\infty})}\tilde{D}_p).
\end{equation}

\underline{\bf Estimate of $\mathcal{I}_3$}.
The treatment of $\tilde{F}$ is similar as that of $\tilde{F}_L$. By using the fact 
\begin{equation*}
    u \otimes\tilde u =u^\ell \otimes \tilde{ u}^\ell+u^h\otimes \tilde{u}^\ell+ u \otimes \tilde{ u}^h
\end{equation*}
then, we can compute
\begin{align*}
\|\mathbb{P}(u\otimes \tilde{ u})\|^\ell_{\dot{B}^{\sigma_1}_{2,\infty}}
\leq& \| u^\ell\|_{\dot{B}^{\sigma_1+\sigma_2}_{2,\infty}}
\|\tilde{u}^\ell\|_{\dot{B}^{\frac{d}{2}-\sigma_2}_{2,1}}+\| u^h\|_{\dot{B}^{\frac{d}{p}-1}_{p,1}}\|\tilde{u}^\ell\|_{\dot{B}^{\frac{d}{2}-1}_{2,1}}\\
& +(\| u^\ell\|_{\dot{B}^{\frac{d}{2}-1}_{2,1}}+\| u^h\|_{\dot{B}^{\frac{d}{p}}_{p,1}})\|\tilde{u}^h\|_{\dot{B}^{\frac{d}{p}-1}_{p,1}}\\
\lesssim& \langle t \rangle^{-\frac{1}{2}(\frac{d}{2}-\sigma_1)}\|u^\ell\|_{\dot{B}^{\sigma_1+\sigma_2}_{2,\infty}}\tilde{D}_p(t)+\langle t \rangle^{-\frac{1}{2}(\frac{d}{2}-\sigma_1)}E_0\tilde{D}_p(t)\\
&+\langle t \rangle^{-\frac{1}{2}(\frac{d}{2}-\sigma_1)}E_0,
\end{align*}
where we used the fact $u^h=u^h_L+\tilde{u}^h$.

Similarly, we have 
\begin{align*}
\|u\cdot  \tilde{\tau}\|^\ell_{\dot{B}^{\sigma_1}_{2,\infty}}&\lesssim \langle t \rangle^{-\frac{1}{2}(\frac{d}{2}-\sigma_1)}\|u^\ell\|_{\dot{B}^{\sigma_1+\sigma_2}_{2,\infty}}\tilde{D}_p(t)+\langle t \rangle^{-\frac{1}{2}(\frac{d}{2}-\sigma_1)}E_0\tilde{D}_p(t) +\langle t \rangle^{-\frac{1}{2}(\frac{d}{2}-\sigma_1)}E_0.
\end{align*}

\par
Finally, we tackle the term
$
\|Q(\tau,\nabla \tilde{u})\|^\ell_{\dot{B}^{\sigma_1-\sigma'}_{2,\infty}}.
$
We write the typical term of $Q(\tau,\nabla \tilde{u})$  by
    \begin{equation*}
        \tau \nabla \tilde{u}=\tau^\ell\nabla \tilde{u}^h+  \tau\nabla \tilde{u}^\ell +  \tau^h \nabla \tilde{u}^h
    \end{equation*}
and proceed with the following three steps.
 
{\bf Step I}.
Note that
$$
\tau^\ell\nabla \tilde{u}^h=\tau_{ik}^\ell\partial_k\tilde{u}_j^h=\partial_k(\tau_{ik}^\ell\tilde{u}_j^h )-\partial_k\tau_{ik}^\ell\tilde{u}_j^h,
$$
similar to \eqref{clfw1}, for the first term, by choosing $\sigma'=1$, we obtain 
\begin{align*}
&\int^t_0\langle t-t'\rangle^{-\frac{1}{2}(\sigma-\sigma_1+1)}\|{\tau}^\ell \cdot \tilde{u}^h\|^\ell_{\dot{B}^{\sigma_1}_{2,\infty}}dt'\\
&\lesssim \int^t_0\langle t-t'\rangle^{-\frac{1}{2}(\sigma-\sigma_1+1)}(\|\tilde{u}^h\|_{\dot{B}^{\frac{d}{p}-1}_{p,1}}\|\tau^\ell\|_{\dot{B}^{\frac{d}{2}-1}_{2,1}})dt'\\
&\lesssim \int^t_0\langle t-t'\rangle^{-\frac{1}{2}(\sigma-\sigma_1+1)}\langle t' \rangle^{-\frac{1}{2}(\frac{d}{2}-\sigma_1)}E_0\tilde{D}_p(t).
\end{align*}
Then, let $ \sigma'=2\sigma''$ be small enough, one has
\begin{align*}
&\int^t_0\langle t-t'\rangle^{-\frac{1}{2}(\sigma-\sigma_1+\sigma')}\|\nabla {\tau}^\ell   \tilde{u}^h\|^\ell_{\dot{B}^{\sigma_1-\sigma'}_{2,\infty}}dt'\\
&\lesssim \int^t_0\langle t-t'\rangle^{-\frac{1}{2}(\sigma-\sigma_1+\sigma')}(\|\nabla{\tau}^\ell\|_{\dot{B}^{\frac{d}{2}-1}_{2,1}}\|\tilde{u}^h\|_{\dot{B}^{\frac{d}{p}-1}_{p,1}})dt'\\
&\lesssim \int^t_0\langle t-t'\rangle^{-\frac{1}{2}(\sigma-\sigma_1+\sigma')}\langle t' \rangle^{-\frac{1}{2}(\frac{d}{2}-\sigma_1)}\langle t' \rangle^{-\frac{1}{2}(\frac{d}{2}-\sigma_1+\sigma_2)}(\tilde{D}_p(t)+\tilde{D}^2_p(t))\\
&\lesssim \langle t\rangle^{-\frac{1}{2}(\sigma-\sigma_1+\sigma'')}\langle t\rangle^{-\frac{\sigma''}{2}}(\tilde{D}_p(t)+\tilde{D}^2_p(t)),
\end{align*}
where $\sigma-\sigma_1+\sigma'\leq d-2\sigma_1+\sigma_2$ and $ d-2\sigma_1+\sigma_2>2$.
\par

{\bf Step II}. Following a similar argument in \eqref{clfw1} and \eqref{clfw2},  one has
\begin{align}\label{aaa1}
&\int^t_0\langle t-t'\rangle^{-\frac{1}{2}(\sigma-\sigma_1+\sigma')}\| {\tau}  \nabla \tilde{u}^\ell\|^\ell_{\dot{B}^{\sigma_1-\sigma'}_{2,\infty}}dt'\nonumber\\
&\lesssim \int^t_0\langle t-t'\rangle^{-\frac{1}{2}(\sigma-\sigma_1+\sigma')}(\|\nabla\tilde{u}^\ell\|_{\dot{B}^{\frac{d}{2}-1}_{2,1}}\|\tau^h\|_{\dot{B}^{\frac{d}{p}-1}_{p,1}}
+\|\tau^\ell\|_{\dot{B}^{\frac{d}{2}-\frac{\sigma'}{2}}_{2,1}}\|\nabla \tilde{u}^\ell\|_{\dot{B}^{\sigma_1-\frac{\sigma'}{2}}_{2,\infty}})dt'\nonumber\\
&\lesssim \int^t_0\langle t-t'\rangle^{-\frac{1}{2}(\sigma-\sigma_1+\sigma')}\big(\langle t' \rangle^{-\frac{1}{2}(\frac{d}{2}-\sigma_1)}\langle t' \rangle^{-\frac{1}{2}(\frac{d}{2}-\sigma_1+\sigma_2)}\tilde{D}_p(t)+\langle t \rangle^{-\frac{1}{2}(\frac{d}{2}+1-\sigma_1-\sigma'+\sigma_2)}\tilde{D}_p(t)\nonumber\\
&\ \ \ \ \ \ \ + \langle t \rangle^{-\frac{1}{2}(\frac{d}{2}+1-\sigma_1-\sigma'+2\sigma_2)} \tilde{D}^2_p(t)\big)dt'\nonumber\\
&\lesssim \langle t\rangle^{-\frac{1}{2}(\sigma-\sigma_1+\sigma'')}\langle t\rangle^{-\frac{\sigma''}{2}}(\tilde{D}_p(t)+\tilde{D}^2_p(t)),
\end{align}
where we have used 
\begin{align*}
\|\tilde{\tau}^\ell\|_{\dot{B}^{\frac{d}{2}-\frac{\sigma'}{2}}_{2,1}}\|\nabla \tilde{u}^\ell\|_{\dot{B}^{\sigma_1-\frac{\sigma'}{2}}_{2,\infty}}&\lesssim \langle t \rangle^{-\frac{1}{2}(\frac{d}{2}+1-\sigma_1-\sigma'+2\sigma_2)}\tilde{D}^2_p(t).
\end{align*}
 
{\bf Step III}. It remains to estimate $\| \tau^h \nabla \tilde{u}^h \|^\ell_{\dot{B}^{\sigma_1-\sigma'}_{2,\infty}}$.
If $2\leq p\leq d$, 
by using  \eqref{A.4} we have
\begin{align*}
\| \tau^h \nabla \tilde{u}^h \|^\ell_{\dot{B}^{\sigma_1-\sigma'}_{2,\infty}}\lesssim \|\nabla \tilde u^h\|_{\dot{B}_{p,1}^{\frac{d}{p}-1}}^\ell \|\tau^h\|_{\dot{B}_{p,1}^{1-\frac{d}{p}}}^\ell\lesssim \|\tilde u\|^h_{\dot{B}_{p,1}^{\frac{d}{p}}}\|\tau\|^h_{\dot{B}_{p,1}^{1-\frac{d}{p}}}.
\end{align*}
If $p>d$, thanks to  Lemma \ref{app2} with $s=1-\frac{d}{p}$ one has
\begin{align*}
\| \tau^h \nabla \tilde{u}^h \|^\ell_{\dot{B}^{\sigma_1-\sigma'}_{2,\infty}}&\leq\|Q(\tau^h,\nabla \tilde{u}^h)\|^\ell_{\dot{B}^{\sigma_0}_{2,\infty}}\\
&\lesssim (\|\nabla \tilde{u}^h\|_{\dot{B}^{1-\frac{d}{p}}_{p,1}}+\|(\nabla \tilde{u}^h)^\ell\|_{L^{p*}})\|\tau^h\|_{\dot{B}^{\frac{d}{p}-1}_{p,1}}\\
&\lesssim\|\tilde{u}^h\|_{\dot{B}^{\frac{d}{p}+1}_{p,1}}\|\tau^h\|_{\dot{B}^{\frac{d}{p}-1}_{p,1}}.
\end{align*}

 Noting the high-frequencies of $(u_L,\tau_L)$ has exponential decay,  we then obtain the upper bound of $\mathcal{I}_3$ by
\begin{equation}\label{upper-3}
    \langle t \rangle^{-\frac{1}{2}(\sigma-\sigma_1+\min\{\sigma'_2,\sigma''\})}(E_0+\langle t\rangle^{-\frac{\sigma''}{2}}\tilde{D}_p+\langle t\rangle^{-\frac{\sigma''}{2}}\tilde{D}_p^2).
\end{equation}

Finally, collecting the upper bounds \eqref{upper-1}, \eqref{upper-2} and \eqref{upper-3}, and then we 
finish the proof of Lemma \ref{cl3}. 
\end{proof}
\begin{lemma}\label{cl5}
Let $p$ satisfy \eqref{ppp} and $\sigma_0\leq \sigma_1<\frac{d}{2}-1.$ Then we have
\begin{align*}
\|(\tilde{u},\tilde{\tau})(t)\|^\ell_{\dot{B}^{\sigma}_{2,1}}
&\lesssim \langle t \rangle^{-\frac{1}{2}(\sigma-\sigma_1+\sigma_2)}(1+E_0)++\langle t\rangle^{-\frac{1}{2}(\sigma-\sigma_1+\sigma'')}\langle t\rangle^{-\frac{\sigma''}{2}}(\tilde{D}^2_p(t)+\tilde{D}_p(t))
\\
&\ \ \ \langle t \rangle^{-\frac{1}{2}(\frac{d}{2}-\sigma_1+\sigma_2-\sigma_3)}(E_0+\|u^\ell_L,\tau^\ell_L\|_{\tilde L^\infty_t(\dot{B}^{\sigma_1+\sigma_3}_{2,\infty})})\tilde{D}_p(t)
\end{align*}
for $t>2$, $\frac{d}{2}-1<\sigma\leq \frac{d}{2}$ and sufficiently small $\sigma''>0$
with $$
\sigma_2 =\min\{\frac{1}{2},(\frac{d}{2}-1-\sigma_1)-,\sigma''\},  \quad  \sigma_3 =\min\{
\frac{d}{2}-\sigma, \frac{d}{2}-1-\sigma_1-\sigma_2\}.
$$

\end{lemma}
\begin{proof} Thanks to \eqref{key ine1}
we write  
\begin{align*}
 \|(\tilde{u},\tilde{\tau})(t)\|^\ell_{\dot{B}^{\sigma}_{2,1}}&\leq\int^t_0\langle t-t'\rangle^{-\frac{1}{2}(\sigma-\sigma_1+\sigma')}\|F,H\|^\ell_{\dot{B}^{\sigma_1-\sigma'}_{2,\infty}}dt'\\
 &=\int^1_0\langle t-t'\rangle^{-\frac{1}{2}(\sigma-\sigma_1+\sigma')}\|F,H\|^\ell_{\dot{B}^{\sigma_1-\sigma'}_{2,\infty}}dt'+\int^t_1\langle t-t'\rangle^{-\frac{1}{2}(\sigma-\sigma_1+\sigma')}\|F,H\|^\ell_{\dot{B}^{\sigma_1-\sigma'}_{2,\infty}}dt'
\end{align*} and then bound the terms 
\begin{equation*}
    \begin{aligned}
        \mathcal{K}_1=&\int^t_1\langle t-t'\rangle^{-\frac{1}{2}(\sigma-\sigma_1+\sigma')}\|F_L\|^\ell_{\dot{B}^{\sigma_1-\sigma'}_{2,\infty}}dt',\\
        \mathcal{K}_2=&\int^t_1\langle t-t'\rangle^{-\frac{1}{2}(\sigma-\sigma_1+\sigma')}\|\tilde F \|^\ell_{\dot{B}^{\sigma_1-\sigma'}_{2,\infty}}dt',
       \\ \mathcal{K}_3=&\int^t_1\langle t-t'\rangle^{-\frac{1}{2}(\sigma-\sigma_1+\sigma')}\|\tilde F_L\|^\ell_{\dot{B}^{\sigma_1-\sigma'}_{2,\infty}}dt'.
    \end{aligned}
\end{equation*}
\underline{\bf Estimate of $\mathcal{K}_1 $}.
Recalling that $$F_L=-\mathbb{P}(u_L\cdot \nabla u_L)-u_L\cdot \nabla \tau_L-Q(\tau_L,\nabla u_L), $$ we will deal with   the three terms respectively.

Firstly, choosing $\sigma'=1,\sigma_1+\sigma_3<\frac{d}{2}$ and  using Lemma \ref{app1},
we arrive at 
\begin{align*}
\|\mathbb{P}(u^\ell_L\otimes  u^\ell_L)\|^\ell_{\dot{B}^{\sigma_1}_{2,\infty}}&\lesssim \|u^\ell_L\|_{\dot{B}^{\sigma_1+\sigma_3}_{2,\infty}}\|u^\ell_L\|_{\dot{B}^{\frac{d}{2}-\sigma_3}_{2,\infty}}
\lesssim \langle t \rangle^{-\frac{1}{2}(\frac{d}{2}-\sigma_1 )}.
\end{align*}

Similarly, using \eqref{A.3} in Corollary \ref{app3},  we have
\begin{align*}
\|\mathbb{P}(u^h_L\otimes   u^\ell_L)\|^\ell_{\dot{B}^{\sigma_1}_{2,\infty}}\lesssim \|u^h_L\|_{\dot{B}^{\frac{d}{p}-1}_{p,1}}\|u^\ell_L\|_{\dot{B}^{\frac{d}{2}-1}_{2,1}}\lesssim \langle t \rangle^{-\frac{1}{2}(\frac{d}{2}-\sigma_1)}.
\end{align*}
Moreover, using the similar process of inequalities \eqref{4.28}-\eqref{4.32}, one has
\begin{align*}
\|\mathbb{P}(u_L\otimes    u^h_L)\|^\ell_{\dot{B}^{\sigma_1}_{2,\infty}}\lesssim (\|u^\ell_L\|_{\dot{B}^{\frac{d}{2}-1}_{2,1}}+\|u^h_L\|_{\dot{B}^{\frac{d}{p}}_{p,1}})\|u^h_L\|_{\dot{B}^{\frac{d}{p}-1}_{p,1}}
\lesssim \langle t \rangle^{-\frac{1}{2}(\frac{d}{2}-\sigma_1)}.
\end{align*}
Collecting the above estimates, we have
\begin{equation*}
    \begin{aligned}
\|\mathbb{P}(u_L\otimes    u_L)\|^\ell_{\dot{B}^{\sigma_1}_{2,\infty}}
&\lesssim \langle t \rangle^{-\frac{1}{2}(\frac{d}{2}-\sigma_1)}.
\end{aligned}
\end{equation*}

Similarly, we can compute
\begin{align*}
\|u_L\cdot \tau_L\|^\ell_{\dot{B}^{\sigma_1}_{2,\infty}}&\lesssim \|u^\ell_L\|_{\dot{B}^{\sigma_1+\sigma_3}_{2,\infty}}\|\tau^\ell_L\|_{\dot{B}^{\frac{d}{2}-\sigma_3}_{2,\infty}}+\|u^h_L\|_{\dot{B}^{\frac{d}{p}-1}_{p,1}}\|\tau^\ell_L\|_{\dot{B}^{\frac{d}{2}-1}_{2,1}}\\
&\quad+
(\|u^\ell_L\|_{\dot{B}^{\frac{d}{2}-1}_{2,1}}+\|u^h_L\|_{\dot{B}^{\frac{d}{p}}_{p,1}})\|\tau^h_L\|_{\dot{B}^{\frac{d}{p}-1}_{p,1}}\\
&\lesssim \langle t \rangle^{-\frac{1}{2}(\frac{d}{2}-\sigma_1)},~~~t>1.
\end{align*}

Secondly,  we turn to  the term $\|Q(\tau_L,\nabla u_L)\|^\ell_{\dot{B}^{\sigma_1-\sigma'}_{2,\infty}}$. For that we write the typical term of $Q(\tau_L,\nabla u_L)$ that 
$$
\tau_L\nabla u_L=\tau^\ell_L\nabla u^\ell_L+\tau^\ell_L\nabla u^h_L+\tau^h_L\nabla u_L
$$
and divide into two steps to proceed.

 {\bf Step I}.
 In this step, we aim to deal with  $\tau^\ell_L\nabla u^\ell_L$ and $\tau^\ell_L\nabla u^h_L$. 
By using Lemma \ref{ap1}, we obtain
\begin{align*}
\|\tau^\ell_L\nabla u^\ell_L\|^\ell_{\dot{B}^{\sigma_1-\sigma'}_{2,\infty}}&\lesssim\|\tau^\ell_L\|_{\dot{B}^{\frac{d}{2}-\frac{\sigma'}{2}}_{2,1}}\|\nabla u^\ell_L\|_{\dot{B}^{\sigma_1-\frac{\sigma'}{2}}_{2,\infty}}\\
&\lesssim \langle t \rangle^{-\frac{1}{2}(\frac{d}{2}+1-\sigma_1-\sigma')},
\end{align*}
where $\sigma_1-\frac{\sigma'}{2}\leq \frac{d}{2}$, $\frac{d}{2}+\sigma_1-\sigma'>0$. 

\par
For the case of $p>2$, we have
\begin{align*}
\|\tau^\ell_L\nabla u^h_L\|^\ell_{\dot{B}^{\sigma_1-\sigma'}_{2,\infty}}\lesssim \|\tau^\ell_L\nabla u^h_L\|_{\dot{B}^{\sigma_1+\frac{d}{p}-\frac{d}{2}}_{2,\infty}}
\lesssim \|\tau^\ell_L\|_{\dot{B}^{\frac{d}{2}-1}_{2,1}}\|\nabla u^h_L\|_{\dot{B}^{\frac{d}{p}-1}_{p,1}},
\end{align*}
where 
$0<\sigma'\leq  \frac{d}{2}-\frac{d}{p}$ and $\frac{d}{p}\geq \frac{d}{2}-1, i.e., p\leq d^*$.
\par
For the case of $p=2$,
\begin{align*}
\|\tau^\ell_L\nabla u^h_L\|^\ell_{\dot{B}^{\sigma_1-\sigma'}_{2,\infty}}\lesssim \|\tau^\ell_L\nabla u^h_L\|_{\dot{B}^{\sigma_0}_{2,\infty}}\lesssim \|\tau^\ell_L\|_{\dot{B}^{\frac{d}{p}-1}_{p,1}}\|\nabla u^h_L\|_{\dot{B}^{1-\frac{d}{p}}_{p,1}},
\end{align*}
where  \eqref{A.4}  and 
  $\sigma_1>\sigma_0$ were used. Thanks to the exponential decay of the high frequencies of $(u^h_L,\tau^h_L)$, the rest of estimates are similar and hence we omit details.

{\bf Step II}.
In this step we treat the term
\begin{align*}
\|\tau^h_L\nabla u_L\|^\ell_{\dot{B}^{\sigma_1-\sigma'}_{2,\infty}}.
\end{align*}

If $2\leq p<d$, according to Lemma \ref{app2} with $s=\frac{d}{p}-1$ we get
\begin{align*}
\|(\tau^h_L  \nabla u_L)\|^\ell_{\dot{B}^{\sigma_1-\sigma'}_{2,\infty}}\leq\|(\tau^h_L  \nabla u_L)\|^\ell_{\dot{B}^{\sigma_0}_{2,\infty}}\leq (\|\nabla u_L\|_{\dot{B}^{\frac{d}{p}-1}_{p,1}}+\|(\nabla u_L)^\ell\|_{L^{p*}})\|\tau^h_L\|_{\dot{B}^{1-\frac{d}{p}}_{p,1}},
\end{align*}
where $\sigma_1>\sigma_0$ and $1-\frac{d}{p}\leq \frac{d}{p}+2$.

If $p=d$, one gets
\begin{align*}
\|(\tau^h_L  \nabla u_L)\|^\ell_{\dot{B}^{\sigma_1-\sigma'}_{2,\infty}}
\leq \|(\tau^h_L  \nabla u_L)\|^\ell_{\dot{B}^{\sigma_0}_{2,\infty}}\lesssim  \|(\tau^h_L  \nabla u_L)\|^\ell_{L^{\frac{d}{2}}}\lesssim \|\nabla u_L\|_{\dot{B}^{0}_{d,1}}\|\tau^h_L\|_{\dot{B}^{0}_{d,1}}
\end{align*}
which together with $\dot{B}^{\frac{d}{p}}_{2,1}\hookrightarrow L^{p*}$ and $\dot{B}^{\frac{d}{p}-1}_{p,1}\hookrightarrow \dot{B}^{0}_{d,1}$ leads to 
$$
\|(\tau^h_L  \nabla u_L)\|^\ell_{\dot{B}^{\sigma_1}_{2,\infty}}\lesssim (\|u^\ell_L\|_{\dot{B}^{\frac{d}{2}}_{2,1}}+\|u^h_L\|_{\dot{B}^{\frac{d}{p}}_{p,1}})\|\tau^h_L\|_{\dot{B}^{\frac{d}{p}-1}_{p,1}}.
$$

If $p>d$, applying  Lemma \ref{lemma12} with $s=1-\frac{d}{p}$ again implies that
$$
\|(\tau^h_L  \nabla u_L)\|^\ell_{\dot{B}^{\sigma_0}_{2,\infty}}\leq (\|\nabla u_L\|_{\dot{B}^{1-\frac{d}{p}}_{p,1}}+\|(\nabla u_L)^\ell\|_{L^{p*}})\|\tau^h_L\|_{\dot{B}^{\frac{d}{p}-1}_{p,1}}.
$$
By using the embedding $\dot{B}^{1+\sigma_0}_{2,1}\hookrightarrow \dot{B}^{1-\frac{d}{p}}_{p,1}$, $\frac{d}{2}-1<1+\sigma_0$ and $p>d$, we obtain
$$
\|(\tau^h_L  \nabla u_L)\|^\ell_{\dot{B}^{\sigma_1}_{2,\infty}}\lesssim (\|u^\ell_L\|_{\dot{B}^{\frac{d}{2}}_{2,1}}+\|u^h_L\|_{\dot{B}^{\frac{d}{p}+1}_{p,1}})\|\tau^h_L\|_{\dot{B}^{\frac{d}{p}-1}_{p,1}}.
$$
 
Thanks to the exponential decay of the high frequencies of $(u^h_L,\tau^h_L)$  and 
\begin{align*}
&\int^t_0\langle t-t'\rangle^{-\frac{1}{2}(\sigma-\sigma_1+1)}\langle t' \rangle^{-\frac{1}{2}(\frac{d}{2} -\sigma_1 )}dt'\\
&\leq \langle t\rangle^{-\frac{1}{2}(\sigma-\sigma_1+\sigma'')},
\end{align*}
and
\begin{align*}
&\int^t_0\langle t-t'\rangle^{-\frac{1}{2}(\sigma-\sigma_1+\sigma')}\langle t' \rangle^{-\frac{1}{2}(\frac{d}{2}+1-\sigma_1-\sigma')}dt'\nonumber\\
&\leq \langle t\rangle^{-\frac{1}{2}(\sigma-\sigma_1+\sigma')},  
\end{align*}
where $0\leq\sigma-\sigma_1+\sigma'\leq\frac{d}{2}+1-\sigma_1-\sigma', \frac{d}{2}+1-\sigma_1-\sigma'>2$ which come from small $\sigma'$,
we can  then  get the desired upper bound of $\mathcal{K}_1$ by
\begin{equation}\label{upp-1}
    \langle t \rangle^{-\frac{1}{2}(\sigma-\sigma_1+\sigma_2)}(1+E_0).
\end{equation}
 \newline
\underline{\bf Estimate of $\mathcal{K}_2 $}. We first deal with the term
$
\|Q(\tau,\nabla \tilde{u})\|^\ell_{\dot{B}^{\sigma_1-\sigma'}_{2,\infty}}.
$ For that we write  the typical term of $Q(\tau,\nabla \tilde{u})$  that 
$$
\tau \nabla \tilde{u}  =\tau^\ell \nabla \tilde u^h + \tau  \nabla \tilde u^\ell +\tau^h \nabla \tilde u ^h
$$
and divide into three steps to proceed.

For the case $\sigma_1 < \sigma \leq \frac{d}{2} - 1$, it holds
\begin{align*}
&\int^t_0\langle t-t'\rangle^{-\frac{1}{2}(\sigma-\sigma_1+1)}\|{\tau}^\ell  \tilde{u}^h\|^\ell_{\dot{B}^{\sigma_1}_{2,\infty}}dt'\\
&\lesssim \int^t_0\langle t-t'\rangle^{-\frac{1}{2}(\sigma-\sigma_1+1)}(\|\tilde{u}^h\|_{\dot{B}^{\frac{d}{p}-1}_{p,1}}\|\tau^\ell\|_{\dot{B}^{\frac{d}{2}-1}_{2,1}})dt'\\
&\lesssim \int^t_0\langle t-t'\rangle^{-\frac{1}{2}(\sigma-\sigma_1+1)}\langle t' \rangle^{-\frac{1}{2}(\frac{d}{2}-\sigma_1)}E_0\tilde{D}_p(t)dt'\\
&\lesssim \langle t \rangle^{-\frac{1}{2}(\sigma-\sigma_1+\sigma'')}E_0\tilde{D}_p(t),
\end{align*}
and
\begin{align*}
&\int^t_0\langle t-t'\rangle^{-\frac{1}{2}(\sigma-\sigma_1+\sigma')}\|\nabla {\tau}^\ell   \tilde{u}^h\|^\ell_{\dot{B}^{\sigma_1-\sigma'}_{2,\infty}}dt'\\
&\lesssim \int^t_0\langle t-t'\rangle^{-\frac{1}{2}(\sigma-\sigma_1+\sigma')}(\|\nabla{\tau}^\ell\|_{\dot{B}^{\frac{d}{2}-1}_{2,1}}\|\tilde{u}^h\|_{\dot{B}^{\frac{d}{p}-1}_{p,1}})dt'\\
&\lesssim \int^t_0\langle t-t'\rangle^{-\frac{1}{2}(\sigma-\sigma_1+\sigma')}\langle t' \rangle^{-\frac{1}{2}(\frac{d}{2}-\sigma_1)-}\langle t' \rangle^{-\frac{1}{2}(\frac{d}{2}-\sigma_1+\sigma_2)}\tilde{D}^2_p(t)\\
& \lesssim \langle t\rangle^{-\frac{1}{2}(\sigma-\sigma_1+\sigma'')}\langle t\rangle^{-\frac{\sigma''}{2}}\tilde{D}^2_p(t),
\end{align*}
where $\sigma-\sigma_1+\sigma'\leq d-2\sigma_1+\sigma_2$.

Consequently, we get
\begin{align*}
&\int^t_0\langle t-t'\rangle^{-\frac{1}{2}(\sigma-\sigma_1+1)}\|{\tau}^\ell  \nabla\tilde{u}^h\|^\ell_{\dot{B}^{\sigma_1}_{2,\infty}}dt' \\
&\lesssim\langle t \rangle^{-\frac{1}{2}(\sigma-\sigma_1+\sigma'')}E_0\tilde{D}_p(t)+\langle t\rangle^{-\frac{1}{2}(\sigma-\sigma_1+\sigma'')}\langle t\rangle^{-\frac{\sigma''}{2}}\tilde{D}^2_p(t).
\end{align*}
Direct calculation leads to
\begin{align*}
&\int^t_0\langle t-t'\rangle^{-\frac{1}{2}(\sigma-\sigma_1+\sigma')}\| {\tau}  \nabla \tilde{u}^\ell\|^\ell_{\dot{B}^{\sigma_1-\sigma'}_{2,\infty}}dt'\\
&\lesssim \int^t_0\langle t-t'\rangle^{-\frac{1}{2}(\sigma-\sigma_1+\sigma')}(\|\nabla\tilde{u}^\ell\|_{\dot{B}^{\frac{d}{2}-1}_{2,1}}\|\tau^h\|_{\dot{B}^{\frac{d}{p}-1}_{p,1}}
+\|\tau^\ell\|_{\dot{B}^{\frac{d}{2}-\frac{\sigma'}{2}}_{2,1}}\|\nabla \tilde{u}^\ell\|_{\dot{B}^{\sigma_1-\frac{\sigma'}{2}}_{2,\infty}})dt'\\
&\lesssim \int^t_0\langle t-t'\rangle^{-\frac{1}{2}(\sigma-\sigma_1+\sigma')}(\langle t' \rangle^{-\frac{1}{2}(\frac{d}{2}-\sigma_1)}\langle t' \rangle^{-\frac{1}{2}(\frac{d}{2}-\sigma_1+\sigma_2)}E_0+\langle t \rangle^{-\frac{1}{2}(\frac{d}{2}+1-\sigma_1-\sigma')}E_0)dt',
\end{align*}
where $ \sigma'=2\sigma''$ is enough small.

The estimates of high frequencies  
$
\| \tau^h \nabla \tilde{u}^h \|^\ell_{\dot{B}^{\sigma_1-\sigma'}_{2,\infty}}
$ is similar.

Secondly,  we will treat the terms $-\mathbb{P}({u}\cdot \nabla \tilde u )$ and $ - {u}\cdot \nabla \tilde\tau $, direct computation shows
\begin{align*}
\|\mathbb{P}(u\otimes\tilde{ u})\|^\ell_{\dot{B}^{\sigma_1}_{2,\infty}}&=\|\mathbb{P}(u^\ell\otimes \tilde{u}^\ell)\|^\ell_{\dot{B}^{\sigma_1}_{2,\infty}}
+\|\mathbb{P}(u^h\otimes \tilde{u}^\ell)\|^\ell_{\dot{B}^{\sigma_1}_{2,\infty}}+\|\mathbb{P}(u\otimes \tilde{ u}^h)\|^\ell_{\dot{B}^{\sigma_1}_{2,\infty}}\\
&\leq \|{u}^\ell\|_{\dot{B}^{\sigma_1+\sigma_3}_{2,\infty}}\|\tilde{u}^\ell\|_{\dot{B}^{\frac{d}{2}-\sigma_3}_{2,1}}+\| u^h\|_{\dot{B}^{\frac{d}{p}-1}_{p,1}}\|\tilde{u}^\ell\|_{\dot{B}^{\frac{d}{2}-1}_{2,1}}\\
&\quad +(\| u^\ell\|_{\dot{B}^{\frac{d}{2}-1}_{2,1}}+\| u^h\|_{\dot{B}^{\frac{d}{p}}_{p,1}})\|\tilde{u}^h\|_{\dot{B}^{\frac{d}{p}-1}_{p,1}}
\\
& \lesssim \langle t \rangle^{-\frac{1}{2}(\frac{d}{2}-\sigma_1+\sigma_2-\sigma_3)}(E_0+\|u^\ell_L\|_{L^\infty_t(\dot{B}^{\sigma_1+\sigma_3}_{2,\infty}}))\tilde{D}_p(t),
\end{align*}
where we used $u^h=u^h_L+\tilde{u}^h$,
and  $\sigma_2, \sigma_3$ will be determined later. 
\par
Then we can also deal with the term $\|u\cdot \tilde{\tau}\|^\ell_{\dot{B}^{\sigma_1}_{2,\infty}}$by a similar argument.
Thus, we have
\begin{align*}
&\int^t_1\langle t-t'\rangle^{-\frac{1}{2}(\sigma-\sigma_1+1)}(\|\mathbb{P}(u\otimes\tilde{ u})\|^\ell_{\dot{B}^{\sigma_1}_{2,\infty}}+\|u\cdot \tilde{\tau}\|^\ell_{\dot{B}^{\sigma_1}_{2,\infty}})dt'\\
&\lesssim \int^t_0 \langle t-t'\rangle^{-\frac{1}{2}(\sigma-\sigma_1+1)}\langle t' \rangle^{-\frac{1}{2}(\frac{d}{2}-\sigma_1+\sigma_2-\sigma_3)}dt' (E_0+\|u^\ell_L,\tau^\ell_L\|_{L^\infty_t(\dot{B}^{\sigma_1+\sigma_3}_{2,\infty}}))\tilde{D}_p(t).
\end{align*}
\par
Collecting the above estimates and noticing  the exponential decay of the high frequencies, the desired upper bound of $\mathcal{K}_2$ is given by
\begin{equation}\label{upp-2}
    \langle t \rangle^{-\frac{1}{2}(\frac{d}{2}-\sigma_1+\sigma_2-\sigma_3)}(E_0+\|u^\ell_L,\tau^\ell_L\|_{\tilde 
L^\infty_t(\dot{B}^{\sigma_1+\sigma_3}_{2,\infty})})\tilde{D}_p(t).
  \end{equation}

\underline{\bf Estimate of $\mathcal{K}_3 $}.
Choosing $\sigma'=1$, we have
\begin{align*}
 \|\mathbb{P}(\tilde{u}^\ell\otimes u^\ell_L)\|^\ell_{\dot{B}^{\sigma_1}_{2,\infty}}&\lesssim \| u^\ell_L\|_{\dot{B}^{\sigma_1+\sigma_3}_{2,\infty}}\|\tilde{u}^\ell\|_{\dot{B}^{\frac{d}{2}-\sigma_3}_{2,1}}\\
 &\lesssim \langle t \rangle^{-\frac{1}{2}(\frac{d}{2}-\sigma_1+\sigma_2-\sigma_3)}\| u^\ell_L\|_{\dot{B}^{\sigma_1+\sigma_3}_{2,\infty}},
\end{align*}
where $\sigma_1+\sigma_3<\frac{d}{2}$ and $\sigma_1+\frac{d}{2}\geq 0$.

Similarly, we can obtain
\begin{align*}
\|\mathbb{P}(\tilde{u}^\ell\otimes u^h_L)\|^\ell_{\dot{B}^{\sigma_1}_{2,\infty}}&\lesssim \|  u^h_L\|_{\dot{B}^{\frac{d}{p}-1}_{p,1}}\|\tilde{u}^\ell\|_{\dot{B}^{\frac{d}{2}-1}_{2,1}}
\lesssim  \langle t \rangle^{-\frac{1}{2}(\frac{d}{2}-\sigma_1+\sigma_2-\sigma_3)}E_0, 
\end{align*}
and
\begin{align*}
\|\mathbb{P}(\tilde{u}^h\otimes u_L)\|^\ell_{\dot{B}^{\sigma_1}_{2,\infty}}&\lesssim (\| u^\ell_L\|_{\dot{B}^{\frac{d}{2}-1}_{2,1}}+\| u^h_L\|_{\dot{B}^{\frac{d}{p}}_{p,1}})\|\tilde{u}^h\|_{\dot{B}^{\frac{d}{p}-1}_{p,1}}
\lesssim \langle t \rangle^{-\frac{1}{2}(\frac{d}{2}-\sigma_1+\sigma_2-\sigma_3)}E_0(\tilde{D}_p(t)+1).
\end{align*}

 Using a similar estimate, direct calculation implies
\begin{align*}
\|(\tilde{u}\cdot  \tau_L)\|^\ell_{\dot{B}^{\sigma_1}_{2,\infty}}&\leq \|(\tilde{u}^\ell\cdot  \tau^\ell_L)\|^\ell_{\dot{B}^{\sigma_1}_{2,\infty}}
+\|(\tilde{u}^\ell\cdot  \tau^h_L)\|^\ell_{\dot{B}^{\sigma_1}_{2,\infty}}+\|(\tilde{u}^h\cdot  \tau_L)\|^\ell_{\dot{B}^{\sigma_1}_{2,\infty}}\\
&\lesssim \langle t \rangle^{-\frac{1}{2}(\frac{d}{2}-\sigma_1+\sigma_2-\sigma_3)}(E_0+\|\tau^\ell_L\|_{\dot{B}^{\sigma_1+\sigma_3}_{2,\infty}})(\tilde{D}_p(t) \\
&\quad+  \langle t \rangle^{-\frac{1}{2}(\frac{d}{2}-\sigma_1+\sigma_2-\sigma_3)}E_0.
\end{align*}

Finally, we turn to the term $\|Q(\tilde{\tau},\nabla u_L)\|^\ell_{\dot{B}^{\sigma_1}_{2,\infty}}.$
Note that the typical term of $Q(\tilde{\tau},\nabla u_L)$ can been written by 
$$
\tilde \tau \nabla u_L  =\tau^\ell \nabla \tilde u_L^\ell + \tau^\ell  \nabla \tilde u_L^h +\tau^h \nabla \tilde u_L
$$

Direct calculation reveals
\begin{align*}
\| \tilde{\tau}^\ell \nabla u^\ell_L \|^\ell_{\dot{B}^{\sigma_1-\sigma'}_{2,\infty}}&\lesssim\|\tau^\ell\|_{\dot{B}^{\frac{d}{2}-\frac{\sigma'}{2}}_{2,1}}\|\nabla u^\ell_L\|_{\dot{B}^{\sigma_1-\frac{\sigma'}{2}}_{2,\infty}}\\
&\lesssim \langle t \rangle^{-\frac{1}{2}(\frac{d}{2}+1-\sigma_1-\sigma')}+\langle t \rangle^{-\frac{1}{2}(\frac{d}{2}+1-\sigma_1-\sigma'+\sigma_2)}\tilde{D}_p(t),
\end{align*}
and
\begin{align*}
&\| \tilde{\tau}^\ell \nabla u^h_L \|^\ell_{\dot{B}^{\sigma_1-\sigma'}_{2,\infty}}\leq \|\tilde{\tau}^\ell\|_{\dot{B}^{\frac{d}{2}-1}_{2,1}}\|\nabla u^h_L\|_{\dot{B}^{\frac{d}{p}-1}_{p,1}}+\|\tilde{\tau}^\ell\|_{\dot{B}^{\frac{d}{2}-1}_{2,1}}\|\nabla u^h_L\|_{\dot{B}^{1-\frac{d}{2}}_{2,1}},
\end{align*}
and
\begin{align*}
\| \tilde{\tau}^h \nabla u_L \|^\ell_{\dot{B}^{\sigma_1-\sigma'}_{2,\infty}} &\lesssim (\| u^\ell_L\|_{\dot{B}^{\frac{d}{2}}_{2,1}}+\| u^h_L\|_{\dot{B}^{\frac{d}{p}+1}_{p,1}})\|\tilde{\tau}^h\|_{\dot{B}^{\frac{d}{p}-1}_{p,1}}\\
&\lesssim \langle t \rangle^{-\frac{1}{2}(d-2\sigma_1+\sigma_2)} \tilde{D}_p(t).
\end{align*}

Finally, in order to obtain the upper bound of $\mathcal{K}_3$, we note that
\begin{align*}
&\int^t_0 \langle t-t'\rangle^{-\frac{1}{2}(\sigma-\sigma_1+1)}\langle t \rangle^{-\frac{1}{2}(\frac{d}{2}-\sigma_1+\sigma_2-\sigma_3)}dt'\lesssim 
\langle t \rangle^{-\frac{1}{2}(\sigma-\sigma_1+\sigma_2)},
\end{align*}
where $$
\sigma_2=\min\{\frac{1}{2},(\frac{d}{2}-1-\sigma_1)-,\sigma''\},
$$
 and $\sigma_3=\min\{
            \frac{d}{2}-\sigma, \frac{d}{2}-1-\sigma_1-\sigma_2\}$.

In fact,
in the case of  $\frac{1}{2}(\sigma-\sigma_1+1)\leq \frac{1}{2}(\frac{d}{2}-\sigma_1+\sigma_2-\sigma_3),$
\begin{align*}
&\int^t_0 \langle t-t'\rangle^{-\frac{1}{2}(\sigma-\sigma_1+1)}\langle t \rangle^{-\frac{1}{2}(\frac{d}{2}-\sigma_1+\sigma_2-\sigma_3)}dt'\\
&
\lesssim 
\begin{cases}
\langle t \rangle^{-\frac{1}{2}(\sigma-\sigma_1+1)},~~ \qquad\qquad \qquad \text{if}~\frac{1}{2}(\frac{d}{2}-\sigma_1+\sigma_2-\sigma_3)>1,\\
\langle t \rangle^{-\frac{1}{2}(\sigma-\sigma_1+1)-},~~ \ \qquad\quad\qquad\text{if}~\frac{1}{2}(\frac{d}{2}-\sigma_1+\sigma_2-\sigma_3)=1,\\
\langle t \rangle^{-\frac{1}{2}(\sigma-\sigma_1+\frac{d}{2}-1-\sigma_1-\sigma_3)}, ~~\qquad \text{if}~\frac{1}{2}(\frac{d}{2}-\sigma_1+\sigma_2-\sigma_3)<1
\end{cases}\\
&
\lesssim 
\langle t \rangle^{-\frac{1}{2}(\sigma-\sigma_1+\sigma_2)}
\end{align*}
and
in the case of $\frac{1}{2}(\sigma-\sigma_1+1)\geq\frac{1}{2}(\frac{d}{2}-\sigma_1+\sigma_2-\sigma_3)$,
\begin{align*}
&\int^t_0 \langle t-t'\rangle^{-\frac{1}{2}(\sigma-\sigma_1+1)}\langle t \rangle^{-\frac{1}{2}(\frac{d}{2}-\sigma_1+\sigma_2-\sigma_3)}dt'\\
&
\lesssim 
\begin{cases}
\langle t \rangle^{-\frac{1}{2}(\frac{d}{2}-\sigma_1+\sigma_2-\sigma_3)},~~\qquad \qquad \qquad \text{if}~\frac{1}{2}(\frac{d}{2}-\sigma_1+1)>1,\\
\langle t \rangle^{-\frac{1}{2}(\frac{d}{2}-\sigma_1+\sigma_2-\sigma_3)-},~~\qquad\quad \ \qquad \text{if}~\frac{1}{2}(\frac{d}{2}-\sigma_1+1)=1,\\
\langle t \rangle^{-\frac{1}{2}(\sigma-\sigma_1+\frac{d}{2}-1-\sigma_1-\sigma_3)}, ~~\qquad \qquad \text{if}~\frac{1}{2}(\frac{d}{2}-\sigma_1+1)<1.
\end{cases}\\
&
\lesssim 
\langle t \rangle^{-\frac{1}{2}(\sigma-\sigma_1+\sigma_2)}
\end{align*}
\par

Moreover, we can easily find
\begin{align*}
&\int^t_0 \langle t-t'\rangle^{-\frac{1}{2}(\sigma-\sigma_1+\sigma_1')}\big(\langle t \rangle^{-\frac{1}{2}(\frac{d}{2}+1-\sigma_1-\sigma')}+\langle t \rangle^{-\frac{1}{2}(\frac{d}{2}+1-\sigma_1-\sigma'+\sigma_2)}\\&\qquad+\langle t \rangle^{-\frac{1}{2}(d-2\sigma_1+\sigma_2)}\big)dt'\\&\lesssim 
\langle t \rangle^{-\frac{1}{2}(\sigma-\sigma_1+2\sigma'')}.
\end{align*}
Accordingly we  obtain the desired upper bound of $\mathcal{K}_3$  by
\begin{equation}\label{upp-3}
\langle t\rangle^{-\frac{1}{2}(\sigma-\sigma_1+\sigma'')}\langle t\rangle^{-\frac{\sigma''}{2}}(\tilde{D}^2_p(t)+\tilde{D}_p(t)).  
\end{equation}

Collecting \eqref{upp-1}, \eqref{upp-2} and \eqref{upp-3}, the proof of  Lemma \ref{cl5} is completed. 
\end{proof}

\subsubsection{Bounds for the high frequencies}
In order to obtain the estimates for high frequencies of $(\tilde{u}, \tilde{ \tau})$,  we firstly establish the following lemma .
\begin{lemma}\label{cl10}
Let $p$ satisfy \eqref{ppp} and $\sigma_0\leq \sigma_1<\frac{d}{2}-1.$ It holds that
\begin{align*}
&
\|\langle s \rangle^{\alpha_*}(\nabla \tilde {u}, {\tilde\tau})\|^h_{\tilde{L}^\infty_t(\dot{B}^{\frac{d}{p}}_{p,1})}+\|s^{\alpha_*}{\tilde\tau}\|^h_{\tilde{L}^\infty(1,t;\dot{B}^{\frac{d}{p}+2}_{p,1})}\lesssim
E_0+\|s^{\frac{1}{2}(\frac{d}{2}-\sigma_1)-}u\|_{\tilde{L}^\infty(1,t;\dot{B}^{\frac{d}{p}+1}_{p,1})}+E_0\tilde{D}_p(t). 
\end{align*}
\end{lemma}
\begin{proof}
In order to utilize commutator estimate, we deal with the solution $(u, \tau)$ firstly, and then the proof proceeds in two steps.

{\bf Step I}.
In this part, we will show
\begin{align}\label{14167}
\|\langle s \rangle^{\alpha_*}(\nabla  {u}, { \tau})\|^h_{ \tilde {L}^\infty_t(\dot{B}^{\frac{d}{p}}_{p,1})} \lesssim
E_0+\|s^{\frac{1}{2}(\frac{d}{2}-\sigma_1)-}u\|_{\tilde{L}^\infty(1,t;\dot{B}^{\frac{d}{p}+1}_{p,1})}+E_0\tilde{D}_p(t). 
\end{align}
 Similarly to the argument of    Lemma \ref{lem3.2}, 
  there exists a constant $C>0$  such that   for $k\geq k_0-1$
	\begin{align*} 
		\frac{d}{dt} \widetilde{\mathcal{J}}_k(t)+C(\|\nabla   u_k\|_{L^p}+2^{2k}\|\mathbb{P}\tau_k\|_{L^p}+2^{2k}\|  W_k\|_{L^p})\leq C S_k,
	\end{align*}
where 
\begin{align*}
		\widetilde{\mathcal{J}}_k(t)&= 2\gamma\|\nabla \tilde u_k\|_{L^p}+\gamma\|\mathbb{P} \tau_k\|_{L^p}+\|  W_k\|_{L^p}.\end{align*}
        and \begin{align*}
		S_k= \|R_k\|_{L^p}+\|\bar{R}_k\|_{L^p}+\|H_k\|_{L^p}+2^{-k}\|F_k\|_{L^p},\ \ R_k=[u\cdot\nabla,\nabla\dot{\Delta}_k\mathbb{P} ] u,
        \end{align*}
We then obtain 
\begin{equation}\label{231}
e^{C t}\|(\nabla 
  u_k,\mathbb{P} \tau_k,  W_k)\|_{L^p}\lesssim 
 \int^t_0e^{C s}S_k(s)ds.
\end{equation}
Due to the second equality in \eqref{L1}, we have
	\begin{equation}\label{zz4}
		\|(\nabla   u_k,2^{2k} W_k,2^{2k}\mathbb{P} \tau_k)\|_{L^p}\approx	\|(\nabla   u_k,2^{2k} \tau_k)\|_{L^p}\approx \|(2^k  u_k,2^{2k} \tau_k)\|_{L^p},
	\end{equation}
	for all $k\geq k_0-1$. 
    
Then, multiplying both sides of  \eqref{231} by $\langle t \rangle^{\alpha_*}2^{j{\frac{d}{p}}}$ and employing \eqref{zz4} gives
\begin{align}\label{high fre1}
\|\langle s \rangle^{\alpha_*}(\nabla u, \tau) \|^h_{\tilde{L}^\infty_t(\dot{B}^{\frac{d}{p}}_{p,1})}\lesssim &
 \sum_{k\geq k_0-1}2^{k(d/p-1)}\sup_{0\leq s\leq t}\Bigg(\langle s \rangle^{\alpha_*}\int^s_0e^{-C(s-z)}F_k(z)dz\Bigg)\nonumber\\
&+\sum_{k\geq k_0-1}2^{k(d/p)}\sup_{0\leq s\leq t}\Bigg(\langle s \rangle^{\alpha_*}\int^s_0e^{-C(s-z)}(R_k+H_k)(z)dz\Bigg).
\end{align}

Without loss of generality one can assume that $t>2$. For clarity, we write
\begin{equation*}
    \begin{aligned}
        \mathcal L_1=&\sup_{0\leq s\leq 2}\Bigg(\langle s \rangle^{\alpha_*}\int^s_0e^{-C(s-z)}Z_k(z)dz\Bigg),\\
        \mathcal L_2=&\sup_{2\leq s\leq t}\Bigg(\langle s \rangle^{\alpha_*}\int^1_0e^{-C(s-z)}Z_k(z)dz\Bigg),\\
      \mathcal L_3=&\sup_{2\leq s\leq t}\Bigg(\langle s \rangle^{\alpha_*}\int^s_1e^{-C(s-z)}Z_k(z)dz\Bigg),
    \end{aligned}
\end{equation*}
 where $Z_k(z)=Z_{1k}(z)+Z_{2k}(z)=\big(R_k(z)+H_k(z)\big)+F_k(z)$.

Direct calculation shows 
\begin{align*}
\mathcal L_1 \lesssim\int^2_0Z_{1k}(z)dz+\int^2_0Z_{2k}(z)dz=:\mathcal L_{11}+\mathcal L_{12}
\end{align*}
and
\begin{align*}
\mathcal L_2
\lesssim \int^1_0 Z_{1k}(z)dz+\int^1_0Z_{2k}(z)dz=:\mathcal L_{21}+\mathcal L_{22}.
\end{align*}
On the other hand, one has 
\begin{align*}
\mathcal L_3 
&\lesssim \sup_{2\leqslant s\leqslant t}\sup_{1\leqslant z\leqslant s}(z^{\alpha_*}Z_k(z))\langle s \rangle^{\alpha_*}\int^s_1e^{-C(s-z)}z^{-\alpha_*}dz\\
&\lesssim \sup_{1\leqslant s\leqslant t}(s^{\alpha_*}Z_{1k}(s))+\sup_{1\leqslant s\leqslant t}(s^{\alpha_*}Z_{2k} (s))\\
&=:\mathcal L_{31}+\mathcal L_{32}.
\end{align*}

Using  Lemma \ref{app1} and product law on Besov space, one has
\begin{align*}
&\sum_{k\geq k_0+1}2^{k(\frac{d}{p}-1)}(\mathcal L_{11}+\mathcal L_{21})+\sum_{k\geq k_0+1}2^{\frac{kd}{p}}(\mathcal L_{12}+\mathcal L_{22})\\
&\lesssim \|u\cdot \nabla\tau\|^h_{\dot{B}^{\frac{d}{p}}_{p,1}}+\|\mathrm{Q}(\tau,\nabla u)\|^h_{\dot{B}^{\frac{d}{p}}_{p,1}}+\sum_{k\geq k_0-1}2^{\frac{kd}{p}}\|R_k\|_{L^p}  +\|u\otimes u\|^h_{\dot{B}^{\frac{d}{p}}_{p,1}}\\
& \lesssim\|u\|_{\tilde{L}^\infty_t(\dot{B}^{\frac{d}{p}}_{p,1})}\|\tau\|_{\tilde{L}^1_t(\dot{B}^{\frac{d}{p}+1}_{p,1})}+\|\tau\|_{\dot{B}^{\frac{d}{p}}_{p,1}}\|u\|_{\dot{B}^{\frac{d}{p}+1}_{p,1}} + \|u\|_{\dot{B}^{\frac{d}{p}}_{p,1}}\|u\|_{\dot{B}^{\frac{d}{p}}_{p,1}}+\|u\|^2_{\dot{B}^{\frac{d}{p}+1}_{p,1}}\\
&\lesssim E^2_0(t).
\end{align*}

With the help of 
\begin{align*}
\|\langle s \rangle^{\frac{1}{2}(\frac{d}{2}-\sigma_1)-}u_L\|_{\tilde{L}^\infty_t(\dot{B}^{\frac{d}{p}+1}_{p,1})}\lesssim 1+E_{0},\ \ \|\langle s \rangle^{\alpha_*}\tilde{u}\|_{\tilde{L}^\infty_t(\dot{B}^{\frac{d}{p}+1}_{p,1})}\lesssim \tilde{D}_p(t)
\end{align*}
and  $\alpha_*< \frac{d}{2}-\sigma_1$, direct computation shows that
\begin{align*}
&\sum_{k\geq k_0-1}2^{\frac{kd}{p}}\sup_{1\leqslant s\leqslant t}(s^{\alpha_*}Z_{1k}(s))\\
&\lesssim \|s^{\alpha_*} u\cdot \nabla\tau\|^h_{\dot{B}^{\frac{d}{p}}_{p,1}}+\|s^{\alpha_*}\mathrm{Q}(\tau,\nabla u)\|^h_{\dot{B}^{\frac{d}{p}}_{p,1}}+\sum_{k\geq k_0-1}2^{\frac{d}{p}k}\|s^{\alpha_*}R_k\|_{L^p}\\
&\lesssim \|s^{\alpha_*}(u\cdot \nabla \tau_L)\|^h_{\tilde{L}^\infty(1,t;\dot{B}^{\frac{d}{p}}_{p,1})}+\|s^{\alpha_*}(u\cdot \nabla \tilde{\tau})\|^h_{\tilde{L}^\infty(1,t;\dot{B}^{\frac{d}{p}}_{p,1})}\\
&\quad+\|s^{\alpha_*}(\tau_L\cdot \nabla u)\|^h_{\tilde{L}^\infty(1,t;\dot{B}^{\frac{d}{p}}_{p,1})}+\|s^{\alpha_*}(\tilde{\tau}\cdot \nabla u )\|^h_{\tilde{L}^\infty(1,t;\dot{B}^{\frac{d}{p}}_{p,1})}\\
&\quad+\|s^{\alpha_*}u\|_{\tilde{L}^\infty(1,t;\dot{B}^{\frac{d}{p}+1}_{p,1})} (\|u\|^l_{\tilde{L}^\infty(1,t;\dot{B}^{\frac{d}{2}+1}_{2,1})}+\|u\|^h_{\tilde{L}^\infty(1,t;\dot{B}^{\frac{d}{p}+1}_{p,1})})\\
&\lesssim E_0\tilde{D}_p(t)+\|\langle s \rangle^{\frac{1}{2}(\frac{d}{2}-\sigma_1)-}\tau_L\|_{\tilde{L}^\infty_t(\dot{B}^{\frac{d}{p}+1}_{p,1})}\|s^{\frac{1}{2}(\frac{d}{2}-\sigma_1)-}u\|_{\tilde{L}^\infty(1,t;\dot{B}^{\frac{d}{p}}_{p,1})}\\
&\quad+\|u\|_{\tilde{L}^\infty(1,t;\dot{B}^{\frac{d}{p}}_{p,1})}\| s ^{\alpha_*}\tilde{\tau}\|_{\tilde{L}^\infty(1,t;\dot{B}^{\frac{d}{p}+1}_{p,1})}
+\|s^{\alpha_*}u\|_{\tilde{L}^\infty(1,t;\dot{B}^{\frac{d}{p}+1}_{p,1})}\|\tau_L\|_{\tilde{L}^\infty(1,t;\dot{B}^{\frac{d}{p}}_{p,1})}\\
&\quad+\|u\|_{\tilde{L}^\infty(1,t;\dot{B}^{\frac{d}{p}+1}_{p,1})}\|\langle s \rangle^{\alpha_*}\tilde{\tau}\|_{\tilde{L}^\infty_t(\dot{B}^{\frac{d}{p}}_{p,1})}\\
&\lesssim \|s^{\frac{1}{2}(\frac{d}{2}-\sigma_1)-}u\|_{\tilde{L}^\infty(1,t;\dot{B}^{\frac{d}{p}+1}_{p,1})}+E_0\tilde{D}_p(t)
\end{align*}
and
\begin{align*}
&\sum_{k\geq j_0-1}2^{k(\frac{d}{p}-1)}\sup_{1\leqslant s\leqslant t}(s^{\alpha_*}Z_{2k}(s))\\
&\lesssim \|s^{\alpha_*}(u\otimes u)\|^h_{\tilde{L}^\infty(1,t;\dot{B}^{\frac{d}{p}}_{p,1})}\\
&\lesssim \|s^{\alpha_*}(u_L\otimes u)\|^h_{\tilde{L}^\infty(1,t;\dot{B}^{\frac{d}{p}}_{p,1})}+\|s^{\alpha_*}(\tilde{u}\otimes u)\|^h_{\tilde{L}^\infty(1,t;\dot{B}^{\frac{d}{p}}_{p,1})}\\
&\lesssim \|s^{\frac{1}{2}(\frac{d}{2}-\sigma_1)-}u\|_{\tilde{L}^\infty(1,t;\dot{B}^{\frac{d}{p}}_{p,1})}+E_0\tilde{D}_p(t).
\end{align*}

As a result, we complete the estimate of \eqref{14167}.

 {\bf Step II}.
In this part, we focus on  the high frequencies of $\tau$,
which depends on the parabolic maximal regularity for the parabolic system (see 
(3.39) in \cite{BCD-2011}):
\begin{align*}
\begin{cases}
\partial_t(t \tau)-\Delta(t \tau)= \tau+tD(  u)+tH,\\
t \tau|_{t=0}=0.
\end{cases}
\end{align*}

For   $0\leq t\leq1$,  we have 
\begin{align*}
\|t \tau\|^h_{\tilde{L}^\infty(0,1;\dot{B}^{\frac{d}{p}+2}_{p,1})}\lesssim& \|( \tau,\nabla  u)\|^h_{\tilde{L}^\infty(0,1;\dot{B}^{\frac{d}{p}}_{p,1})}+\|tH\|^h_{\tilde{L}^\infty(0,1;\dot{B}^{\frac{d}{p}}_{p,1})}\\
\lesssim& E_0+\| u\|_{\tilde{L}^\infty(0,1;\dot{B}^{\frac{d}{p}}_{p,1})}\|t\nabla \tau\|^h_{\tilde{L}^\infty(0,1;\dot{B}^{\frac{d}{p}}_{p,1})}+\|\nabla u\|_{\tilde{L}^\infty(0,1;\dot{B}^{\frac{d}{p}}_{p,1})}\|t \tau\|^h_{\tilde{L}^\infty(0,1;\dot{B}^{\frac{d}{p}}_{p,1})},
\end{align*}
which implies that
\begin{align}
\sup_{t\in [0,1]}\|t  \tau\|^h_{\dot{B}^{\frac{d}{p}+2}_{p,1}}\lesssim E_0.\label{tau-111}
\end{align}

For  $t\geq 1$, we consider the following system
\begin{align*}
\partial_t(t^{\alpha_*} \tau)-\Delta(t^{\alpha_*} \tau)=\alpha_*t^{\alpha_*-1} \tau+t^{\alpha_*}D(  u)+t^{\alpha_*}H.
\end{align*}
Based on \eqref{tau-111} we then have
\begin{align*}
\|s^{\alpha_*} \tau\|^h_{\tilde{L}^\infty(1,t;\dot{B}^{\frac{d}{p}+2}_{p,1})}&\lesssim\| \tau(1)\|^h_{\dot{B}^{\frac{d}{p}}_{p,1}}
+\|s^{\alpha_*}D(  u)\|^h_{\tilde{L}^\infty(1,t;\dot{B}^{\frac{d}{p}}_{p,1})}+\|s^{\alpha_*}H\|^h_{\tilde{L}^\infty(1,t;\dot{B}^{\frac{d}{p}}_{p,1})}\\
&\lesssim\| \tau(1)\|^h_{\dot{B}^{\frac{d}{p}}_{p,1}} +\|s^{\frac{1}{2}(\frac{d}{2}-\sigma_1)-} u\|_{\tilde{L}^\infty(1,t;\dot{B}^{\frac{d}{p}+1}_{p,1})}+E_0\tilde{D}_p(t)\\
&\lesssim E_0 +\|s^{\frac{1}{2}(\frac{d}{2}-\sigma_1)-}  u\|_{\tilde{L}^\infty(1,t;\dot{B}^{\frac{d}{p}+1}_{p,1})} +E_0\tilde{D}_p(t)
\end{align*}
which together with \eqref{14167}, Proposition \ref{prop1} and $(\tilde{u}, \tilde\tau)=( {u},  \tau)-({u}_L,  \tau_L)$  lead to the desired conclusion.
\end{proof}
\begin{lemma}\label{cl11}
Let $p$ satisfy \eqref{ppp}, then there exists large enough $t_L>0$  such that  
\begin{align*}
&
\|s^{\frac{1}{2}(\frac{d}{2}-\sigma_1)-}u\|_{\tilde{L}^\infty(1,t;\dot{B}^{\frac{d}{p}+1}_{p,1})}\lesssim
1+E_0+t_L^{-\frac{\sigma_2}{4}}\tilde{D}_p(t). 
\end{align*}
\end{lemma}
\begin{proof}
Recall that  $u=u_L+\tilde{u}$, direct calculation and \eqref{wi1}  yield
\begin{align*}
\|s^{\frac{1}{2}(\frac{d}{2}-\sigma_1)-}u\|_{\tilde{L}^\infty(1,t;\dot{B}^{\frac{d}{p}+1}_{p,1})}=&\|s^{\frac{1}{2}(\frac{d}{2}-\sigma_1)-}u_L\|_{\tilde{L}^\infty(1,t;\dot{B}^{\frac{d}{p}+1}_{p,1})}
+\|s^{\frac{1}{2}(\frac{d}{2}-\sigma_1)-}\tilde{u}\|_{\tilde{L}^\infty(1,t;\dot{B}^{\frac{d}{p}+1}_{p,1})}\\
\lesssim& \|s^{\frac{1}{2}(\frac{d}{2}-\sigma_1)-}u^\ell_L\|_{\tilde{L}^\infty(1,t;\dot{B}^{\frac{d}{p}+1}_{p,1})}+\|s^{\frac{1}{2}(\frac{d}{2}-\sigma_1)-}u^h_L\|_{\tilde{L}^\infty(1,t;\dot{B}^{\frac{d}{p}+1}_{p,1})}\\
&+\|s^{\frac{1}{2}(\frac{d}{2}-\sigma_1)-}\tilde{u}\|_{\tilde{L}^\infty(1,t;\dot{B}^{\frac{d}{p}+1}_{p,1})}\\
\lesssim& 1+E_0+\|s^{\frac{1}{2}(\frac{d}{2}-\sigma_1)-}\tilde{u}^\ell\|_{\tilde{L}^\infty(1,t;\dot{B}^{\frac{d}{2}-}_{2,1})}+\|s^{\frac{1}{2}(\frac{d}{2}-\sigma_1)-}\tilde{u}^h\|_{\tilde{L}^\infty(1,t;\dot{B}^{\frac{d}{p}+1}_{p,1})}
\\ \lesssim& E_0+t_L^{-\frac{\sigma_2}{4}} \tilde{D}_p(t)+1,
\end{align*}
where the last inequality comes from 
\begin{align*}
\|s^{\frac{1}{2}(\frac{d}{2}-\sigma_1)-}\tilde{u}^h\|_{\tilde{L}^\infty(1,t;\dot{B}^{\frac{d}{p}+1}_{p,1})}&\lesssim t_L^{-\frac{\sigma_2}{4}}\|\langle s \rangle^{\alpha_*}(\nabla {\tilde{u}}, {\tilde{\tau}})\|^h_{\tilde{L}^\infty_{t_L,t}(\dot{B}^{\frac{d}{p}}_{p,1})}+\|\tilde{u}^h\|_{\tilde{L}^\infty(\dot{B}^{\frac{d}{p}+1}_{p,1})}\\&
\lesssim t_L^{-\frac{\sigma_2}{4}}\tilde{D}_p(t)+E_0+1.
\end{align*}
Then the proof of Lemma \ref{cl11} is completed.
\end{proof}
\subsubsection{The proof of the Proposition \ref{cob2}}
\label{ddd}
\begin{proof}
Using the interpolation inequality and \eqref{negative-bosov-of-utau} to imply
$$
\|(u_L,\tau_L,u,\tau)^\ell\|_{\dot{B}^{\sigma_1+\sigma_2}_{2,\infty}}\lesssim \|(u_L,\tau_L,u,\tau)^\ell\|^{\theta}_{\dot{B}^{\sigma_1}_{2,\infty}}
\|(u_L,\tau_L,u,\tau)^\ell\|^{1-\theta}_{\dot{B}^{\frac{d}{2}}_{2,\infty}}\lesssim C^{\theta}E_0^{1-\theta}
$$
for some $\theta\in (0,1)$.

Collecting   Lemma \ref{cl1}-\ref{cl11}, we obtain that
\begin{align*}
\tilde{D}_p(t)&\lesssim  E_0 +1+E_0\tilde{D}_p(t)+ \langle t_L\rangle^{-\frac{\sigma''}{4}} \tilde{D}^2_p(t)+ t_L^{-\frac{\sigma_2}{4}}\tilde{D}_p(t).
\end{align*}
Noting that $\sigma''$ is independent of the $\sigma$ and 
  large $t_L$  and using the standard boothstrap argument yields
 $$
 \tilde{D}_p(t)\lesssim  1+E_0 .
 $$
Thus we end the proof of Proposition \ref{cob2}.
\end{proof}

\subsection{Necessary condition}
\label{sec:necessary}

{ 
This part is devoted to establishing the necessary condition in Theorem 
\ref{Thm}, namely the ``only if'' part of the decay characterization. The 
central result is Proposition \ref{pro3.1}, which shows that if the solution 
$(u,\tau)$ to the Cauchy problem \eqref{OB-1} satisfies the upper decay bound 
\eqref{Theo1} and the two-sided bound \eqref{Theo2}, then the difference 
$(\tilde{u},\tilde{\tau}) = (u - u_L, \tau - \tau_L)$ between the nonlinear 
solution and the linear solution enjoys a strictly faster decay rate at low 
frequencies, namely
\begin{equation*}
	\|(\tilde{u},\tilde{\tau})\|^\ell_{\dot{B}^\sigma_{2,1}} \lesssim 
	\langle t \rangle^{-\frac{1}{2}(\sigma - \sigma_1 + \sigma_2)}, 
	\quad \sigma_1 < \sigma \leq \tfrac{d}{2},\ t > 0,
\end{equation*}
where $\sigma_2 > 0$ is the gain exponent defined in Proposition \ref{cob2}. 
The proof of Proposition \ref{pro3.1} proceeds by Duhamel's principle applied 
to the error system \eqref{OBC-4}, combined with the nonlinear estimates 
developed in Section \ref{sec:sufficient}. The key point is that the nonlinear 
interactions $F = -\mathbb{P}(u \cdot \nabla u)$ and $H = -u \cdot \nabla\tau 
- Q(\tau, \nabla u)$ decay faster than the linear solution, and this faster 
decay propagates to $(\tilde{u},\tilde{\tau})$ through the convolution 
inequality \eqref{key inequality 1}.

\medskip 
Once Proposition \ref{pro3.1} is established, the necessary condition follows 
in Subsection \ref{ddd} by a comparison argument. Specifically, the low-frequencies 
norm of the linear solution $(u_L,\tau_L)$ is controlled from above and 
below by
\begin{equation*}
	 \|(u,\tau)^\ell(t)\|_{\dot{B}^\sigma_{2,1}} 
	- \|(\tilde{u},\tilde{\tau})^\ell(t)\|_{\dot{B}^\sigma_{2,1}}\leq\|(u_L,\tau_L)^\ell(t)\|_{\dot{B}^\sigma_{2,1}} 
	\leq \|(u,\tau)^\ell(t)\|_{\dot{B}^\sigma_{2,1}} 
	+ \|(\tilde{u},\tilde{\tau})^\ell(t)\|_{\dot{B}^\sigma_{2,1}},
\end{equation*}
and the faster decay of $(\tilde{u},\tilde{\tau})$ established in Proposition 
\ref{pro3.1} ensures that the two-sided bound on $(u,\tau)$ transfers to 
$(u_L,\tau_L)$. Invoking the decay characterization \eqref{wi3} of Proposition 
\ref{prop1} then yields the desired conclusion $(u_0,\tau_0)^\ell \in 
\dot{\mathbf{B}}^{\sigma_1}_{2,\infty}$, completing the proof of Theorem 
\ref{Thm}.
}

\begin{proposition}\label{pro3.1}
Let $\sigma_0\leq \sigma_1<\frac{d}{2}-1$. Assume  $(u,\tau)$  be the solution to the Cauchy problem \eqref{OB-1} and satisfy \eqref{Theo1} and \eqref{Theo2}, then $(\tilde{u}, \tilde{\tau})$ has faster decay rates at low frequencies 
\begin{align*}
\|(\tilde{u}, \tilde{\tau})\|^\ell_{\dot{B}^\sigma_{2,1}}\lesssim \langle t \rangle^{-\frac{1}{2}(\sigma-\sigma_1+\sigma_2)}
\end{align*}
for $\sigma_1<\sigma\leq \frac{d}{2}$ and $t>0$, where $\sigma_2$ is defined in Proposition \ref{cob2}.
\end{proposition}
\begin{proof}
Recall that
\begin{align*}
\|(\tilde{u},\tilde{\tau})(t)\|^\ell_{\dot{B}^{\sigma}_{2,1}}\lesssim \int^t_0\langle t-t'\rangle^{-\frac{1}{2}(\sigma-\sigma_1+\sigma')}\|F,H\|^\ell_{\dot{B}^{\sigma_1-\sigma'}_{2,\infty}}dt',~~\sigma>\sigma_1.
\end{align*}

For $0<t\leq t_0$ with $t_0>0$ defined in Theorem \ref{Thm}, noticing that $\langle t \rangle\approx 1$, $\langle t-s\rangle\approx 1$ for $0\leq s\leq t\leq t_0$, we then obtain
\begin{align*}
\|F,H\|^\ell_{\dot{B}^{\sigma_1-\sigma'}_{2,\infty}}\lesssim E^2_0+E_0\|(u^\ell,\tau^\ell)\|_{\dot{B}^{\sigma_1}_{2,\infty}}.
\end{align*}
Hence,  Proposition \ref{bdd-negative-norm} and direct calculation imply
\begin{align*}
\int^t_0\langle t-t'\rangle^{-\frac{1}{2}(\sigma-\sigma_1+\sigma')}\|F,H\|^\ell_{\dot{B}^{\sigma_1-\sigma'}_{2,\infty}}dt'\lesssim \langle t\rangle^{-\frac{1}{2}(\sigma-\sigma_1+\sigma_2)}.
\end{align*}

For $t>t_0$, we proceed with two cases $\sigma_1<\sigma\leq \frac{d}{2}-1$ and $\frac{d}{2}-1<\sigma\leq \frac{d}{2}$.

\underline{\bf Case I}.
In the case of $\sigma_1<\sigma\leq \frac{d}{2}-1$, one has 
\begin{align*}
 &\int^t_0\langle t-t'\rangle^{-\frac{1}{2}(\sigma-\sigma_1+\sigma')}\|F,H\|^\ell_{\dot{B}^{\sigma_1-\sigma'}_{2,\infty}}dt'\\
 &=\int^{t_0}_0\langle t-t'\rangle^{-\frac{1}{2}(\sigma-\sigma_1+\sigma')}\|F,H\|^\ell_{\dot{B}^{\sigma_1-\sigma'}_{2,\infty}}dt'+\int^t_{t_0}\langle t-t'\rangle^{-\frac{1}{2}(\sigma-\sigma_1+\sigma')}\|F,H\|^\ell_{\dot{B}^{\sigma_1-\sigma'}_{2,\infty}}dt'.
\end{align*}

 Direct calculation shows
\begin{align*}
&\int^{t_0}_0\langle t-t'\rangle^{-\frac{1}{2}(\sigma-\sigma_1+\sigma')}\|F,H\|^\ell_{\dot{B}^{\sigma_1-\sigma'}_{2,\infty}}dt'\lesssim \langle t\rangle^{-\frac{1}{2}(\sigma-\sigma_1+\sigma_2)}.
\end{align*}

On the other hand, we choose $\sigma'=1$ and then obtain
\begin{align*}
\|(F, u\cdot \nabla \tau)\|^\ell_{\dot{B}^{\sigma_1-1}_{2,\infty}}&\lesssim (\|u^\ell\|_{\dot{B}^{\frac{d}{2}-1}_{2,1}}+\|u^h\|_{\dot{B}^{\frac{d}{p}}_{p,1}})\|u^h\|_{\dot{B}^{\frac{d}{p}-1}_{p,1}}+\|u^\ell\|_{\dot{B}^{\frac{d}{2}-1}_{2,1}}\|u^\ell\|_{\dot{B}^{\sigma_1+1}_{2,1}}\\
&\quad+\|u^h\|_{\dot{B}^{\frac{d}{p}-1}_{p,1}}\|u^\ell\|_{\dot{B}^{\frac{d}{2}-1}_{2,1}}+\|u^\ell\|_{\dot{B}^{\frac{d}{2}-1}_{2,1}}\|\tau^\ell\|_{\dot{B}^{\sigma_1+1}_{2,1}}+\|u^h\|_{\dot{B}^{\frac{d}{p}-1}_{p,1}}\|\tau^\ell\|_{\dot{B}^{\frac{d}{2}-1}_{2,1}}\\
&\quad+
(\|u^\ell\|_{\dot{B}^{\frac{d}{2}-1}_{2,1}}+\|u^h\|_{\dot{B}^{\frac{d}{p}}_{p,1}})\|\tau^h\|_{\dot{B}^{\frac{d}{p}-1}_{p,1}}\\
&\lesssim \langle t\rangle^{-\frac{1}{2}(\frac{d}{2}-\sigma_1)}.
\end{align*}

Next we deal with the term $Q(\tau, \nabla u)$. 
To proceed we write the typical term by
    \begin{equation*}
        \tau \nabla u=  \tau^\ell \nabla u^\ell+ \tau^h \nabla u^\ell + \tau^\ell  \nabla u^h+ \tau^h \nabla u^h,
    \end{equation*}
    then we bound the  corresponding  four terms respectively. 

Fristly, by using \eqref{A.2} we deduce 
\begin{align*}
\|\tau^\ell \nabla u^\ell\|^\ell_{\dot{B}^{\sigma_1-\sigma'}_{2,\infty}}\lesssim \|\tau^\ell\|_{\dot{B}^{\frac{d}{2}-\frac{\sigma'}{2}}_{2,1}}\|\nabla u^\ell\|_{\dot{B}^{\sigma_1-\frac{\sigma'}{2}}_{2,\infty}},
\end{align*}
and
\begin{align*}
\| {\tau}^h   \nabla {u}^\ell\|^\ell_{\dot{B}^{\sigma_1-\sigma'}_{2,\infty}}\lesssim \|\nabla\tilde{u}^\ell\|_{\dot{B}^{\frac{d}{2}-1}_{2,1}}\|\tau^h\|_{\dot{B}^{\frac{d}{p}-1}_{p,1}}.
\end{align*}
Using a similar argument in \eqref{aaa1}, we have
\begin{align*}
\|\tau^\ell\nabla u^h\|^\ell_{\dot{B}^{\sigma_1-\sigma'}_{2,\infty}}\leq  \|{\tau}^\ell \cdot {u}^h\|^\ell_{\dot{B}^{\sigma_1-\sigma'+1}_{2,\infty}}+\|\nabla{\tau}^\ell\|_{\dot{B}^{\frac{d}{2}-1}_{2,1}}\|{u}^h\|_{\dot{B}^{\frac{d}{p}-1}_{p,1}}
\end{align*}
and 
\begin{align*}
 \|  \tau^h \cdot \nabla {u}^h \|^\ell_{\dot{B}^{\sigma_1-\sigma'}_{2,\infty}}\leq \|u^h\|_{\dot{B}^{\frac{d}{p}+1}_{p,1}}\|\tau^h\|_{\dot{B}^{\frac{d}{p}+1}_{p,1}}.
\end{align*}

     Therefore,  let $\sigma'=2\sigma''$ be small  enough  such that $\frac{d}{2}+1-\sigma_1-\sigma'>2$ and employ \eqref{key inequality 1}, we then arrive at 
\begin{align*}
&\int^t_{t_0}\langle t-t'\rangle^{-\frac{1}{2}(\sigma-\sigma_1+\sigma')}\|Q(\tau,\nabla u)\|^\ell_{\dot{B}^{\sigma_1-\sigma'}_{2,\infty}}dt\\ & \lesssim \int^t_{t_0}\langle t-t'\rangle^{-\frac{1}{2}(\sigma-\sigma_1+\sigma')}\langle t \rangle^{-\frac{1}{2}(\frac{d}{2}+1-\sigma_1-\sigma')}\\
&\ \ \ \ +\int^t_{t_0}\langle t-t'\rangle^{-\frac{1}{2}(\sigma-\sigma_1+\sigma')}\langle t \rangle^{-\frac{1}{2}(d-2\sigma_1)}+\int^t_{t_0}\langle t-t'\rangle^{-\frac{1}{2}(\sigma-\sigma_1+1)}\langle t \rangle^{-\frac{1}{2}(\frac{d}{2}-\sigma_1)}\\
&\lesssim \langle t\rangle^{-\frac{1}{2}(\frac{d}{2}-\sigma_1)} .
\end{align*}
 
\underline{\bf Case II}. In the case of $\frac{d}{2}-1<\sigma\leq \frac{d}{2},$
  let $\sigma'=\sigma_2$ and we can calculate
\begin{align*}
&\|\mathbb{P}(u^\ell\otimes u^\ell)\|^\ell_{\dot{B}^{\sigma_1+1-\sigma_2}_{2,\infty}}+\|u^\ell\otimes {\tau}^\ell\|^\ell_{\dot{B}^{\sigma_1+1-\sigma_2}_{2,\infty}}\\
&\lesssim \|u^\ell\|_{\dot{B}^{\sigma_1+1-\sigma_2}_{2,\infty}}\|u^\ell\|_{\dot{B}^{\frac{d}{2}}_{2,\infty}}+ \|u^\ell\|_{\dot{B}^{\sigma_1+1-\sigma_2}_{2,\infty}}\|\tau^\ell\|_{\dot{B}^{\frac{d}{2}}_{2,1}}\\
&\lesssim \langle t \rangle^{\frac{d}{2}-\sigma_1+1-\sigma_2}(1+E_0),
\end{align*}
Due to  $\sigma_2=\min\{\frac{1}{2},(\frac{d}{2}-1-\sigma_1),\sigma''\}$, which indicates that $\sigma-\sigma_1+\sigma_2\leq \frac{d}{2}-\sigma_1+1-\sigma_2$, $\frac{d}{2}-\sigma_1+1-\sigma_2>2$, then one gets
\begin{align*}
\int^t_0 \langle t-t'\rangle^{-\frac{1}{2}(\sigma-\sigma_1+1)}\langle t \rangle^{-\frac{1}{2}(\frac{d}{2}-\sigma_1+1-\sigma_2)}dt'\lesssim \langle t\rangle^{-\frac{1}{2}(\sigma-\sigma_1+\sigma_2)}.
\end{align*}

For the remaining terms,  direct computations yield that 
\begin{align*}
&\|\mathbb{P}(u^h\cdot \nabla u^\ell)\|^\ell_{\dot{B}^{\sigma_1-1}_{2,\infty}}+\|\mathbb{P}(u\cdot \nabla u^h)\|^\ell_{\dot{B}^{\sigma_1-1}_{2,\infty}} +\|u^h\cdot \nabla \tau^\ell\|^\ell_{\dot{B}^{\sigma_1-1}_{2,\infty}}+\|u\cdot \nabla \tau^h\|^\ell_{\dot{B}^{\sigma_1-1}_{2,\infty}}\\
&\lesssim  \|u^h\|_{\dot{B}^{\frac{d}{p}-1}_{p,1}}\|u^\ell\|_{\dot{B}^{\frac{d}{2}-1}_{2,1}}+(\|u^\ell\|_{\dot{B}^{\frac{d}{2}-1}_{2,1}}+\|u^h\|_{\dot{B}^{\frac{d}{p}}_{p,1}})\|u^h\|_{\dot{B}^{\frac{d}{p}-1}_{p,1}}\\
&~~~\qquad+\|u^h\|_{\dot{B}^{\frac{d}{p}-1}_{p,1}}\|\tau^\ell\|_{\dot{B}^{\frac{d}{2}-1}_{2,1}}+
(\|u^\ell\|_{\dot{B}^{\frac{d}{2}-1}_{2,1}}+\|u^h\|_{\dot{B}^{\frac{d}{p}}_{p,1}})\|\tau^h\|_{\dot{B}^{\frac{d}{p}-1}_{p,1}}
\\
&\lesssim \langle t \rangle^{-\frac{1}{2}(d-1-2\sigma_1)}.
\end{align*}

 To obtain the desired upper bound, we need the following estimate:
 \begin{equation}\label{1732}
     \int^t_0 \langle t-t'\rangle^{-\frac{1}{2}(\sigma-\sigma_1+1)}\langle t \rangle^{-\frac{1}{2}(d-1-2\sigma_1)}dt' \lesssim \langle t \rangle^{-\frac{1}{2}(\sigma-\sigma_1+\sigma_2)}, \  \sigma\in (\frac{d}{2}-1,\frac{d}{2}] .
 \end{equation}
In fact, \eqref{1732} comes from the following two cases.
For the case of $\frac{1}{2}(\sigma-\sigma_1+1)\leq \frac{1}{2}(d-1-2\sigma_1)$, we have
\begin{align*}
&\int^t_0 \langle t-t'\rangle^{-\frac{1}{2}(\sigma-\sigma_1+1)}\langle t \rangle^{-\frac{1}{2}(d-1-2\sigma_1)}dt'\\
&\lesssim 
\begin{cases}
\langle t \rangle^{-\frac{1}{2}(\sigma-\sigma_1+1)},~~\qquad \qquad \text{if}~\frac{1}{2}(d-1-2\sigma_1)>1,\\
\langle t \rangle^{-\frac{1}{2}(\sigma-\sigma_1+1)-},~~\qquad\quad \ ~ \text{if}~\frac{1}{2}(d-1-2\sigma_1)=1,\\
\langle t \rangle^{-\frac{1}{2}(\sigma-\sigma_1+d-2-2\sigma_1)},\ \quad \text{if}~\frac{1}{2}(d-1-2\sigma_1)<1,
\end{cases}\\
&\lesssim \langle t \rangle^{-\frac{1}{2}(\sigma-\sigma_1+\sigma_2)},
\end{align*}
and for the case of $\frac{1}{2}(\sigma-\sigma_1+1)> \frac{1}{2}(d-1-2\sigma_1)$, we have
\begin{align*}
&\int^t_0 \langle t-t'\rangle^{-\frac{1}{2}(\sigma-\sigma_1+1)}\langle t \rangle^{-\frac{1}{2}(d-1-2\sigma_1)}dt'\\
&
\lesssim  
\begin{cases}
\langle t \rangle^{-\frac{1}{2}(d-1-2\sigma_1)},~~\qquad \qquad \text{if}~\frac{1}{2}(\sigma-\sigma_1+1)>1,\\
\langle t \rangle^{-\frac{1}{2}(d-1-2\sigma_1)-},~~\qquad \quad \ \text{if}~\frac{1}{2}(\sigma-\sigma_1+1)=1,\\
\langle t \rangle^{-\frac{1}{2}(\sigma-\sigma_1+d-2-2\sigma_1)}, ~~\ ~ \quad\text{if}~\frac{1}{2}(\sigma-\sigma_1+1)<1,
\end{cases}\\
&\lesssim \langle t \rangle^{-\frac{1}{2}(\sigma-\sigma_1+\sigma_2)}.
\end{align*}

Thus the proof of Proposition \ref{pro3.1} is completed.
\end{proof}

\textit{Completeness of the proof of  Theorem \ref{Thm}.}
According to the Proposition \ref{pro3.1},  $(\tilde{u},\tilde{\tau})$ exhibits the faster decay. Consequently, we obtain
\begin{align*}
\|(u_L,\tau_L)^\ell(t)\|_{\dot{B}^{\sigma}_{2,1}}&\lesssim 
\|(u,\tau)^\ell(t)\|_{\dot{B}^{\sigma}_{2,1}}+\|(\tilde{u},\tilde{\tau})^\ell(t)\|_{\dot{B}^{\sigma}_{2,1}}\\
&\leq \langle t \rangle^{-\frac{1}{2}(\sigma-\sigma_1)}+\langle t \rangle^{-\frac{1}{2}(\sigma-\sigma_1+\sigma_2)}\\
&\lesssim \langle t \rangle^{-\frac{1}{2}(\sigma-\sigma_1)},~~~t>t_0,~\sigma_1<\sigma\leq \frac{d}{2}.
\end{align*}
\par
On the other hand, a direct computation yields the following lower bound
\begin{align*}
\|(u_L,\tau_L)^\ell(t)\|_{\dot{B}^{\sigma}_{2,1}}&\geq 
\|(u,\tau)^\ell(t)\|_{\dot{B}^{\sigma}_{2,1}}-\|(\tilde{u},\tilde{\tau})^\ell(t)\|_{\dot{B}^{\sigma}_{2,1}}\\
&\gtrsim  C \langle t \rangle^{-\frac{1}{2}(\sigma-\sigma_1)}-\langle t \rangle^{-\frac{1}{2}(\sigma-\sigma_1+\sigma_2)}\\
&\gtrsim  \langle t \rangle^{-\frac{1}{2}(\sigma-\sigma_1)},~~~t>t_0,~\sigma_1<\sigma\leq \frac{d}{2}.
\end{align*}
In view of \eqref{wi3} in Proposition \ref{prop1}, we have $(u_0,\tau_0)^\ell\in\dot{\textbf{B}}^{\sigma_1}_{2,\infty} $.

Moreover, if the initial data $(u_0, \tau_0)$ satisfies $(u_0,{\mathbb{P}\Lambda^{-1}\mathrm{div}\tau}_0)^\ell\in \dot{\bf{B}}^{\sigma_1}_{2,\infty}$ and $\tau_0^\ell\in \dot{{B}}^{\sigma_1}_{2,\infty}$,  then by \eqref{wi4}  and  Proposition \ref{cob2}, we obtain the desired bound \eqref{6127}.

Consequently, we complete the proof of Theorem \ref{Thm}.

\vskip .2in 
\appendix
\section{Basic facts}

In this section, we recall some basic facts.
\begin{lemma}
    For any $t\geq 0$, we have
\begin{align}\label{key inequality 1}
\int^t_0\langle t-s\rangle^{-\gamma_1}\langle s \rangle^{-\gamma_2}ds=
\begin{cases}
\langle t \rangle^{-\gamma_1},~~\qquad\ if~~\gamma_2>1,\\
\langle t \rangle^{-\gamma_1-}~~\qquad\ if~~\gamma_2=1,\\
\langle t\rangle^{-\gamma_1-\gamma_2+1},~~if~~\gamma_2<1.
\end{cases}
\end{align}
for $0 \leq \gamma_1 \leq \gamma_2$.

\end{lemma}
 
\begin{lemma}[\cite{Xin-Xu-2021}]\label{lemma12}
	There exists a universal integer $N_0$ such that for $2\leq p\leq4$ and $s>0$, we have
	\begin{align*}
	&	\|FG^h\|^\ell_{\dot{B}^{\sigma_0}_{2,\infty}}\lesssim(\|F\|_{\dot{B}^{s}_{p,1}}+\|\dot{S}_{k_0+N_0}F\|_{L^{p^*}})\|G^h\|_{\dot{B}^{-s}_{p,\infty}},
\\
	&	\|F^hG\|^\ell_{\dot{B}^{\sigma_0}_{2,\infty}}\lesssim(\|F^h\|_{\dot{B}^{s}_{p,1}}+\|\dot{S}_{k_0+N_0}F^h\|_{L^{p^*}})\|G\|_{\dot{B}^{-s}_{p,\infty}}.
	\end{align*}
	 In particular, IF $\sigma_0\leq\sigma_1< \frac{d}{2}-1$ and $p$ satisfy  \eqref{ppp}, then one has
	\begin{equation}\label{equ34}
		\|FG^h\|^\ell_{\dot{B}^{\sigma_1}_{2,\infty}}\lesssim(\|F^\ell   \|_{\dot{B}^{\frac{d}{2}-1}_{2,1}}+\|F^h\|_{\dot{B}^{\frac{d}{p}-1}_{p,1}})\|G^h\|_{\dot{B}^{\frac{d}{p}-1}_{p,1}}, \quad \text{ for }\  2\leq p\leq d,
	\end{equation}
and
	\begin{equation}\label{equ56}
		\|FG^h\|^\ell_{\dot{B}^{-\sigma_1}_{2,\infty}}\lesssim(\|F^\ell   \|_{\dot{B}^{\frac{d}{2}-1}_{2,1}}+\|F^h\|_{\dot{B}^{\frac{d}{p}}_{p,1}})\|G^h\|_{\dot{B}^{\frac{d}{p}-1}_{p,1}}, \quad \text{ for }\  p\geq d.
	\end{equation}
\end{lemma}

\begin{lemma}\label{app1}
Let \( 1 \leq p, r \leq \infty \). Then
\begin{align}
&\|FG\|_{\dot{B}_{p,r}^s} \lesssim \|F\|_{L^\infty} \|G\|_{\dot{B}_{p,r}^s} + \|G\|_{L^\infty} \|F\|_{\dot{B}_{p,r}^s}, \quad \text{if } s > 0,\label{ap1} \\
&\|FG\|_{\dot{B}_{p,1}^{s_1+s_2-\frac{d}{p}}} \lesssim \|F\|_{\dot{B}_{p,1}^{s_1}} \|G\|_{\dot{B}_{p,1}^{s_2}}, \quad \text{if } s_1, s_2 \leq \frac{d}{p} \text{ and } s_1 + s_2 > d \max(0, \frac{2}{p} - 1), \label{ap2}\\
&\|FG\|_{\dot{B}_{p,\infty}^{s_1+s_2-\frac{d}{p}}} \lesssim \|F\|_{\dot{B}_{p,1}^{s_1}} \|G\|_{\dot{B}_{p,\infty}^{s_2}}, \quad \text{if } s_1 \leq \frac{d}{p}, s_2 < \frac{d}{p} \text{ and } s_1 + s_2 \geq d \max(0, \frac{2}{p} - 1).\label{ap3}
\end{align}
\end{lemma}
\begin{lemma}[\cite{Xu-2019}]\label{app2}
Let the real numbers \( s_1, s_2, p_1 \) and \( p_2 \) be such that
\[
s_1 + s_2 > 0, \quad s_1 \leq \frac{d}{p_1}, \quad s_2 \leq \frac{d}{p_2}, \quad s_1 \geq s_2, \quad \frac{1}{p_1} + \frac{1}{p_2} \leq 1.
\]
Then it holds that
\[
\|FG\|_{\dot{B}_{q,1}^{s_2}} \lesssim \|F\|_{\dot{B}_{p_1,1}^{s_1}} \|G\|_{\dot{B}_{p_2,1}^{s_2}}, \quad \text{with} \quad \frac{1}{q} = \frac{1}{p_1} + \frac{1}{p_2} - \frac{s_1}{d}.
\]
Additionally, for \( s > 0 \) and \( 1 \leq p_1, p_2, q \leq \infty \) satisfying
\[
\frac{d}{p_1} + \frac{d}{p_2} - d \leq s \leq \min\left(\frac{d}{p_1}, \frac{d}{p_2}\right) \quad \text{and} \quad \frac{1}{q} = \frac{1}{p_1} + \frac{1}{p_2} - \frac{s}{d},
\]
one has
\[
\|FG\|_{\dot{B}_{q,\infty}^{-s}} \lesssim \|F\|_{\dot{B}_{p_1,1}^{s}} \|G\|_{\dot{B}_{p_2,\infty}^{-s}}.
\]
\end{lemma}
\begin{corollary}\label{app3}
Let \(\sigma_0 \leq \sigma < \frac{d}{2} - 1\) and \(p\) satisfy \eqref{ppp}. It holds that
\begin{align}
\|FG\|_{\dot{B}_{2,\infty}^\sigma} &\lesssim \|F\|_{\dot{B}_{p,1}^{\frac{d}{p}}} \|G\|_{\dot{B}_{2,1}^\sigma}, \label{A.2} \\
\|FG\|_{\dot{B}_{2,\infty}^{\sigma + \frac{d}{p} - \frac{d}{2}}} &\lesssim \|F\|_{\dot{B}_{p,1}^{\sigma + \frac{d}{p} - \frac{d}{2}}} \|G\|_{\dot{B}_{2,1}^{\frac{d}{p}}}. \label{A.3}
\end{align}
In addition, we have
\begin{equation}
\|FG\|_{\dot{B}_{2,\infty}^{\sigma_0}}^\ell \lesssim \|F\|_{\dot{B}_{p,1}^{\frac{d}{p}-1}}^\ell \|G\|_{\dot{B}_{p,1}^{1-\frac{d}{p}}}^\ell \label{A.4}
\end{equation}
for \(2 \leq p \leq d\).
\end{corollary}
\begin{corollary}\label{app5}
Let \(\sigma_0 \leq \sigma < \frac{d}{2} - 1\) and \(p\) satisfy \eqref{ppp}. It holds that
\begin{align}\label{A6}
\|FG\|_{\dot{B}_{2,\infty}^{\sigma+\frac{d}{p}-\frac{d}{2}}} \lesssim \|F\|_{\dot{B}_{p,1}^{\frac{d}{p}-1}} \|G\|_{\dot{B}_{2,\infty}^{\sigma+\frac{d}{p}-\frac{d}{2}+1}}. 
\end{align}
\end{corollary}

\section*{Acknowledgements.}
Z. Chen was supported by the National Natural Science Foundation of China (No. 12501294) Start-up funds for doctoral research of Anhui Normal University (No. 762350).
M. Fei is supported partly by the National Science Foundation of China (No.12271004, No.12471222) and the Natural Science Foundation of Anhui Province of China (No.2308085J10). L. Liu was supported by Natural Science Foundation of China (No.12571243) and Anhui Provincial Natural Science Foundation (No.2508085Y004). J. Wu was supported by the National Science Foundation of the United States (DMS 2104682 and DMS 2309748).

\section*{Declaration }
\textbf{Conflict of interest:} The authors declared that they have no conflicts of interest to this work.


\end{document}